\documentclass[a4paper]{article}

\usepackage[all]{xy}\usepackage[latin1]{inputenc}       
\usepackage[dvips]{graphics,graphicx}
\usepackage{amsfonts,amssymb,amsmath,xcolor,mathrsfs, amstext}
\usepackage{amsbsy, amsopn, amscd, amsxtra, amsthm, authblk, enumerate,dsfont}
\usepackage{upref}
\usepackage{geometry}
\usepackage[displaymath]{lineno}
\usepackage{float}
\usepackage{bbm}
\usepackage{multicol}
\usepackage{comment}
\usepackage{longtable, booktabs}

\allowdisplaybreaks

\usepackage[colorlinks,
            linkcolor=blue,
            anchorcolor=green,
            citecolor=blue
            ]{hyperref}

\numberwithin{equation}{section}

\newtheorem{theorem}{Theorem}[section]
\newtheorem{lemma}[theorem]{Lemma}
\newtheorem{remark}[theorem]{Remark}
\newtheorem{conjecture}[theorem]{Conjecture}
\newtheorem{proposition}[theorem]{Proposition}
\newtheorem*{proposition*}{Proposition}

\newtheorem{definition}[theorem]{Definition}
\newtheorem*{definitions*}{Definitions}

\newtheorem*{example*}{\bf Example}

\newtheorem{prob}[theorem]{\bf Problem}

\numberwithin{equation}{section}

\newcommand\dd{\mathop{}\!\mathrm{d}} 

\newcommand{\eqb}{\begin{equation}}
\newcommand{\eqe}{\end{equation}}

\title{Schramm-Loewner evolution contains a topological Sierpi\'nski carpet when $\kappa$ is close to 8}
\author{Haoyu Liu\thanks{Peking University} \qquad \qquad Zijie Zhuang\thanks{University of Pennsylvania}}

\begin{document}

\maketitle

\begin{abstract}
    We consider the Schramm-Loewner evolution (SLE$_\kappa$) for $\kappa \in (4,8)$, which is the regime where the curve is self-intersecting but not space-filling. We show that there exists $\delta_0>0$ such that for $\kappa \in (8 - \delta_0,8)$, the range of an SLE$_\kappa$ curve almost surely contains a topological Sierpi\'nski carpet. Combined with a result of Ntalampekos (2021), this implies that in this parameter range, SLE$_\kappa$ is almost surely conformally non-removable, and the conformal welding problem for SLE$_\kappa$ does not have a unique solution. Our result also implies that for $\kappa \in (8 - \delta_0,8)$, the adjacency graph of the complementary connected components of the SLE$_\kappa$ curve is disconnected.
\end{abstract}

\setcounter{tocdepth}{1}
\tableofcontents

\section{Introduction}\label{sec:intro}

The Schramm-Loewner evolution (SLE$_\kappa$) is a one-parameter family of random fractal curves connecting two boundary points in a simply connected domain. It was introduced by Schramm~\cite{Schramm-SLE} to describe the conjectural scaling limits of interfaces in two-dimensional statistical mechanics models. Since then, many such models have been proved to converge to SLE$_\kappa$ in the scaling limit, see, e.g.~\cite{Smirnov-per, LSW-looperased, Smirnov-ising}. For $\kappa>0$, SLE$_\kappa$ in the upper half-plane $\mathbb{H}$ from $0$ to $\infty$ is defined via a family of conformal maps $(g_t)_{t \geq 0}$ solving the Loewner equation:
\begin{equation}\label{eq:loewner}
\partial_t g_t(z) = \frac{2}{g_t(z) - W_t}, \quad g_0(z) = z,
\end{equation}
where $W_t = \sqrt{\kappa} B_t$ and $B_t$ is a standard Brownian motion. For $z \in \overline{\mathbb{H}}$, let $\tau(z) = \inf \{t\geq 0: \lim_{s \uparrow t} |g_s(z) - W_s| = 0\}$ and let $K_t = \{z \in \overline{\mathbb{H}}: \tau(z) \leq t \}$ for $t \geq 0$, which is an increasing family of compact sets. It was shown in~\cite{Rohde-Schramm-basic, LSW-looperased} that $(K_t)_{t \geq 0}$ is almost surely generated by a unique continuous curve $\eta:[0,\infty) \to \overline{\mathbb{H}}$ such that for each $t \geq 0$, the set $K_t$ is the union of $\eta([0, t])$ and the bounded connected components of $\overline{\mathbb{H}} \setminus \eta([0, t])$. The parameter $\kappa$ determines the roughness of SLE$_\kappa$: the random curve is simple for $\kappa \leq 4$, self-intersecting but not space-filling for $\kappa \in (4,8)$, and space-filling for $\kappa \geq 8$~\cite{Rohde-Schramm-basic}. Moreover, the Hausdorff dimension of the range of an SLE$_\kappa$ curve is $\min\{ 1 + \kappa/8, 2\}$~\cite{Rohde-Schramm-basic, Beffara-dimension}.

Since the introduction of SLE by Schramm, many other representations of SLE$_\kappa$ have appeared. For instance, in~\cite{SS-contour, Dubedat-SLE-GFF, IG1, IG4}, SLE$_\kappa$ curves are realized as flow lines of the Gaussian free field (GFF). Also SLE$_\kappa$ curves are characterized by the conformal restriction property and admit rich algebraic structures~\cite{LSW-restriction, Werner-sle-loop, baverez2024cft}. In addition, SLE$_\kappa$ curves play a central role in the study of Liouville quantum gravity (LQG), which describes the scaling limit of random planar maps~\cite{Sheffield-zipper, DMS21-LQG-MRT}. In particular, SLE$_\kappa$ arises as the welding interface between two conformally welded LQG disks, as we now recall. Let $\widehat{\mathbb{C}}$ be the Riemann sphere. Suppose $\mathbb{D}_1$ and $\mathbb{D}_2$ are two copies of the unit disk, and let $\phi: \partial \mathbb{D}_1 \to \partial \mathbb{D}_2$ be a homeomorphism. A \textit{conformal welding} with \textit{welding homeomorphism} $\phi$ corresponds to a Jordan loop $\eta$ on $\widehat{\mathbb{C}}$ and a pair of conformal maps $\psi_i$ mapping $\mathbb{D}_i$ to the two connected components of $\widehat{\mathbb{C}} \setminus \eta$ such that $\phi = \psi_2^{-1} \circ \psi_1$; see Figure~\ref{fig:weld} (left). Given a homeomorphism $\phi$, it is not obvious in general whether a conformal welding exists or whether it is unique up to a M\"obius transformation. Sheffield~\cite{Sheffield-zipper} showed that if one endows $\mathbb{D}_1$ and $\mathbb{D}_2$ with the random geometry of LQG and defines $\phi$ by identifying the LQG quantum boundary length measures on $\partial \mathbb D_1$ and $\partial \mathbb D_2$, then a conformal welding exists and the interface is an SLE$_\kappa$-type curve for $\kappa \in (0,4)$.\footnote{Sheffield's original conformal welding is for chordal SLE$_\kappa$ where one glues $\mathbb{D}_1$ and $\mathbb{D}_2$ along two boundary arcs to obtain a chordal SLE$_\kappa$. The version we presented here is for SLE$_\kappa$ loop which was proved in~\cite{AHS-loop} based on Sheffield's result~\cite{Sheffield-zipper}. The case for $\kappa \in (4,8)$ was proved in~\cite{ACSW-loop}.} Sheffield's result was extended to the case $\kappa = 4$~\cite{HP-critical-welding}; see also~\cite{Baverez-welding} for a new proof of Sheffield's welding result for $\kappa \in (0,4]$. For $\kappa > 4$, several conformal welding results were established in~\cite{DMS21-LQG-MRT}. In particular, SLE$_\kappa$-type curves for $\kappa \in (4,8)$ arise as the welding interface when one glues together two independent stable looptrees, where each loop is replaced with a conditionally independent quantum disk; see Figure~\ref{fig:weld} (right).

\begin{figure}[htbp]
    \centering
    \begin{minipage}{0.47\textwidth}
    \includegraphics[width=\textwidth]{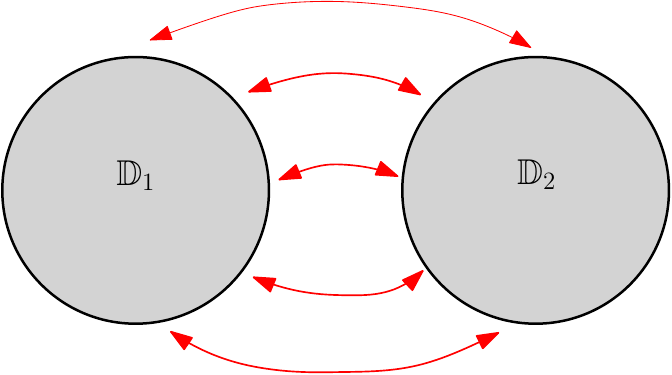}
    \end{minipage}
    \hfill
    \begin{minipage}{0.47\textwidth}
    \includegraphics[width=\textwidth]{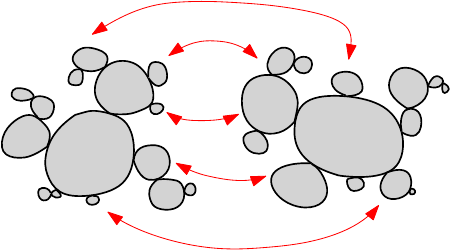}
    \end{minipage}
    \caption{SLE$_\kappa$ curves arise as conformal welding interfaces~\cite{Sheffield-zipper, DMS21-LQG-MRT}. \textbf{Left:} The case $\kappa \in (0,4]$. \textbf{Right:} The case $\kappa \in (4,8)$.}
    \label{fig:weld}
\end{figure}

The conformal removability of the welding interface is a sufficient condition for the uniqueness of conformal welding. Recall that a compact set $K \subset \widehat{\mathbb{C}}$ is said to be \emph{conformally removable} if every homeomorphism $\phi: \widehat{\mathbb{C}} \to \widehat{\mathbb{C}}$ that is conformal on $\widehat{\mathbb{C}} \setminus K$ is conformal on all of $\widehat{\mathbb{C}}$, and thus M\"obius. In light of this, it is natural to ask whether SLE$_\kappa$ curves are conformally removable for different values of $\kappa$. 
In the case $\kappa \in (0,4)$, it was proved in~\cite{Rohde-Schramm-basic} that SLE$_\kappa$ is the boundary of a H\"older domain and such boundaries are conformally removable by a criterion of Jones and Smirnov~\cite{JS-removable}. For $\kappa = 4$, SLE$_4$ does not form the boundary of a H\"older domain and moreover it was shown in~\cite{KMS-SLE4regularity} that the condition in~\cite{JS-removable} for conformal removability does not hold for SLE$_4$. Recently,~\cite{KMS-removability-SLE4} proved the conformal removability of SLE$_4$ using a new criterion. For $\kappa \geq 8$, SLE$_\kappa$ curves have positive Lebesgue measure and are thus conformally non-removable. In the regime $\kappa \in (4,8)$, which is the remaining case, the complement of an SLE$_\kappa$ curve consists of a countable collection of simply connected components, making it more difficult to check the conformal removability~\cite[Question 8]{Sheffield-zipper}. However, in a recent work~\cite{KMS-removability-nonsimple}, the authors proved the conformal removability of SLE$_\kappa$ for those values of $\kappa$ for which the adjacency graph of the complementary connected components is connected (see Theorem~\ref{thm:disconnect} for details). This condition was shown in~\cite{GP20-connectivity} to hold for $\kappa \in (4,\kappa_0)$, where $\kappa_0 \approx 5.6158$.

In this paper, we will show that SLE$_\kappa$ is conformally non-removable when $\kappa$ is close to 8. We will use a criterion for conformal non-removability established in~\cite{Ntalampekos-carpet}. Recall that a topological Sierpi\'nski carpet $S\subset \widehat{\mathbb{C}}$ is a set that is homeomorphic to the standard Sierpi\'nski carpet. By a result of Whyburn~\cite{Whyburn-carpet}, a set $S \subset \widehat{\mathbb{C}}$ is a topological Sierpi\'nski carpet if and only if $S$ has empty interior and can be written as $S = \widehat{\mathbb{C}} \setminus \cup_{i=1}^\infty Q_i$, where $\{Q_i\}_{i \in \mathbb{N}}$ is a family of Jordan domains satisfying:
\begin{enumerate}[(a)]
    \item $\overline{Q}_i \cap \overline{Q}_j = \emptyset$ for all $i \neq j$.
    \item The spherical diameter of $Q_i$ converges to 0 as $i \to \infty$.
\end{enumerate}
The following result is taken from~\cite[Theorem 1.1]{Ntalampekos-carpet}.

\begin{theorem}[\cite{Ntalampekos-carpet}]\label{thm:carpet-non-remove}
    All topological Sierpi\'nski carpets are conformally non-removable.
\end{theorem}

In fact, for any two topological Sierpi\'nski carpets $S_1$ and $S_2$, Ntalampekos~\cite{Ntalampekos-carpet} constructed in a topological way a homeomorphism $\phi: \widehat{\mathbb{C}} \to \widehat{\mathbb{C}}$ such that $\phi(S_1) = S_2$ and $\phi$ is conformal on $\widehat{\mathbb{C}} \setminus S_1$. The construction relies on inductively partitioning Sierpi\'nski carpets into smaller Sierpi\'nski carpets along with finite collections of peripheral disks. At each stage, we assign conformal maps to matching peripheral disks and extend them to a global homeomorphism. These homeomorphisms converge to a limiting homeomorphism that maps each peripheral disk of $S_1$ conformally onto one of $S_2$, thereby satisfying the desired properties. This implies Theorem~\ref{thm:carpet-non-remove} since if $S_1$ has zero (resp.\ positive) Lebesgue measure, we may choose $S_2$ to have positive (resp.\ zero) Lebesgue measure so that $\phi$ cannot be M\"obius.

For $\kappa \in (4,8)$, let $\eta$ be an SLE$_\kappa$ curve from $0$ to $\infty$ in $\mathbb{H}$. Our main result is the following.

\begin{theorem}\label{thm:carpet}
     There exists $\delta_0>0$ such that for $\kappa \in (8 - \delta_0,8)$, the range of the curve $\eta$ almost surely contains a topological Sierpi\'nski carpet.
\end{theorem}

The key intuition in proving Theorem~\ref{thm:carpet} comes from Mandelbrot's fractal percolation model~\cite{mandelbrot-fractal}. Mandelbrot introduced this model in 1974 to generate random self-similar fractal sets, which were shown to undergo phase transitions~\cite{CCD-fractal, CCD-3d-fractal}, as we will recall in a moment. Later on, fractal percolation was used to study various random fractal objects---for instance, Brownian motion~\cite{Peres-intersection}, Brownian loop soup and CLE~\cite{SW-CLE}, and log-correlated Gaussian fields~\cite{DGZ-thick-point}. Our work is the first to apply fractal percolation to non-simple SLE, which is itself a random fractal object. We now describe the model of fractal percolation in detail. Let $p \in (0,1)$ denote the retaining probability. We begin with the unit square $[0,1]^2$.
In each subsequent step, every retained box is divided into $3 \times 3$ disjoint sub-boxes of equal side length, and each sub-box is independently retained with probability $p$. When $p > 1/9$, the final retained set---defined as the intersection of the retained sets at all steps---is non-empty with positive probability. In~\cite{CCD-fractal}, a phase transition for the final retained set was established: there exists $p_c \in [\frac{1}{\sqrt{3}},1)$ such that when $p < p_c$, the set is almost surely totally disconnected, whereas when $p \in [p_c, 1)$, it contains a horizontal crossing of $[0,1]^2$ with positive probability. Using the idea from~\cite{CCD-fractal}, one can show that when $p$ is sufficiently close to 1, the final retained set contains a topological Sierpi\'nski carpet with positive probability; see Section~\ref{sec:fractal}. We will define a similar fractal percolation process within the framework of the imaginary geometry construction of SLE$_\kappa$~\cite{IG1,IG4}, which realizes SLE curves as flow lines of the Gaussian free field, so that the final retained set is contained in the trace of SLE$_\kappa$; see Section~\ref{sec:main-proof}. Applying the fractal percolation argument, we show that this final retained set in imaginary geometry contains a topological Sierpi\'nski carpet when $\kappa$ is sufficiently close to 8, analogous to the case of $p$ close to 1 in Mandelbrot's fractal percolation, thereby completing the proof of Theorem~\ref{thm:carpet}. However, since SLE is not independent across disjoint regions but exhibits correlations, we make suitable adjustments to the fractal percolation argument.

By definition, if $E \subseteq E'$ and $E$ is conformally non-removable, then $E'$ is also conformally non-removable. Combining this property with Theorems~\ref{thm:carpet} and~\ref{thm:carpet-non-remove}, we obtain the following result.

\begin{theorem}\label{thm:non-removable}
    For $\kappa \in (8 - \delta_0,8)$, the range of an SLE$_\kappa$ curve is almost surely conformally non-removable.
\end{theorem}

Theorem~\ref{thm:carpet}, together with a refined result in~\cite{Ntalampekos-carpet}, implies the non-uniqueness of conformal welding for SLE$_\kappa$ in the same parameter range.

\begin{theorem}\label{thm:welding}
    For $\kappa \in (8-\delta_0,8)$, the solution to the conformal welding problem for SLE$_\kappa$ is almost surely not unique.
\end{theorem}

Let us mention that the conformal non-removability does not imply the non-uniqueness of conformal welding, so Theorem~\ref{thm:welding} is really not a direct consequence of Theorem~\ref{thm:non-removable}. We will discuss conformal welding in more detail in Section~\ref{sec:welding}.
However, we remark that an abstract measurability argument was used in~\cite{DMS21-LQG-MRT} to show that there is almost surely a unique way to obtain an SLE$_\kappa$-decorated LQG surface from conformal welding for all $\kappa \in (4,8)$. A stronger version of the uniqueness result was proved in~\cite{MMQ-welding} where it was shown that the conformal welding problem for SLE$_\kappa$ has a unique solution if one requires the welding interface to be an SLE$_\kappa$ curve.

\begin{remark}(Different domains or variants of SLE)
    Since the properties in Theorem~\ref{thm:carpet} are topological, an analogous result holds for SLE$_\kappa$ in any simply connected domain when $\kappa \in (8 - \delta_0,8)$. (Note that the topological Sierpi\'nski carpet found in the curve $\eta$ from Theorem~\ref{thm:carpet} can be required to be bounded and have positive distance from $\partial \mathbb{H}$, so Theorem~\ref{thm:carpet} indeed extends to simply connected domains whose boundaries are not simple curves.) The SLE$_\kappa$ curve considered above is usually referred to as chordal SLE$_\kappa$ curve, which connects two boundary points of a simply connected domain. There are also other variants of SLE$_\kappa$ such as radial SLE$_\kappa$, SLE$_\kappa$ loops, or multiple SLE$_\kappa$ curves. For each of these variants, we can find a part of it that has the same law as a chordal SLE$_\kappa$. Therefore, Theorem~\ref{thm:carpet} also extends to these variants. As a consequence, by Theorem~\ref{thm:carpet-non-remove} and the arguments in Section~\ref{sec:welding}, for $\kappa \in (8 - \delta_0, 8)$, the range of any variant of SLE$_\kappa$ is almost surely conformally non-removable, and the solution to the conformal welding for these SLE$_\kappa$ variants is almost surely not unique.
\end{remark}

Theorem~\ref{thm:carpet} also partially answers the bubble connectivity problem posed in Question 11.2 of~\cite{DMS21-LQG-MRT}. A \textit{bubble} of $\eta$ is a connected component of $\mathbb{H} \setminus \eta$. Consider the adjacency graph of bubbles where two bubbles are adjacent if their boundaries intersect. It was shown in~\cite{GP20-connectivity} for $\kappa \in (4, \kappa_0)$, where $\kappa_0 \approx 5.6158$, the adjacency graph is almost surely connected, i.e., for any two bubbles $U$ and $V$, there exists a finite sequence of adjacent bubbles $U = U_1, U_2, \ldots, U_n = V$. Note that if the range of $\eta$ contains a topological Sierpi\'nski carpet, then all connected bubbles in the adjacency graph must lie within the same Jordan domain in Whyburn's criterion. In this case, the adjacancy graph must be disconnected. Therefore, we obtain the following theorem, and Theorem~\ref{thm:carpet} cannot hold for $\kappa \in (4,\kappa_0)$.

\begin{theorem}\label{thm:disconnect}
    For $\kappa \in (8 - \delta_0,8)$, the adjacency graph of the bubbles of $\eta$ is almost surely disconnected.
\end{theorem}

In general, there is no simple characterization of conformal removability. In particular, it was proved in~\cite{bishop2020conformal} that the collection of conformally removable compact subsets of $[0,1]^2$ is not Borel with respect to the Hausdorff metric. Also, it was only known recently that the Sierpi\'nski gasket is conformally non-removable~\cite{Ntalampekos-gasket}. However, we have the following conjecture for the topological properties of SLE$_\kappa$ when $\kappa \in (4,8)$.\footnote{We learned from Ewain Gwynne and Konstantinos Kavvadias that they have an argument showing that the set of $\kappa$ in $(4,8)$ for which the adjacency graph of the complementary connected components  of SLE$_\kappa$ is connected is open. \label{footnote:gk-open}}

\begin{conjecture}\label{conj:removability}
    For $\kappa \in (4,8)$, let $\eta$ be an SLE$_\kappa$ curve from $0$ to $\infty$ in $\mathbb{H}$. There exists $\kappa_c \in (4,8)$ such that
    \begin{enumerate}[(1)]
        \item \label{conj:case1} For $\kappa \in (4, \kappa_c)$, the adjacency graph for the connected components of $\mathbb{H} \setminus \eta$ is almost surely connected.

        \item \label{conj:case2} For $\kappa \in (\kappa_c,8)$, the adjacency graph for the connected components of $\mathbb{H} \setminus \eta$ is almost surely disconnected, and moreover, the range of $\eta$ contains a topological Sierpi\'nski carpet. 
            
    \end{enumerate}
\end{conjecture}

If Conjecture~\ref{conj:removability} were proved, one could essentially determine the conformal removability of SLE$_\kappa$ from its topological properties, except at the critical value. In particular, for $\kappa \in (4,\kappa_c)$, by~\cite{KMS-removability-nonsimple}, we know that the range of $\eta$ is almost surely conformally removable. For $\kappa \in (\kappa_c,8)$, by~\cite{Ntalampekos-carpet}, we know that the range of $\eta$ is almost surely conformally non-removable. We also conjecture a phase transition for the conformal removability and bubble connectivity of exotic SLE$_\kappa(\rho)$ with $\kappa \in (0,4)$ and $\rho>-2-\frac{\kappa}{2}$, which includes Conjecture~\ref{conj:removability} as a special case; see Section~\ref{sec:open}.

For fractal percolation, we can in fact show that throughout the connected phase $p \in (p_c, 1)$, the final retained set contains a topological Sierpi\'nski carpet with positive probability; see Proposition~\ref{prop:super-sierpinski}. However, the behavior at the critical point $p = p_c$ remains unclear. Similar phase transitions also occur in Brownian loop soups~\cite{SW-CLE}, where one samples a Poissonian collection of Brownian loops in the unit disk $\mathbb D$, with the intensity $c >0$ serving as the percolation parameter. Two Brownian loops are considered connected if they can be linked by a finite chain of intersecting loops, thereby belonging to the same loop soup cluster. It was shown in~\cite{SW-CLE} that when the intensity $c>1$, there is a single Brownian loop soup cluster whose closure is $\overline{\mathbb{D}}$, and, if we remove from $\mathbb{D}$ the regions enclosed by Brownian loops, the retained set is totally disconnected (this is analogous to the totally disconnected phase in fractal percolation if we interpret the regions enclosed by Brownian loops as removed boxes). When the intensity $c \in (0,1]$, there are infinitely many Brownian loop soup clusters. Moreover, if we remove from $\overline{\mathbb{D}}$ the regions enclosed by Brownian loops, the retained set contains a topological Sierpi\'nski carpet (analogous to the connected phase). In fact, the set of points in $\overline{\mathbb D}$ not surrounded by any Brownian loop soup cluster has the law of a CLE$_\kappa$ carpet\footnote{For $\kappa \in (8/3,4]$, CLE$_\kappa$ is a random collection of simple loops in $\mathbb{D}$, each of which locally looks like an SLE$_\kappa$ curve. The CLE$_\kappa$ carpet is defined as the set of points in $\overline{\mathbb{D}}$ that are not surrounded by any CLE$_\kappa$ loop.} for an explicit $\kappa = \kappa(c) \in (8/3,4]$, which is almost surely a topological Sierpi\'nski carpet. A key additional property of the Brownian loop soup is its conformal invariance, which is absent in fractal percolation. In light of the Brownian loop soup case, we conjecture that at $\kappa = \kappa_c$, scenario~\eqref{conj:case2} in Conjecture~\ref{conj:removability} holds. We refer to Section~\ref{sec:open} for further discussion and open problems related to other random fractals.

\medskip
\noindent\textbf{Organization of the paper.} In Section~\ref{sec:fractal}, we review Mandelbrot's fractal percolation. We will prove our main result Theorem~\ref{thm:carpet} in Section~\ref{sec:main-proof} modulo some technical lemmas that will be proved in subsequent sections. In Section~\ref{sec:SLE}, we provide the necessary background on SLE, GFF, and imaginary geometry (IG) and prove Lemmas~\ref{lem:flow-line-not-touch} and~\ref{lem:kappa-to-8}. In Section~\ref{sec:GFF}, we prove Lemma~\ref{lem:induction-bound} which is entirely about the Gaussian free field. Section~\ref{sec:carpet} is dedicated to proving Proposition~\ref{prop:find-carpet}. In Section~\ref{sec:welding}, we discuss the conformal welding problem for SLE and prove Theorem~\ref{thm:welding}. Finally, in Section~\ref{sec:open}, we discuss some open problems.

\medskip
\noindent\textbf{Basic Notations.} In this paper, we use $|\cdot|_\infty$ to denote the $L^\infty$-norm. Let $B(x,r) = \{ z \in \mathbb{C}: |z-x|_\infty<r \}$ denote the open Euclidean box.
For two sets $A,B \subset \mathbb R^2$, let $d(A,B)=\inf\{|x-y|_\infty:x \in A, y \in B\}$ and $\mathrm{diam}(A)=\sup\{|x-y|_\infty:x,y \in A\}$. Define $\mathbb{B} = (-1,1)^2$. For two quantities $a,b$ depending on a parameter $r$, we write $a=o_r(b)$ if $a/b \to 0$ as $r \to \infty$ or $r \to 0$, depending on the context.

\medskip
\noindent\textbf{Acknowledgements.} We are grateful to Xin Sun for encouragement and enlightening discussions at various stages of this project. We also thank Morris Ang, Ewain Gwynne, Konstantinos Kavvadias and Pu Yu for helpful discussions and comments on the manuscript. H.L.\ is partially supported by National Key R\&D Program of China (No.\ 2023YFA1010700) and the Fundamental Research Funds for the Central Universities, Peking University. Z.Z.\ was partially supported by NSF grant DMS-1953848.

\section{Sierpi\'nski carpet in Mandelbrot's fractal percolation}
\label{sec:fractal}
As mentioned earlier, the proof of Theorem~\ref{thm:carpet} is based on insights from fractal percolation~\cite{mandelbrot-fractal,CCD-fractal}. To better illustrate the proof idea, we will prove a parallel result for this model in this section.

Mandelbrot's fractal percolation model is defined as follows. Fix a positive integer $N \ge 2$, and let $p \in [0,1]$. We partition the unit box $[0,1]^2$ into $N$-adic boxes: for $n \ge 0$, let $\mathfrak B_n=\{[\frac{i-1}{N^n},\frac{i}{N^n}] \times [\frac{j-1}{N^n},\frac{j}{N^n}]:1 \le i,j \le N^n\}$ be the collection of boxes of side-length $N^{-n}$, such that each box in $\mathfrak B_n$ contains $N^2$ sub-boxes in $\mathfrak B_{n+1}$. For $B \in \cup_{n \ge 1} \mathfrak B_n$, let $\sigma_B$ be independent Bernoulli random variables with $\mathbb{P}[\sigma_B=1]=1-\mathbb{P}[\sigma_B=0]=p$. Let $A_0=[0,1]^2$. For $n \ge 1$, we inductively define $A_n=A_{n-1} \cap (\cup_{B \in \mathfrak B_n: \sigma_B=1} B)$. We may view (modulo some boundary issues) that boxes $B \in \mathfrak B_n$ with $\sigma_B=0$ as being removed at step $n$, and $A_n$ is the closed set of points retained after step $n$.

Let $A_\infty=\cap_{n \ge 0} A_n$ be the retained set, which is a random fractal set with Hausdorff dimension $\max\{2+\frac{\log p}{\log N},0\}$. It is easy to see that when $p \le \frac{1}{N^2}$, we have $A_\infty=\emptyset$ almost surely; whereas when $p > \frac{1}{N^2}$, the set $A_\infty$ is non-empty with positive probability. The seminal work~\cite{CCD-fractal} studied the connectivity properties of $A_\infty$ and proved the existence of a non-trivial phase transition: there exists $p_c=p_c(N) \in [\frac{1}{\sqrt{N}},1)$ such that when $\frac{1}{N^2}<p<p_c$, the set $A_\infty$ is a.s. totally disconnected, while when $p_c \le p \le 1$, the set $A_\infty$ contains a horizontal crossing of $[0,1]^2$ with positive probability.

The counterpart of Theorem~\ref{thm:carpet} for fractal percolation is the following.

\begin{figure}
    \centering
    \includegraphics[width=0.4\linewidth]{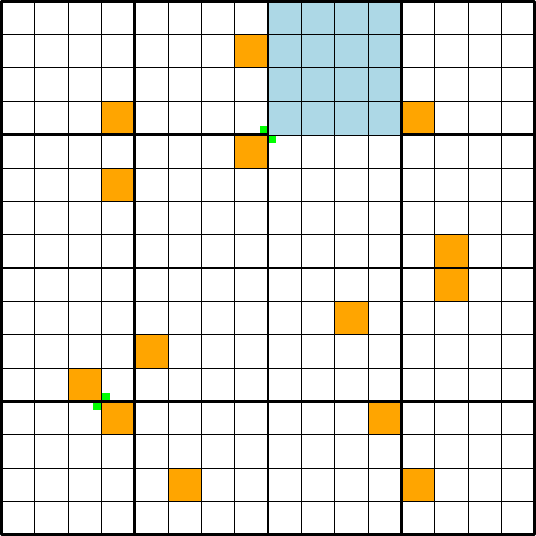}
    \caption{$N=4$ and on the event that $[0,1]^2$ is $\infty$-good. The regions in light-blue (resp.\ in light-blue or orange) are $[0,1]^2 \setminus A_1^\dagger$ (resp.\ $[0,1]^2 \setminus A_2^\dagger$). The set $A_2^*$ is obtained from $A_1^* \cap A_2^\dagger$ by further removing the green boxes.}
    \label{fig:fractal}
\end{figure}

\begin{proposition}\label{prop:carpet-fractal}
    There exists $p_0=p_0(N) \in (0,1)$ such that for $p \in (p_0,1)$, with positive probability, $A_\infty$ contains a topological Sierpi\'nski carpet.
\end{proposition}

A variant of Proposition~\ref{prop:carpet-fractal} (with the event replaced by ``there exists a horizontal crossing of $[0,1]^2$ in $A_\infty$") was proved in~\cite{CCD-fractal} and used to deduce that $p_c<1$. In fact, their proof (with a little more thought) also works for our formulation.

\begin{proof}[Proof of Proposition~\ref{prop:carpet-fractal}]
    For convenience, we will consider $N\ge 6$. Since the case for $N$ can be treated by comparing it with the case $N'= N^2$, the conclusion then holds for all $N \ge 2$.
    
    For each box $B \in \cup_{n \ge 1} \mathfrak B_n$, $B$ is called \emph{$0$-good} if $\sigma_B=1$; and we always say $[0,1]^2$ is $0$-good. $B$ is called \emph{$1$-good} if at least $(N^2-1)$ of its sub-boxes are $0$-good. For $m \ge 2$, $B$ is called \emph{$m$-good} if at least $(N^2-1)$ of its sub-boxes are both $0$-good and $(m-1)$-good. Note that the $m$-good properties of disjoint boxes of the same side-length are independent from each other.
    Also observe that for $m \ge 1$, the $m$-good condition becomes sharper as $m$ increases (that is, for $1 \le m<m'$, if $B$ is $m'$-good, then $B$ is also $m$-good). If $B$ is $m$-good for all $m \ge 0$, we say $B$ is \emph{$\infty$-good}.
    
    We claim that if $[0,1]^2$ is $\infty$-good, then
    \begin{equation}\label{eq:fractal-carpet-1}
        \mbox{the retained set } A_\infty \mbox{ contains a topological Sierpi\'nski carpet.}
    \end{equation}
    Suppose that $[0,1]^2$ is $\infty$-good. Let $A_0^\dagger = [0,1]^2$. For $n \geq 1$, we define $A_n^\dagger \subseteq A_n \cap A_{n-1}^\dagger$ to be the union of boxes in $\mathfrak B_n$ that are $\infty$-good and contained in $A_n \cap A_{n-1}^\dagger$. For each $B \in \mathfrak B_n$ such that $B \subset A_n^\dagger$, we see that at least $(N^2-1)$ of its sub-boxes are retained in $A_{n+1}^\dagger$. Typically, $\cap_{n \geq 0} A_n^\dagger$ is not a topological Sierpi\'nski carpet as the closures of its complementary connected components may intersect at the corner. Therefore, we will modify $A_n^\dagger$ to define $A_n^*$ by removing sub-boxes at the corner so that $\cap_{n \geq 0} A_n^*$, which is contained in $A_\infty$, is a topological Sierpi\'nski carpet.
    
    Let $A_0^*=[0,1]^2$, and for $n \ge 1$, we inductively define $A_n^* \subseteq A_n^\dagger$ as follows. Let $\mathcal C_n$ denote the collection of all complementary connected components of $A_{n-1}^* \cap A_n^\dagger$. Whenever the closures of two components in $\mathcal C_n$ intersect exactly at one point, we further remove from $A_{n-1}^* \cap A_n^\dagger$ the other two boxes in $\mathfrak B_{n+1}$ incident to that point, and denote the closure of the resulting set by $A_n^*$; see Figure~\ref{fig:fractal}. It is clear that the complementary components of $A_n^*$ have pairwise disjoint closures, and $\{A_n^*\}_{n \ge 0}$ is a decreasing sequence of closed sets.  

    We now show that
    \begin{equation}\label{eq:fractal-carpet-2}
        A_\infty^*:=\cap_{n \ge 0} A_n^* \subseteq A_\infty  \mbox{ is a topological Sierpi\'nski carpet,}
    \end{equation}
    which directly implies~\eqref{eq:fractal-carpet-1}.
    Let $\{Q_i\}_{i \ge 1}$ be the complementary connected components of $A_\infty^*$, including the unbounded one. We first observe that for any $k \ge 1$ and a complementary component $\mathcal Q^k$ of $A_k^*$, if we denote by $\mathcal Q^{k+m}$ the complementary component of $A_{k+m}^*$ that contains $\mathcal Q^k$, then $\cup_{m \ge 0} \mathcal Q^{k+m}$ is a complementary component of $A_\infty^*$. Conversely, for each complementary component $Q_i$ of $A_\infty^*$, there exists a positive integer $k$ such that $Q_i$ contains a complementary component of $A_k^*$. The crucial observation for proving~\eqref{eq:fractal-carpet-2} is that the evolutions of two complementary components of $A_k^*$ can never touch each other. In fact, if $\mathcal Q_1^k$ and $\mathcal Q_2^k$ are two complementary components of $A_k^*$, then $d(\mathcal Q_1^k,\mathcal Q_2^k) \ge N^{-k}(1-2N^{-1})$. This is because before possibly removing some boxes in $\mathfrak B_{k+1}$ at the corner, the $|\cdot|_\infty$-distance between two different components is at least $N^{-k}$, and such removal operation reduces the distance by no more than $2N^{-k-1}$.
    For $m \ge 0$, let $\mathcal Q_1^{k+m}$ (resp.\ $\mathcal Q_2^{k+m}$) be the complementary connected components of $A_{k+m}^*$ that contains $\mathcal Q_1^k$ (resp.\ $\mathcal Q_2^k$). Recall that for each $B \in \mathfrak B_k$ such that $B \subset A_k^\dagger$, at most one sub-box of $B$ fails to remain in $A_{k+1}^\dagger$, so we have
    \[ d(\mathcal Q_1^{k+1},\mathcal Q_2^{k+1}) \ge d(\mathcal Q_1^k,\mathcal Q_2^k)-2N^{-k-1}-2N^{-k-2} \ge N^{-k}(1-4N^{-1}-2N^{-2}), \]
    where the additional $2N^{-k-2}$ comes from the potential removal of some boxes in $\mathfrak B_{k+2}$ at the corner in the construction of $A_{k+1}^*$. In general, we can prove inductively that
    \[ d(\mathcal Q_1^{k+m},\mathcal Q_2^{k+m}) \ge N^{-k}(1-4N^{-1}-4N^{-2}-\cdots-4N^{-m}-2N^{-m-1}), \quad \mbox{for all } m \ge 1. \]
    This implies that $d(\mathcal Q_1^{k+m},\mathcal Q_2^{k+m}) \ge N^{-k}/5$ for all $m \ge 0$ (where we used $N \geq 6$), so $\cup_{m \ge 0} \mathcal Q_1^{k+m}$ and $\cup_{m \ge 0} \mathcal Q_2^{k+m}$ have disjoint closures.
    Therefore, for different $Q_i$ and $Q_j$, there exists a positive integer $k$ such that $Q_i$ (resp.\ $Q_j$) contains a \emph{unique} complementary connected component $\mathcal Q_i^k$ (resp.\ $\mathcal Q_j^k$) of $A_k^*$, and hence $Q_i=\cup_{m \ge 0} \mathcal Q_i^{k+m}$ and $Q_j=\cup_{m \ge 0} \mathcal Q_j^{k+m}$ have disjoint closures.
    A similar argument shows that each $Q_i$ is a simply connected domain whose boundary is a simple continuous curve; hence, each $Q_i$ is a Jordan domain. Moreover, since $A_\infty$ has Hausdorff dimension smaller than 2, $A_\infty^*$ has no interior points. By Whyburn's criterion~\cite{Whyburn-carpet}, $A_\infty^*$ is a topological Sierpi\'nski carpet, which proves~\eqref{eq:fractal-carpet-2}. The desired proposition now follows from Proposition~\ref{prop:fractal-rsw} below.
\end{proof}

\begin{proposition}\label{prop:fractal-rsw}
    For each $m \ge 0$, let $\vartheta_m$ be the probability that $[0,1]^2$ is $m$-good. Then there exists $p_0=p_0(N) \in (0,1)$ such that $\inf_{m \ge 0} \vartheta_m \ge 2/3$ for all $p \in (p_0,1)$.
\end{proposition}

\begin{proof}
    By definition, $\vartheta_0=1$, and for $m \ge 1$, $\vartheta_m=\phi(\vartheta_{m-1})$, where
    \begin{align*}
        \phi(x)&=(px)^{N^2}+N^2(px)^{N^2-1}(1-px) \\
        &=N^2 p^{N^2-1} x^{N^2-1}-(N^2-1) p^{N^2} x^{N^2}.
    \end{align*}
    Note that $\phi$ increases on $[0,1]$. For sufficiently small $\nu>0$ and $p>p_0:=(1-\nu/2)^{1/N^2}$, we have
    \begin{align*}
        \phi(1-\nu) &\ge N^2 p^{N^2-1}(1-(N^2-1)\nu)-(N^2-1) p^{N^2}(1-N^2\nu+\tfrac{1}{2}N^2(N^2-1)\nu^2) \\
        &=N^2(p^{N^2-1}-p^{N^2})+p^{N^2}-\tfrac{1}{2} N^2(N^2-1)^2 p^{N^2} \nu^2 \\
        &\ge p^{N^2}-\tfrac{1}{2} N^2(N^2-1)^2 \nu^2>1-\nu.
    \end{align*}
    The last inequality follows by choosing $\nu$ sufficiently small depending on $N$. Then, it follows from simple induction that $\vartheta_m>1-\nu>2/3$ for all $m \ge 1$ and $p \in (p_0,1)$.
\end{proof}

The result of Proposition~\ref{prop:carpet-fractal} can be improved via a qualitative argument. One may hope to extend this argument to SLE$_\kappa$ to prove Conjecture~\ref{conj:removability}, but the main difficulty is the absence of an increasing coupling for SLE$_\kappa$ with $\kappa \in (4,8)$, under which the connectivity of the adjacency graph is decreasing.

\begin{proposition}\label{prop:super-sierpinski}
    For any $p>p_c$, the retained set $A_\infty$ contains a topological Sierpi\'nski carpet with positive probability.
\end{proposition}

\begin{proof}
    Consider the removed boxes in fractal percolation, i.e., the $N$-adic boxes for which $\sigma_B = 0$ and which are not contained in any other box with $\sigma_B = 0$. We say that two removed boxes $A$ and $B$ are \emph{connected} if there exists a finite chain of removed boxes $A = A_1, A_2, \ldots, A_{n-1}, A_n = B$ such that each pair $A_i$ and $A_{i+1}$ shares a non-trivial boundary segment. Indeed, if $A_i$ and $A_{i+1}$ intersect at exactly one point, there almost surely exists another removed box $A'$ containing that point such that $A'$ intersects both $A_i$ and $A_{i+1}$ along a non-trivial boundary segment. We define \emph{removed box clusters} as the unions of removed boxes that are connected in this way. The final retained set $A_\infty$ can be equivalently defined by removing from $[0,1]^2$ the interiors of all removed box clusters.

    For a removed box cluster $\mathcal{O}$, its \emph{filling} $F(\mathcal{O})$ is defined to be the complement of the unbounded connected component of $\mathbb{C} \setminus \overline{\mathcal{O}}$, and its \emph{outer boundary} $\partial^o\mathcal{O}$ is defined to be the boundary of $F(\mathcal{O})$. Consider the union of $\mathbb{C} \setminus [0,1]^2$ and all removed box clusters that intersect the boundary of $[0,1]^2$. Define $E$ as the event that its complement has interior points. It follows from the FKG inequality that $E$ occurs with positive probability when $p \geq p_c$~\cite{CCD-fractal}. On the event $E$, we choose an arbitrary connected component of its interior, denoted by $U$.
    We say that a removed box cluster $\mathcal{O}$ contained in $U$ is an outermost removed box cluster if it is not enclosed by the outer boundary of any other removed box cluster in $U$. Suppose that for some $p \in [p_c, 1)$, the following holds almost surely on the event $E$: $\partial U$ is a simple loop and for any two outermost removed box clusters $\mathcal{O}, \mathcal{O}'$ contained in $U$,
    \begin{equation}\label{eq:outer-non-intersect}
        \partial^o \mathcal{O} \mbox{ and } \partial^o \mathcal{O}' \mbox{ are simple loops and are disjoint from each other and from } \partial U. 
    \end{equation}
    We claim that for $p$ satisfying~\eqref{eq:outer-non-intersect}, on the event $E$, the set $A_\infty$ contains a topological Sierpi\'nski carpet. Indeed, let $\{F(\mathcal{O}_i)\}_{i \geq 1}$ denote the fillings of all outermost removed box clusters in $U$. Then $A_\infty \supset \overline{U} \setminus \cup_{i \geq 1}(F(\mathcal{O}_i))^\circ$, where $(F(\mathcal{O}_i))^\circ$ denotes the interior of $F(\mathcal{O}_i)$. Moreover, as implied by~\eqref{eq:outer-non-intersect}, there are infinitely many such outermost removed box clusters (since $A_\infty$ has no interior points). By~\eqref{eq:outer-non-intersect} and Whyburn's criterion, the set $\overline{U} \setminus \cup_{i \geq 1}(F(\mathcal{O}_i))^\circ$ is a topological Sierpi\'nski carpet.

    We next show that~\eqref{eq:outer-non-intersect} holds for any $p>p_c$ except for a countable set. In particular, combined with the monotonicity of $A_\infty$ in $p$, this implies that for any $p>p_c$, the set $A_\infty$ contains a topological Sierpi\'nski carpet with positive probability. We first show that except for a countable set of $p$, for any $p \in (0,1)$, it is almost surely the case that
    \begin{equation}\label{property:strong-disconnect}
        \mbox{the boundaries of any two different removed box clusters do not intersect.}
    \end{equation}
    Fix two $N$-adic boxes $A$ and $B$, and consider the removed box clusters $\mathcal{O}_A$ and $\mathcal{O}_B$ containing $A$ and $B$, respectively. There is a natural increasing coupling of fractal percolation for different $p \in (0,1)$ in terms of $\sigma_B$. The crucial observation is that there is almost surely at most one value of $p$ for which $\mathcal{O}_A \neq \mathcal{O}_B$ and $\overline{\mathcal{O}}_A \cap \overline{\mathcal{O}}_B \neq \emptyset$. Specifically, consider the largest $\widetilde p \in [0,1]$ such that $\overline{\mathcal{O}}_A$ and $\overline{\mathcal{O}}_B $ intersect at some point $z$; see Figure~\ref{fig:fractal-2}. Then, a.s., $\mathcal{O}_A = \mathcal{O}_B$ for all $p < \widetilde p$, since for typical $z$, there is a removed box containing $z$ as an interior point at parameter $p$ for any $p<\widetilde p$ (if $z$ lies on the boundary or at a corner of a removed box, then we can choose at most four removed boxes whose union contains $z$ as an interior point).
    By a first moment estimate and by enumerating over all pairs $(A,B)$, we obtain~\eqref{property:strong-disconnect}.

    We now show that~\eqref{eq:outer-non-intersect} holds for any $p \in [p_c,1)$ that satisfies~\eqref{property:strong-disconnect}. The fact that $\partial^o \mathcal{O}_1$ and $\partial^o \mathcal{O}_2$ are disjoint from each other and from $\partial U$ follows directly from~\eqref{property:strong-disconnect}. Using arguments similar to~\cite[Lemma 9.7]{SW-CLE}, we can show that $\partial^o \mathcal{O}_1$ is a continuous loop. (An important step in~\cite[Lemma 9.6]{SW-CLE} is the use of the BK inequality to show that there are finitely many removed box clusters crossing a fixed annulus, which also applies in fractal percolation.) Moreover, since the removed box cluster has no cut points, $\partial^o \mathcal{O}_1$ is simple. Similar arguments show that $\partial U$ is a continuous simple loop. This yields~\eqref{eq:outer-non-intersect} from~\eqref{property:strong-disconnect}. Combining this with the monotonicity of $A_\infty$ in $p$ implies that for any $p>p_c$, the set $A_\infty$ contains a topological Sierpi\'nski carpet with positive probability, concluding the proposition. \qedhere

   \begin{figure}
    \centering
    \includegraphics[width=0.4\linewidth]{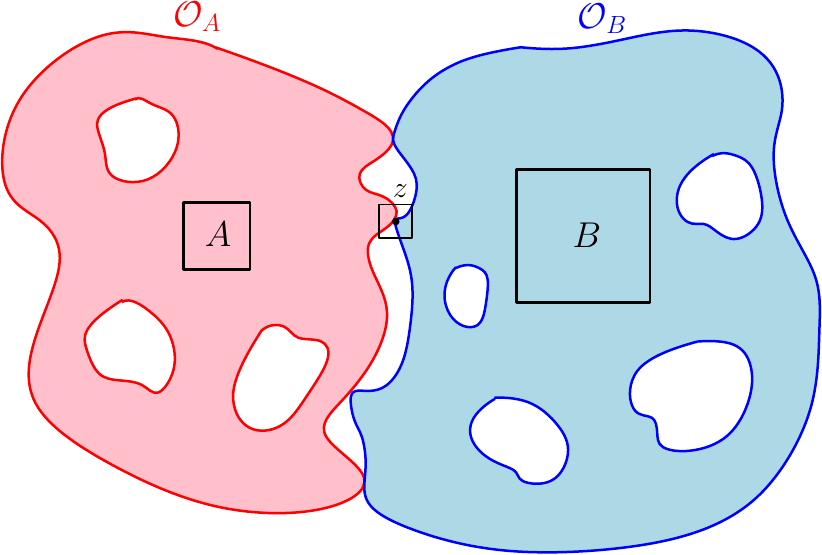}
    \caption{The boundaries of two removed box clusters intersect at some point. Decreasing $p$ by any positive amount will almost surely merge these two clusters into one.}
    \label{fig:fractal-2}
    \end{figure}
    
\end{proof}
    
    Finally, we remark that at present we do not know whether $A_\infty$ contains a topological Sierpi\'nski carpet at the critical point $p = p_c$. We also do not know whether~\eqref{property:strong-disconnect} holds for all $p > p_c$, although we conjecture this to be the case.
    We refer to Section~\ref{sec:open} for further discussion.

\section{A fractal percolation from imaginary geometry}\label{sec:main-proof}

In this section, we describe the proof outline of Theorem~\ref{thm:carpet} using imaginary geometry theory~\cite{IG1, IG4}. For convenience, we will work with chordal SLE$_\kappa$ processes in the box $\mathbb B=(-1,1)^2$.

\subsection{Setting up the fractal percolation}\label{subsec:setup-ig}

Let $\kappa \in (6,8)$, $\underline\kappa = 16/\kappa \in (2,8/3)$, and $\chi = \frac{2}{\sqrt{\underline\kappa}} - \frac{\sqrt{\underline\kappa}}{2} \in (\frac{1}{\sqrt{6}}, \frac{1}{\sqrt{2}})$. We first recall the imaginary geometry (IG) construction of the chordal space-filling SLE$_\kappa$ and chordal SLE$_\kappa$; see~\cite{IG1, IG4} for details. Let $h$ denote the Gaussian free field (GFF) on the box $\mathbb{B}$ with the imaginary geometry boundary conditions on $(\mathbb{B}, -i, i)$ so that the counterflow line of $h$ from $-i$ to $i$ corresponds to a chordal SLE$_\kappa$. For $(\mathbb{H}, 0, \infty)$, the corresponding GFF is a Dirichlet GFF $\widetilde h$ with boundary conditions $\lambda' = \frac{\pi}{\sqrt{\kappa}}$ on $\mathbb{R}_-$ and $-\lambda'$ on $\mathbb{R}_+$. If $\psi$ is a conformal map from $(\mathbb H,0,\infty)$ to $(\mathbb{B}, -i, i)$, then $h$ is given by $h=\widetilde h \circ \psi^{-1}-\chi \arg (\psi^{-1})'$.

For rational points $v \in \overline{\mathbb{B}}$, we can simultaneously define the $\pi/2$ and $-\pi/2$ angle flow lines of $h$ starting from $v$, denoted by $\eta^L(v)$ and $\eta^R(v)$. We orient these flow lines from $-i$ to $i$ after they hit and merge with the boundary. Since flow lines with the same angle will merge together after hitting, these flow lines form a space-filling tree and a space-filling dual tree. There a.s. exists a unique continuous curve $\widetilde \eta$ that traces the entire tree, which is the chordal space-filling SLE$_\kappa$ from $-i$ to $i$. Parameterizing $\widetilde \eta$ using capacity seen from $i$ then yields an ordinary chordal SLE$_\kappa$ from $-i$ to $i$, which is also the counterflow line of $h$ from $-i$ to $i$. Let $\eta$ denote this chordal SLE$_\kappa$. In other words, $\eta$ can be obtained by restricting $\widetilde \eta$ to the set of times when its tip lies on the frontier of the region $\widetilde \eta$ has traversed so far.

We record some qualitative properties that follow from~\cite{IG4} and will be used in this paper. We will elaborate Properties~\eqref{property:simple},~\eqref{property:flow-line-all}, and~\eqref{property:measurable} in more detail in Section~\ref{subsec:ig-prelim}.

\begin{enumerate}[(a)]
    
    \item \label{property:simple} For $\kappa \in (6,8)$, the $\pi/2$ and $-\pi/2$ angle flow lines of $h$ starting from rational points in $\overline{\mathbb{B}}$ are almost surely simple curves and do not cross each other.

    \item \label{property:flow-line-all} For all $v \in \overline{\mathbb{B}}$, let $\sigma_v$ denote the first time that $\widetilde \eta$ hits $v$, and define $\eta^L(v)$ and $\eta^R(v)$ as the left and right boundaries of $\widetilde \eta([\sigma_v,\infty))$. For rational points $v$ in $\overline{\mathbb B}$, these boundaries coincide with the $\pi/2$ and $-\pi/2$ angle flow lines. Moreover, for any $v \in \overline{\mathbb{B}}$, there exists a sequence of rational points $\{v_n\}_{n \geq 1}$ such that $\eta^{\mathsf q}(v_n)$ converges to $\eta^{\mathsf q}(v)$ as $n \to \infty$ for $\mathsf q \in \{L,R\}$ in the sense that for any $\epsilon>0$, $\eta^{\mathsf q}(v_n)$ and $\eta^{\mathsf q}(v)$ coincide after the first time exiting $B(v,\epsilon)$ for all sufficiently large $n$ (depending on $\epsilon$). Hence, $\eta^{\mathsf q}(v)$ are simple for all $v \in \overline{\mathbb{B}}$ and the flow line interaction rule applies to all $v \in \overline{\mathbb{B}}$.

    \item \label{property:measurable} Let $B(z,r) \subset \mathbb{B}$, and for each $v \in B(z, r)$ and $\mathsf q \in \{L,R\}$, let $\tau_v^{\mathsf q}$ be the first exit time of $\eta^{\mathsf q}(v)$ from $B(z,r)$. Then, $\{\eta^{\mathsf q}(v)|_{[0,\tau_v^{\mathsf q}]}\}_{v \in B(z,r), \mathsf q \in \{L,R\}}$ is measurable with respect to $h|_{B(z,r)}$ modulo $2 \pi \chi \mathbb{Z}$.
\end{enumerate}

Observe that if $v \in \mathbb B$ satisfies $\eta^L(v) \cap \eta^R(v)=\{v\}$, then it necessarily lies on the frontier of the range of $\widetilde \eta$ stopped upon hitting $v$; see Figure~\ref{fig:def-n-good} (left). In view of this, the following lemma is immediate from the construction of chordal SLE from space-filling SLE.

\begin{lemma}\label{lem:flow-line-chordal}
    Almost surely, every point $v \in \mathbb{B}$ such that $\eta^L(v) \cap \eta^R(v) = \{v\}$ is on the trace of $\eta$.
\end{lemma}

\begin{figure}[htbp]
    \centering
    \begin{minipage}{0.46\textwidth}
    \includegraphics[width=\textwidth]{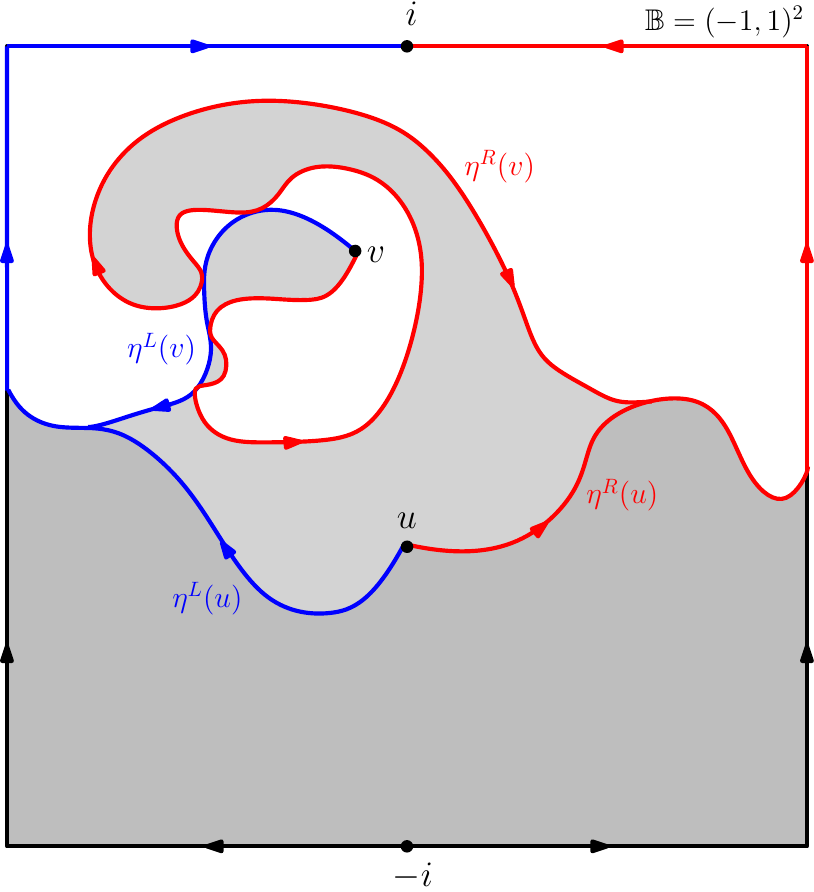}
    \end{minipage}
    \hfill
    \begin{minipage}{0.46\textwidth}
    \includegraphics[width=\textwidth]{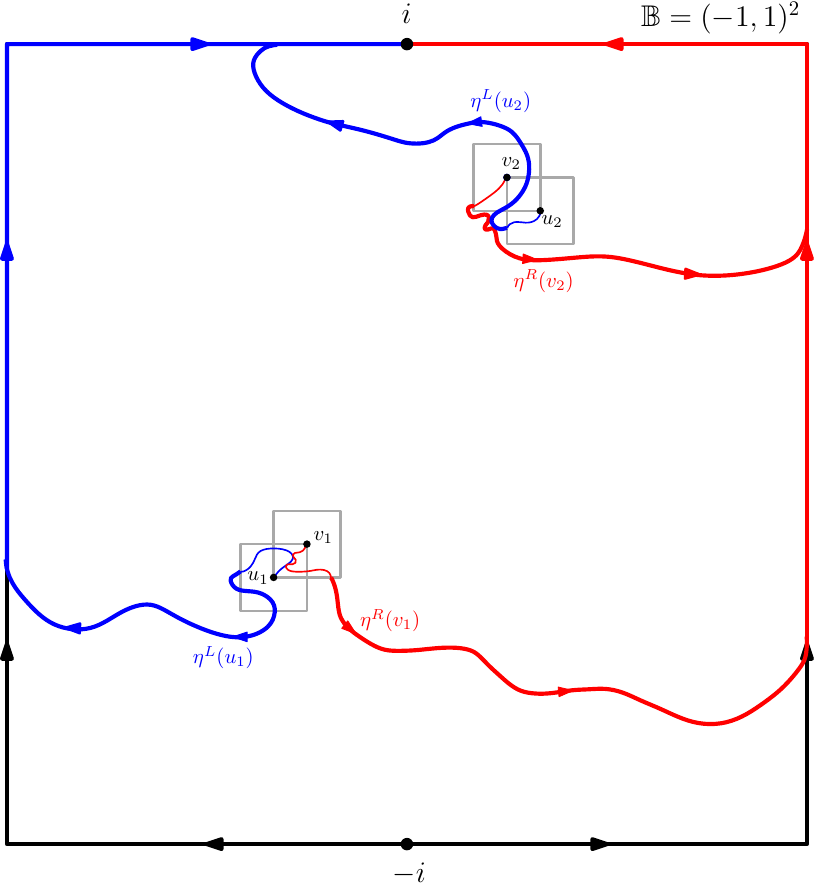}
    \end{minipage}
    \caption{\textbf{Left:} The region in gray (resp.\ in gray or light-gray) represents the range of $\widetilde \eta$ stopped upon hitting $u$ (resp.\ $v$). In this figure, the point $u$ lies on the trace of $\eta$ whereas $v$ does not. \textbf{Right:} In Definition~\ref{def:n-good}, the 0-good property requires that flow lines with different angles, after leaving the $\epsilon$-neighborhoods of starting points (marked in bold), are mutually non-intersecting. In this figure, the pair $(u_1,v_1)$ is admissible whereas $(u_2,v_2)$ is not.}
    \label{fig:def-n-good}
\end{figure}

In light of Lemma~\ref{lem:flow-line-chordal}, we define 
\begin{equation}\label{def:K}
    \mathcal{K} = \{ v \in [-1/2,1/2]^2: \eta^L(v) \cap \eta^R(v) = \{v\} \} \subset \eta([0,\infty]).
\end{equation}
The Hausdorff dimension of the set of points $v \in \overline{\mathbb{B}}$ such that $\eta^L(v) \cap \eta^R(v) = \{v\} $ is expected to be $4-16/\kappa$; see Remark~\ref{rem:dimension-k}. This is smaller than the Hausdorff dimension of $\eta$ which is $1 + \kappa/8$~\cite{Beffara-dimension}, but still tends to 2 as $\kappa$ tends to 8.

To prove Theorem~\ref{thm:carpet}, it suffices to show the following proposition.

\begin{proposition}\label{prop:carpet}
    There exists $\delta_0 > 0$ such that for $\kappa \in (8 - \delta_0,8)$, with positive probability, $\mathcal{K}$ contains a topological Sierpi\'nski carpet.
\end{proposition}

\begin{proof}[Proof of Theorem~\ref{thm:carpet} given Proposition~\ref{prop:carpet}]
    Let $\kappa \in (8-\delta_0, 8)$. Since Proposition~\ref{prop:carpet} is a topological property, it follows from conformal invariance that in any simply connected domain, the trace of a chordal SLE$_\kappa$ curve has a uniformly positive probability, denoted by $\mathfrak p = \mathfrak p(\kappa)$, of containing a topological Sierpi\'nski carpet. To prove Theorem~\ref{thm:carpet}, it suffices to find independent copies of chordal SLE$_\kappa$ in the trace of $\eta$, as we elaborate below. If each copy independently has a uniformly positive probability of containing a topological Sierpi\'nski carpet, then the trace of $\eta$ almost surely contains a topological Sierpi\'nski carpet.

    Let $\eta$ be a chordal SLE$_\kappa$ from $0$ to $\infty$ in $\mathbb{H}$. We parameterize $\eta$ by $[0,\infty)$ so that $\lim_{t \to \infty} \eta(t) = \infty$. Consider $D_k = \mathbb{H} \cap(B(0, 2^{k+1}) \setminus \overline{B(0, 2^k)})$ for $k \geq 1$. For each $k \geq 1$, let $\tau_k$ be the first time that $\eta$ enters $\overline{D}_k$ and $\widetilde{\tau}_k$ be the last time that $\eta$ exits $\overline{D}_k$. 
    By the reversibility of SLE$_\kappa$~\cite{IG3} and the domain Markov property, the conditional law of $\eta|_{[\tau_k, \widetilde{\tau}_k]}$ given $\eta|_{[0, \tau_k] \cup [\widetilde{\tau}_k, \infty)}$ is that of a chordal SLE$_\kappa$ from $\eta(\tau_k)$ to $\eta(\widetilde{\tau}_k)$ in the appropriate connected component of $\mathbb{H} \setminus \eta([0, \tau_k] \cup [\widetilde{\tau}_k, \infty))$. It follows that for any increasing sequence $\{k_n \}_{1 \leq n \leq M}$, conditioned on $E:= \{ \tau_{k_1} < \widetilde{\tau}_{k_1} < \tau_{k_2} < \widetilde{\tau}_{k_2} < \cdots < \tau_{k_M} < \widetilde \tau_{k_M}\}$, the events that $\eta|_{[\tau_k, \widetilde{\tau}_k]}$ contains a topological Sierpi\'nski carpet are independent. Therefore, with conditional probability at least $1 - (1 - \mathfrak p)^M$, at least one of these events occurs, which implies that the trace of $\eta$ contains a topological Sierpi\'nski carpet. Since $\eta$ is a continuous curve tending to infinity, for any fixed $M$, we can choose an increasing sequence $\{k_n\}_{1 \leq n \leq M}$ such that $E$ occurs with probability at least $1 - o_M(1)$. Taking $M$ to infinity yields the desired result.
\end{proof}

\subsection{Finding the Sierpi\'nski carpet}\label{subsec:disconnectivity}

In this section, we outline the proof of Proposition~\ref{prop:carpet} similarly to the proof strategy of Proposition~\ref{prop:carpet-fractal}. The very first idea is that the trace of SLE$_\kappa$ is a random self-similar fractal set, whose dimension $1+\kappa/8$~\cite{Beffara-dimension} is close to 2 when $\kappa$ is close to 8. Therefore, it is plausible that the trace of SLE$_\kappa$ can be compared to a fractal percolation model with parameter $p$ close to 1.

Let $\epsilon \in (0,100^{-1})$ be the coarse-gaining scale, which will be chosen explicitly in Proposition~\ref{prop:sle-rsw}.

\begin{definition}[$n$-good]\label{def:n-good}
We say that the origin is \textbf{$0$-good} if for any $u,v \in (-3/4, 3/4)^2$, the flow lines $\eta^L(u)$ and $\eta^R(v)$ do not intersect after leaving $B(u, \epsilon)$ and $B(v, \epsilon)$, respectively. More precisely, let $\tau_u$ denote the first exit time of $\eta^L(u)$ from $B(u, \epsilon)$, and let $\tau_v$ denote the first exit time of $\eta^R(v)$ from $B(v, \epsilon)$. The $0$-good property requires that $\eta^L(u)|_{[\tau_u,\infty)}$ and $\eta^R(v)|_{[\tau_v,\infty)}$ are disjoint for any $u,v \in (-3/4,3/4)^2$. Note that $u$ and $v$ are not necessarily different. Otherwise, we say the origin is \textbf{$0$-bad}.

For an integer $n \geq 1$ and $x \in \epsilon^n \mathbb{Z}^2 \cap (-3/4,3/4)^2$, we say that $x$ is \textbf{$n$-good} if for any $u,v \in B(x, 3 \epsilon^n)$, the flow lines $\eta^L(u)$ and $\eta^R(v)$ do not intersect after first leaving $B(u, \epsilon^{n+1})$ and $B(v, \epsilon^{n+1})$, respectively, and before leaving $B(x, 4 \epsilon^n)$. Otherwise, we say that $x$ is \textbf{$n$-bad}.
\end{definition}

See Figure~\ref{fig:def-n-good} (right) for an illustration. One should view the $n$-good property as an analog to the $0$-good condition $\sigma_B=1$ in fractal percolation. We introduce an additional labeling $n$ since the $n$-good property here is not homogeneous across scales due to the imaginary geometry boundary conditions, whereas the condition $\sigma_B = 1$ in fractal percolation is homogeneous across scales. Another key difference is that the $n$-goodness is not independent among different vertices on $\epsilon^n \mathbb{Z}^2$; however, it is a local event of the GFF (see Observation~\eqref{claim:remark-mn-3} in Remark~\ref{rmk:(m,n)-good} below), which allows us to control the correlation between distant vertices using the local absolute continuity of the GFF.

The following lemma allows us to find part of $\mathcal{K}$ using the $n$-good property: imagine removing from $[-1/2,1/2]^2$ all boxes $B(x, \epsilon^n)$ where $x$ is an $n$-bad vertices and $n \geq 1$. Assuming that the origin is $0$-good, the retained set is then contained in $\mathcal{K}$. This procedure is similar to that used in fractal percolation. The proof of Lemma~\ref{lem:flow-line-not-touch} will be given in Section~\ref{sec:SLE}.

\begin{lemma}\label{lem:flow-line-not-touch}
    The following holds almost surely. For any $w \in [-1/2,1/2]^2$, suppose that the following conditions are satisfied:
    \begin{enumerate}[(1)]
        \item The origin is $0$-good;
        \item For any $n \geq 1$ and $x \in \epsilon^n \mathbb{Z}^2$ with $w \in B(x, \epsilon^n)$, the point $x$ is $n$-good.
    \end{enumerate}
    Then, $\eta^L(w) \cap \eta^R(w) = \{w \}$, namely, $w \in \mathcal{K}$.
\end{lemma}

The following lemma, which will also be proved in Section~\ref{sec:SLE}, shows that the $n$-good property holds with probability close to 1 when $\kappa$ is close to 8. This follows from the observation that if the flow lines touch, then the chordal SLE$_\kappa$ will form bubbles. On the other hand, as $\kappa$ tends to 8, one can show that all the bubbles of a chordal SLE$_\kappa$ in the bulk have diameters tending to 0; see Lemma~\ref{lem:bubble-diameter}.

\begin{lemma}\label{lem:kappa-to-8}
    Fix any $q>0$ and $\epsilon < 100^{-1}$. There exists $\delta = \delta(q, \epsilon) >0$ such that for all $\kappa \in (8 - \delta, 8)$, we have
    \begin{align*}
    &\mathbb{P}[\mbox{the origin is } 0 \mbox{-good}] \geq 1-q,\\
    \mbox{and} \quad &\mathbb{P}[x \mbox{ is } n \mbox{-good}] \geq 1 - q \quad \mbox{for all } n \geq 1 \mbox{ and } x \in \epsilon^n \mathbb{Z}^2 \cap (-3/4,3/4)^2.
    \end{align*}
\end{lemma}

For simplicity, we define $V_0 = \{0\}$ and $V_n = \epsilon^n \mathbb{Z}^2 \cap (-3/4,3/4)^2$. Recall from Definition~\ref{def:n-good} that the $n$-good and $n$-bad properties have been defined for points in $V_n$. We now introduce the $(m,n)$-good property for points in $V_n$, which is an analog of $m$-goodness in fractal percolation.

\begin{definition}[$(m,n)$-good]\label{def:(m,n)-good}
    We define the \textbf{$(m,n)$-good} property for points in $V_n$ and integers $m \geq 0$. If a point is not $(m,n)$-good, we say that it is \textbf{$(m,n)$-bad}.
    The definition is given inductively on $m$. For $m = 0$ and $n \geq 0$, we say $x \in V_n$ is \textbf{$(0,n)$-good} if and only if $x$ is $n$-good.

    For $m \geq 1$, given the definition of $(j,n)$-good for $0 \leq j \leq m-1$, we define the $(m,n)$-good property for $x \in V_n$ as follows. Let $J_x$ be the set of points in $B(x,\epsilon^n) \cap V_{n+1}$ that are either $(m-1,n+1)$-bad or $(0, n+1)$-bad. Consider the collection of closed boxes $\{ \overline{B(w, \epsilon^{n+1})} \}_{w \in J_x}$ and the connected components formed by their union on $\mathbb{C}$, where two boxes that share exactly one point are also viewed as connected. We say that $x$ is \textbf{$(m,n)$-good} if all these components have $|\cdot|_\infty$-diameter at most $\frac{1}{4} \epsilon^n$.
\end{definition}

\begin{remark}\label{rmk:(m,n)-good}
    We will use the following observations along the proof.
    \begin{enumerate}[(1)]
        \item \label{claim:remark-mn-1} For $m \ge 1$, the $(m,n)$-good condition becomes sharper as $m$ increases. Specifically, for any $1 \le m<m'$ and $n \geq 0$, if $x \in V_n$ is $(m',n)$-good, then $x$ is also $(m,n)$-good.
        
        \item \label{claim:remark-mn-3} It follows from Property~\eqref{property:measurable} that for any $n \geq 1$ and $x \in V_n$, the event that $x$ is $(m,n)$-good is measurable with respect to $h|_{B(x, 4 \epsilon^n)}$ modulo $2 \pi \chi \mathbb{Z}$.
    \end{enumerate}
\end{remark}

For $m \geq 0$, we define the probability
\begin{equation}\label{eq:def-thetam}
    \theta_m = \theta_m(\epsilon) := \inf_{n \geq 0} \inf_{x \in V_n} \mathbb{P}[x \mbox{ is } (m,n)\mbox{-good}].
\end{equation}
Note that when $m=0$, the $(m,n)$-good property is equivalent to the $n$-good property. Hence, the inequalities in Lemma~\ref{lem:kappa-to-8} are equivalent to $\theta_0 \geq 1-q$. The following lemma is an analog of the induction inequality in Proposition~\ref{prop:fractal-rsw}. Since the $(m,n)$-good property is not independent among different vertices, we need an alternative argument relying on the percolation of the GFF, which will be postponed to Section~\ref{sec:GFF}. The proof is based on Observation~\eqref{claim:remark-mn-3} in Remark~\ref{rmk:(m,n)-good}, a coarse-graining argument for the GFF, and Peierls' argument; however, it does not require any knowledge of SLE.

\begin{lemma}\label{lem:induction-bound}
    There exists a universal constant $c_1>0$ and a function $\mathcal{E}: (0,1) \to (0,1)$ such that for all $\chi \in (\frac{1}{\sqrt{6}}, \frac{1}{\sqrt{2}})$, $p \in (0, c_1)$, $\epsilon \in (0,  \mathcal{E}(p))$ and $m \geq 1$, the following holds:
    \[ \mbox{If } \min\{ \theta_0, \theta_{m-1}\} \geq 1-p, \mbox{ then } \theta_m \geq 1-p. \]
\end{lemma}

The following proposition is analogous to Proposition~\ref{prop:fractal-rsw}.

\begin{proposition}\label{prop:sle-rsw}
    There exist $\delta > 0$ and a dyadic $\epsilon < 100^{-1}$ such that the following holds:
    \[ \theta_m \geq \tfrac{2}{3} \quad \mbox{for all } m \geq 0 \mbox{ and } \kappa \in (8 - \delta,8). \]
\end{proposition}

\begin{proof}[Proof of Proposition~\ref{prop:sle-rsw} given Lemmas~\ref{lem:kappa-to-8} and~\ref{lem:induction-bound}]
    Applying Lemma~\ref{lem:kappa-to-8} with $q = \min \{c_1/2, 1/3 \}$ and a dyadic $\epsilon < \min \{100^{-1}, \mathcal E(q) \}$, we obtain a constant $\delta = \delta(q, \epsilon)>0$ such that $\theta_0 \geq 1-q$ for all $\kappa \in (8 - \delta, 8)$. Combining this with Lemma~\ref{lem:induction-bound} and an induction argument yields that $\theta_m \geq 1- q \geq 2/3$ for all $m \geq 0$ and $\kappa \in (8 - \delta, 8)$.
\end{proof}

The following proposition, which will be proved in Section~\ref{sec:carpet}, is based on Lemma~\ref{lem:flow-line-not-touch} and observations similar to~\eqref{eq:fractal-carpet-1} in Proposition~\ref{prop:carpet-fractal}.

\begin{proposition}\label{prop:find-carpet}
    For any dyadic $\epsilon<100^{-1}$, on the event that the origin is $(m,0)$-good for all $m \geq 0$, the set $\mathcal{K}$ contains a topological Sierpi\'nski carpet.
\end{proposition}

Finally, we complete the proof of Proposition~\ref{prop:carpet} assuming Lemmas~\ref{lem:flow-line-not-touch}, \ref{lem:kappa-to-8} and \ref{lem:induction-bound}, as well as Proposition~\ref{prop:find-carpet}.

\begin{proof}[Proof of Proposition~\ref{prop:carpet}]
    Note that Proposition~\ref{prop:sle-rsw} combined with Observation~\eqref{claim:remark-mn-1} in Remark~\ref{rmk:(m,n)-good} implies that the event in Proposition~\ref{prop:find-carpet} occurs with probability at least $1/3$.
\end{proof}

\section{Regularity estimates for SLE$_\kappa$ when $\kappa$ is close to 8}\label{sec:SLE}

In this section, we analyze the behavior of flow lines and SLE$_\kappa$ as $\kappa$ tends to 8. First, in Section~\ref{subsec:sle-gff-prelim}, we will collect the definitions and basic properties of different variants of SLE and GFF; the latter will also be used several times in Section~\ref{sec:GFF}. In Section~\ref{subsec:ig-prelim}, we provide an overview of the imaginary geometry coupling of SLE and GFF, and the construction of space-filling SLE. Specifically, we will make sense of the flow lines emanating from every point as mentioned in Section~\ref{subsec:setup-ig}. In Section~\ref{subsec:flow-line-not-touch}, we prove Lemma~\ref{lem:flow-line-not-touch}. Finally, we will prove Lemma~\ref{lem:kappa-to-8} in Section~\ref{subsec:n-good-probability}.

\subsection{Preliminaries on SLE and GFF}\label{subsec:sle-gff-prelim}

As mentioned in~\eqref{eq:loewner}, chordal SLE$_\kappa$ is a random curve generated by a family of conformal maps $(g_t)_{t \ge 0}$ that solve the Loewner equation with driving function $W_t=\sqrt{\kappa} B_t$. In this work, we will also consider variants of SLE$_\kappa$ called chordal SLE$_\kappa(\underline{\rho})$, in which one keeps track of several additional points, which we refer to as force points.

For $\kappa>0$, SLE$_\kappa(\underline{\rho})$ processes on the upper half-plane $\mathbb H$ with force points $x_1,\ldots,x_m \in \mathbb{R}$ and weights $\rho_1,\ldots,\rho_m$ are defined by the same Loewner equation~\eqref{eq:loewner}, except that the driving function $W$ is given by the solution of
\begin{equation}\label{eq:loewner-general}
    \dd W_t=\sqrt{\kappa} \dd B_t+\sum_{j=1}^m \frac{\rho_j}{W_t-g_t(x_j)} \dd t,
\end{equation}
where the evolutions of force points, $(g_t(x_j))_{t \ge 0}$, are determined by solving~\eqref{eq:loewner} for $z \in \overline{\mathbb H}$. It was proved in~\cite{IG1} that~\eqref{eq:loewner-general} has a unique solution up until the continuation threshold $\inf\{t \ge 0:\sum_{j:x_j<0, g_t(x_j)=W_t} \rho_j \wedge \sum_{j:x_j>0, g_t(x_j)=W_t} \rho_j \le -2\}$, and that SLE$_\kappa(\underline{\rho})$ is a continuous curve. For a simply connected domain $D$ with boundary points $x,y$, the chordal SLE$_\kappa(\underline{\rho})$ in $D$ from $x$ to $y$ is defined using conformal maps.

Let us now review the definition and basic properties of the two-dimensional Gaussian free field (GFF). We refer the reader to~\cite{Sheffield-gff-notes, Dubedat-SLE-GFF, WP-gff-notes} for more details. Let $D \subseteq \mathbb C$ be a simply connected domain with harmonically non-trivial boundary. We denote by $H_0(D)$ the Hilbert closure of $C_0^\infty(D)$ under the Dirichlet inner product $\langle f,g \rangle_\nabla=(2\pi)^{-1} \int_D \nabla f \cdot \nabla g \dd x$. Let $(f_n)_{n \ge 1}$ be an orthonormal basis of $H_0(D)$, and let $(\alpha_n)_{n \ge 1}$ be a sequence of i.i.d. standard Gaussian random variables. The random series $\sum_{n \ge 1} \alpha_n f_n$ converges almost surely in the space of distributions, whose limit is called a \emph{zero-boundary} GFF on $D$. The GFF on $D$ with boundary conditions $\xi$ is referred to as the sum of a zero-boundary GFF and a (deterministic) harmonic function that agrees with $\xi$ on $\partial D$.

The whole-plane GFF is defined in a similar way. For $D=\mathbb C$, let $H(D)$ be the Hilbert closure of $\{f \in C^\infty(D): \int_D f \dd x=0\}$ under the Dirichlet inner product $\langle \cdot,\cdot \rangle_\nabla$. Let $(g_n)_{n \ge 1}$ be an orthonormal basis of $H(D)$, and let $(\alpha_n)_{n \ge 1}$ be a sequence of i.i.d. standard Gaussian random variables. Then $\sum_{n \ge 1} \alpha_n g_n$ viewed as a random distribution modulo global additive constant is called a whole-plane GFF. For $s>0$, we can define the whole-plane GFF modulo $s\mathbb Z$ (so that two distributions are equivalent if they differ by an integer multiple of $s$). We first sample a whole-plane GFF $\mathfrak h$ modulo global additive constant. Then we independently sample a uniform random variable $U$ from $[0,s)$ and fix the additive constant of $\mathfrak h$ by requiring that the average of $\mathfrak h$ on the unit circle is equal to $U$. 

For a random distribution $h$ and an open set $V$, we define the $\sigma$-algebra generated by $h|_V$ as the $\sigma$-algebra generated by the distributional pairing $\phi \mapsto (h, \phi)$, where $\phi$ ranges over smooth functions compactly supported in $V$. For $s>0$, we define the $\sigma$-algebra generated by $h|_V$ modulo $s \mathbb Z$ as follows. Fix a reference function $\phi_0$ that is compactly supported in $V$ and satisfies $\int \phi_0 \dd x = 1$. Let $(h,\phi_0) = ns + r$ where $n$ is an integer and $r \in [0,s)$. The $\sigma$-algebra is generated by the distributional pairing $\phi \mapsto (h - ns, \phi)$, where $\phi$ ranges over smooth functions compactly supported in $V$. We extend these notation to a closed set $K$ by taking the intersection of the $\sigma$-algebras with respect to $O(K,\epsilon)$ for $\epsilon>0$, where $O(K,\epsilon)$ is the $\epsilon$-neighborhood of $K$.

We now recall some basic properties of the GFF. The zero-boundary GFF $h$ on a domain $D \subset \mathbb{C}$ with harmonically non-trivial boundary satisfies the domain Markov property. Specifically, for any subdomain $U \subset D$ with harmonically non-trivial boundary, if $h^U$ denotes the harmonic extension of $h|_{D \setminus U}$ on $U$. Then, $h - h^U$ is a zero-boundary GFF on $U$ and is independent of $\sigma(h|_{D \setminus U})$. We also recall the domain Markov property for the whole-plane GFF from~\cite[Lemma 2.2]{GMS-harmonic}. We will only need the case for boxes.

\begin{lemma}[Domain Markov property]\label{lem:whole-plane-markov}
    For any $B(z,r) \subset B(0,R)$, let $\mathfrak h$ be a whole-plane GFF normalized so that $\mathfrak h_R(0) = 0$, where $\mathfrak h_R(0)$ is the average of $\mathfrak h$ over $\partial B_R(0)$. let $\mathfrak h^{z,r}$ be the harmonic extension of $\mathfrak h|_{\mathbb{C} \setminus B(z,r)}$ on $B(z,r)$, and define $\mathring{\mathfrak h}^{z,r} = \mathfrak h - \mathfrak h^{z,r}$. Then, $\mathring{\mathfrak h}^{z,r}$ is independent of $\mathfrak h|_{\mathbb{C} \setminus B(z,r)}$, and has the same law as a zero-boundary Gaussian free field on $B(z,r)$.
\end{lemma}

The following estimates will be used repeatedly in the proof.

\begin{lemma}\label{lem:gff-estimate-5.11}
    Fix $s \in (0,1)$, there exists a universal constant $C = C(s)>0$ such that the following holds.
    \begin{enumerate}
        \item Let $\mathfrak h$ be a whole-plane GFF normalized so that $\mathfrak h_R(0) = 0$. For any $B(z,r) \subset B(0,R)$, let $\mathfrak h^{z,r}$ be the harmonic extension of $\mathfrak h|_{\mathbb{C} \setminus B(z,r)}$ in $B(z,r)$. Then, for any $t \geq 1$,
        $$
        \mathbb{P} \Big[ \sup_{u,v \in B(z, sr)} |\mathfrak h^{z, r}(u) - \mathfrak h^{z,r}(v)| \geq t \Big] \leq C e^{-t^2 / C}.
        $$ \label{lem5.9-claim-1}
        \item Let $B(z,r)$ and $B(y,l)$ be two boxes such that $B(z,r) \subset B(y,l)$. Let $h_{y,l}$ be a zero-boundary GFF on $B(y,l)$ and let $h^{z,r}_{y,l}$ be the harmonic extension of $h_{y,l}|_{B(y,l) \setminus B(z,r)}$ on $B(z,r)$. Then, for any $t \geq 1$,
        $$
        \mathbb{P} \Big[ \sup_{u,v \in B(z, sr)} |h^{z, r}_{y,l}(u) - h^{z,r}_{y,l}(v)| \geq t \Big] \leq C e^{-t^2 / C}.
        $$ \label{lem5.9-claim-2}
    \end{enumerate}
\end{lemma}

\begin{proof}
    By Green's function estimates and Fernique's inequality, it is easy to see that the expectation of the supremum is uniformly bounded; see, e.g.,~\cite[Lemma 2.9]{DingDunlap20}. The tail estimates then follow from the Borell-TIS inequality, see, e.g.,~\cite[Theoerm 2.1.1]{adler-taylor-fields}.
\end{proof}

We will also rely on the following local absolute continuity property.

\begin{lemma}[Local absolute continuity]\label{lem:whole-plane-absolute}
    Fix $0 < b < b' < 1$ and $T>0$. For any $M>0$ and $\alpha>0$, there exists a constant $C = C(b, b', T, \alpha, M)>0$ such that the following holds. Let $h$ be a zero-boundary GFF on $B(z,r)$. Let $g$ be any deterministic harmonic function in $B(z,r)$ satisfying $\sup_{u,v \in B(z,b'r)} |g(u) - g(v)| \leq M$, and let $s \in (0,T)$. Then, $(h+g)|_{B(z, br)}$ and $h|_{B(z,br)}$ are absolutely continuous with respect to each other viewed as random generalized functions modulo $s \mathbb{Z}$. Furthermore, the Radon-Nikodym derivative $H_g$ of the former with respect to the latter satisfies
        $$
        \max\{ \mathbb{E}[(H_g)^{-\alpha}], \mathbb{E}[(H_g)^\alpha] \} \leq C.
        $$
\end{lemma}

\begin{proof}
    We first replace $g$ with $\tilde g:= g - ns$ for some integer $n$ such that $\sup_{u \in B(z,b'r)} |\tilde g(u)| \leq M + T$. Similar to~\cite[Lemma 4.1]{MQ18-geodesic}, we can explicitly compute the Radon-Nikodym derivative $H_{\tilde g}$ of $(h+g - ns)|_{B(z, br)}$ and $h|_{B(z,br)}$ viewed as random generalized functions and show that $\max\{ \mathbb{E}[(H_{\tilde g})^{-\beta}], \mathbb{E}[(H_{\tilde g})^\beta] \} \leq C'(\beta)$ for any $\beta>0$. The desired lemma then follows from H\"older's inequality.
\end{proof}

The following lemma, which will be used in Section~\ref{subsec:n-good-probability}, is a consequence of Lemmas~\ref{lem:gff-estimate-5.11} and~\ref{lem:whole-plane-absolute}.

\begin{lemma}\label{lem:transfer}
    For $\kappa \in (6,8)$, let $\mathbb{P}$ be the law of a imaginary geometry GFF on $\mathbb{B}$ as defined at the beginning of Section~\ref{subsec:setup-ig}, and $\widetilde{\mathbb{P}}$ be the law of a whole-plane GFF modulo $2 \pi \chi \mathbb{Z}$. Let $B(z,r) \subset (-7/8,7/8)^2$ and $A$ be an event that is measurable with respect to $h|_{B(z,r)}$ modulo $2 \pi \chi \mathbb{Z}$. Suppose that $\mathbb P[A]=1-o_\nu(1)$ as $\nu \to 0$, then $\widetilde{\mathbb P}[A]=1-o_\nu(1)$ with a rate uniformly in $\kappa$ and $B(z,r)$; and vice versa.
\end{lemma}

\begin{proof}
    Write $h$ for the imaginary geometry GFF on $\mathbb B$. Let $h'$ be the harmonic extension of $h|_{\mathbb B \setminus (-31/32,31/32)^2}$ on $(-31/32,31/32)^2$. For $M>0$, let $E_M$ be the event that
    \begin{equation}\label{eq:bounded-fluctuation}
        \sup_{u,v \in (-15/16,15/16)^2} |h'(u)-h'(v)| \le M.
    \end{equation}    
    Let $\mathcal{F}$ be the $\sigma$-algebra generated by $h|_{\mathbb B \setminus (-31/32,31/32)^2}$, and let $\widehat{\mathbb P}$ be the law of a zero-boundary GFF on $(-31/32,31/32)^2$. By Lemma~\ref{lem:whole-plane-absolute} and Cauchy-Schwarz inequality, on the event $E_M$, we have
    \[ \widehat{\mathbb P}[A^c] =\mathbb E[(H_{h'})^{-1} \mathbf 1_{A^c}|\mathcal{F}] \le \mathbb E[(H_{h'})^{-2}|\mathcal{F}]^{1/2} \cdot \mathbb P[A^c|\mathcal F]^{1/2} \le c(M) \cdot \mathbb P[A^c|\mathcal F]^{1/2} \mbox{ for some constant } c(M)>0.\]
    By Lemma~\ref{lem:gff-estimate-5.11}, the event $E_M$ occurs with probability $1-o_M(1)$ as $M \to \infty$. Sending $M \to \infty$ and then $\nu \to 0$ gives $\widehat{\mathbb P}[A^c]=o_\nu(1)$. Using the same argument, we get $\widetilde{\mathbb P}[A^c]=o_\nu(1)$ which concludes the first assertion. The other direction follows from a similar argument.
\end{proof}

\subsection{Further background on imaginary geometry}\label{subsec:ig-prelim}

In this subsection, we provide a further review of imaginary geometry~\cite{IG1,IG4}. For $\kappa>4$ and $\underline{\kappa}=16/\kappa \in (0,4)$, write
\begin{equation}\label{eq:ig-parameter}
    \lambda=\frac{\pi}{\sqrt{\underline{\kappa}}}, \quad \lambda'=\frac{\pi}{\sqrt{\kappa}}, \quad \mbox{and} \quad \chi=\frac{2}{\sqrt{\underline{\kappa}}}-\frac{\sqrt{\underline{\kappa}}}{2}.
\end{equation}
Let $\eta$ be an SLE$_{\underline{\kappa}}(\underline{\rho})$ from 0 to $\infty$ on $\mathbb H$, with force points located at $\underline{x}=(\underline{x}_L,\underline{x}_R)=(x_{k,L}<\cdots<x_{1,L}<0<x_{1,R}<\cdots<x_{\ell,R})$, and corresponding weights $\underline{\rho}=(\underline{\rho}_L;\underline{\rho}_R)=(\rho_{k,L},\ldots,\rho_{1,L};\rho_{1,R},\ldots,\rho_{\ell,R})$. Let $h$ be the GFF on $\mathbb H$ with boundary data given by
\begin{equation}\label{eq:ig-boundary-data-simple}
    -\lambda(1+\sum_{i=1}^j \rho_{i,L}) \quad \mbox{for} \quad x \in (x_{j+1,L},x_{j,L}] \quad \mbox{and} \quad \lambda(1+\sum_{i=1}^{j'} \rho_{i,R}) \quad \mbox{for} \quad x \in (x_{j',R},x_{j'+1,R}]
\end{equation}
for $0 \le j \le k$ and $0 \le j' \le \ell$, where we set $x_{0,L}=0^-$, $x_{0,R}=0^+$, $x_{\ell+1,L}=-\infty$, $x_{k+1,R}=+\infty$ by convention. By~\cite[Theorem 1.1]{IG1}, $h$ can be coupled with $\eta$ in the following sense. Let $(g_t)_{t \ge 0}$ and $W$ be the family of conformal maps and the driving function associated with $\eta$ in the Loewner equation, and write $f_t=g_t-W_t$. Then, for each stopping time $\tau$ of $\eta$, the conditional law given $\eta([0,\tau])$ of $h \circ f_\tau^{-1}-\chi \arg (f_\tau^{-1})'$ is a GFF on $\mathbb H$ with boundary data given by $-\lambda$ for $x \in (f_\tau(x_{0,L}),0^-]$, $\lambda$ for $x \in (0^+,f_\tau(x_{0,R})]$, and
\[ -\lambda(1+\sum_{i=1}^j \rho_{i,L}) \quad \mbox{for} \quad x \in (f_\tau(x_{j+1,L}),f_\tau(x_{j,L}]) \quad \mbox{and} \quad \lambda(1+\sum_{i=1}^{j'} \rho_{i,R}) \quad \mbox{for} \quad x \in (f_\tau(x_{j',R}),f_\tau(x_{j'+1,R})] \]
for $0 \le j \le k$ and $0 \le j' \le \ell$. Moreover, $\eta$ is a.s. locally determined by $h$ in this coupling~\cite[Theorem 1.2]{IG1}. We call $\eta$ the \emph{flow line} of $h$ from 0 to $\infty$. The flow line of $h$ with angle $\theta \in \mathbb R$ is defined as the flow line of $h+\theta\chi$. 

For $\kappa>4$, SLE$_\kappa(\underline{\rho})$ process $\eta'$ from 0 to $\infty$ on $\mathbb H$ can be coupled with a GFF $-h$ on $\mathbb H$ in a similar manner, where the boundary data of $h$ is given by
\begin{equation}\label{eq:ig-boundary-data-nonsimple}
    \lambda'(1+\sum_{i=1}^j \rho_{i,L}) \quad \mbox{for} \quad x \in (x_{j+1,L},x_{j,L}] \quad \mbox{and} \quad -\lambda'(1+\sum_{i=1}^{j'} \rho_{i,R}) \quad \mbox{for} \quad x \in (x_{j',R},x_{j'+1,R}]
\end{equation}
for $0 \le j \le k$ and $0 \le j' \le \ell$. In this coupling, $\eta'$ is a.s. locally determined by $h$ and is called the \emph{counterflow line} of $h$ from 0 to $\infty$.

We can also define flow lines starting from interior points. Let $\mathfrak h$ be the whole-plane GFF modulo $2\pi\chi \mathbb Z$. Then we can generate flow lines of $\mathfrak h$ starting from a point $z \in \mathbb C$ with different angles, the marginal law of which is a whole-plane SLE$_{\underline{\kappa}}(2-\underline{\kappa})$ process from $z$ to $\infty$~\cite[Theorem 1.1]{IG4}. For $\underline{\kappa} \in (0,8/3]$ (i.e. $\kappa \ge 6$), the flow lines are simple curves (Property~\eqref{property:simple}). Moreover, these flow lines are a.s. locally determined by the field~\cite[Theorem 1.2]{IG4}. We remind that adding a global additive multiple of $2\pi\chi$ to the field does not change its flow lines. Similarly, if we start with a GFF $h$ on a simply connected domain $D \subseteq \mathbb C$ and then view it as a distribution modulo $2\pi\chi \mathbb Z$, we can also define its flow lines starting from interior points. For simplicity, we denote by $\eta^L(v)$ (resp.\ $\eta^R(v)$) the flow line starting from $v$ with angle $\pi/2$ (resp.\ $-\pi/2$).

We now describe the flow line interaction rules, which are elaborated in~\cite[Theorem 1.5]{IG1} for boundary emanating flow lines and in~\cite[Theorem 1.7]{IG4} for interior emanating ones. Let $\theta_c:=\pi\underline{\kappa}/(4-\underline{\kappa})$ be the critical angle. Let $\eta_1,\eta_2$ be two boundary emanating flow lines started from $x_1 \ge x_2$ with angles $\theta_1,\theta_2 \in \mathbb R$, respectively, then the interaction of $\eta_1,\eta_2$ depends on the \emph{angle gap} $\theta_2-\theta_1$: If $\theta_2-\theta_1>0$, then $\eta_1$ almost surely stays to the right of $\eta_2$. Moreover, if $\theta_2-\theta_1 \in (0,\theta_c)$, then $\eta_1$ and $\eta_2$ may bounce off each other, otherwise they almost surely do not intersect. If $\theta_2-\theta_1=0$, then $\eta_1$ merges with $\eta_2$ upon intersecting. If $\theta_2-\theta_1 \in (-\pi,0)$, then $\eta_1$ crosses $\eta_2$ upon intersecting and never crosses back. For interior emanating flow lines $\eta_1,\eta_2$, we have that $\eta_1$ may hit the right side of $\eta_2$ if the angle gap upon intersecting is in $(-\pi,\theta_c)$, and their subsequent behavior can be described in the same way as the boundary emanating case. Note that for $\kappa \in (4,8)$ (so that $\underline{\kappa} \in (2,4)$ and $\theta_c>\pi$), the flow lines $\eta^L(\cdot)$ and $\eta^R(\cdot)$ may bounce off but never cross each other.

We now review the construction of space-filling SLE$_\kappa$~\cite{IG4}. We start with the chordal space-filling SLE$_\kappa(\rho_1;\rho_2)$ from 0 to $\infty$ on $\mathbb H$ for $\kappa>4$ and $\rho_1,\rho_2 \in (-2,\frac{\kappa}{2}-2)$, which extends to other simply connected domains by conformal maps. Let $h$ be a GFF on $\mathbb H$ with boundary data given by $\lambda'(1+\rho_1)$ on $\mathbb R_-$ and $-\lambda'(1+\rho_2)$ on $\mathbb R_+$. Let $(x_k)$ be a countable dense set in $\mathbb H$. We give an ordering of $(x_k)$ by declaring that $x_j$ comes before $x_k$ if $\eta^L(x_j)$ merges with $\eta^L(x_k)$ on its left side, or equivalently, $\eta^R(x_j)$ merges with $\eta^R(x_k)$ on its right side. Here, we also view $\mathbb R_-$ and $\mathbb R_+$ as flow lines (so the interaction rules apply) and define the orientation to be from $\infty$ to 0 if $\rho_j \in (-2,\frac{\kappa}{2}-4]$, and from 0 to $\infty$ if $\rho_j \in (\frac{\kappa}{2}-4,\frac{\kappa}{2}-2)$. In particular, when $\rho_1=\rho_2=0$, the flow lines $\eta^L(\cdot)$ and $\eta^R(\cdot)$ are both oriented toward the target point after merging with the boundary segments. The chordal space-filling SLE$_\kappa(\rho_1;\rho_2)$ is then defined as the unique space-filling curve that visits $(x_n)$ according to this ordering, and is a.s. continuous when parametrized by Lebesgue measure. In this coupling, the curve and the field $h$ almost surely determine each other~\cite[Theorem 1.16]{IG4}. When reparameterized by half-plane capacity, the curve is almost surely equal to the counterflow line of $h$ from 0 to $\infty$, which is an ordinary chordal SLE$_\kappa(\rho_1;\rho_2)$.

We also consider the whole-plane space-filling SLE$_\kappa$ from $\infty$ to $\infty$~\cite[Section 1.4.1]{DMS21-LQG-MRT}, which can be defined similarly to the chordal case. Let $\mathfrak h$ be the whole-plane GFF modulo $2\pi\chi \mathbb Z$. For $v \in \mathbb C$, let $\overline\eta^L(v)$ (resp.\ $\overline\eta^R(v)$) be the flow line of $\mathfrak h$ starting from $v$ with angle $\pi/2$ (resp.\ $-\pi/2$). For a countable dense set $(x_k)$ of $\mathbb C$, we declare that $x_j$ comes before $x_k$ if $\overline\eta^L(x_j)$ merges with $\overline\eta^L(x_k)$ on its left side. Then there is a unique space-filling curve that visits $(x_n)$ according to this order, which is defined to be the whole-plane space-filling SLE$_\kappa$. Another way to visualize this is that flow lines with the same angle eventually merge and form a space-filling tree, and the whole-plane space-filling SLE$_\kappa$ is a Peano curve that traces the boundary of the entire tree.

In this work, it will be important to consider flow lines emanating from \emph{every} point simultaneously. We will explain for the chordal case. Recall the setting in Section~\ref{subsec:setup-ig}. For each $u \in \overline{\mathbb B}$, let $\sigma_u$ be the first time that $\widetilde \eta$ hits $u$. For all $v \in \overline{\mathbb B} \cap \mathbb Q^2$, the $\pi/2$ and $-\pi/2$ angle flow lines $\eta^L(v)$ and $\eta^R(v)$ emanating from $v$ coincide with the left and right boundaries of $\widetilde \eta([\sigma_v,\infty))$, respectively. For general $v \in \overline{\mathbb B}$, we directly define $\eta^L(v)$ and $\eta^R(v)$ as the left and right boundaries of $\widetilde \eta([\sigma_v,\infty))$. Since $\widetilde \eta$ is a continuous space-filling curve, for each $n$ we can find $v_n \in B(v,1/n) \cap \mathbb Q^2$ such that $\sigma_{v_n}>\sigma_v$, and $\widetilde \eta([\sigma_v,\sigma_{v_n}]) \subset B(v,1/n)$. In this way, $\eta^L(v)$ and $\eta^R(v)$ agree with $\eta^L(v_n)$ and $\eta^R(v_n)$ respectively after the first time exiting $B(v,1/n)$. We also refer to $\eta^L(v)$ and $\eta^R(v)$ as the $\pi/2$ and $-\pi/2$ angle flow lines emanating from $v$. It is easy to verify that they are simple curves and satisfy the aforementioned flow line interaction rules. Moreover, the flow lines $\{\eta^{\mathsf q}(v):v \in B(z,r), \mathsf q \in \{L,R\}\}$ stopped upon exiting $B(z,r)$ are measurable with respect to $h|_{B(z,r)}$ modulo $2\pi\chi \mathbb Z$. In fact, this is clearly true for rational points. The assertion for general points follows from the fact that the segments of $\widetilde \eta$ in $B(z,r)$ and their orderings are determined by the flow lines emanating from rational points, which in turn are determined by the field. This implies Properties~\eqref{property:flow-line-all} and~\eqref{property:measurable}.

\subsection{Characterizing $\mathcal K$: Proof of Lemma~\ref{lem:flow-line-not-touch}}\label{subsec:flow-line-not-touch}

This subsection is devoted to the proof of Lemma~\ref{lem:flow-line-not-touch}, which relies on the following two properties:
\begin{enumerate}[(i)]
    \item \label{property:disk-mrt} Let $\kappa \in (4,8)$ and $\gamma=4/\sqrt{\kappa}$, and consider a weight $2-\frac{\gamma^2}{2}$ thin quantum disk with two boundary marked points~\cite{AHS20}, decorated with an independent chordal space-filling SLE$_\kappa$ $\widetilde \eta$ on top of it, parameterized by quantum area. Let $L_t$ (resp.\ $R_t$) be the quantum length of the left (resp.\ right) boundary of $\widetilde \eta([t,A])$, where $A$ is the random quantum area of the quantum disk. Then $(L_t,R_t)_{t \in [0,A]}$ evolves as a correlated two-dimensional Brownian excursion in the cone $\mathbb R_+^2$ (with starting point chosen from some random measure and ending at the origin), and the correlation coefficient is given by $-\cos(4\pi/\kappa)$.

    \item \label{property:ig-dense} Recall the setting in Section~\ref{subsec:setup-ig}. The following holds almost surely. Let $u \in \mathbb{B}$. The set of points $v \in \eta^L(u)$ such that $\eta^L(v)$ coincides with the part of $\eta^L(u)$ after hitting $v$ is dense on $\eta^L(u)$. The same holds for flow lines $\eta^R(\cdot)$.
\end{enumerate}

Property~\eqref{property:disk-mrt} is the finite-area variant of the mating-of-trees result from~\cite[Theorem 1.9]{DMS21-LQG-MRT}, proved in~\cite[Corollary 7.4]{AHS20}. We now prove Property~\eqref{property:ig-dense}.

\begin{proof}[Proof of Property~\eqref{property:ig-dense}]
    Recall the setting of Property~\eqref{property:disk-mrt}. A weight $2-\frac{\gamma^2}{2}$ thin quantum disk consists of a countable collection of thick quantum disks with two boundary marked points. We parameterize a chosen one by $(\mathbb{B}, -i, i)$ and adopt the notation of flow lines $(\eta^{\mathsf q}(u))$ to it for $u \in \overline{\mathbb{B}}$ and $\mathsf q \in \{L,R\}$. Assume that $v \in \eta^L(u)$ and $\eta^L(v)$ is not the part of $\eta^L(u)$ after hitting $v$. For $x \in \overline{\mathbb{B}}$, let $\sigma_x$ be the first time that $\widetilde \eta$ hits $x$. Since $\sigma_v < \sigma_u$, we know that $\eta^L(v)$ merges with $\eta^L(u)$ on its left side; see Figure~\ref{fig:flow-line} (left). Let $D_1$ be the bounded connected component of $\mathbb C \setminus (\eta^L(u) \cup \eta^L(v))$, and let $D_2$ be the union of bounded connected components of $\mathbb C \setminus (\eta^L(u) \cup \eta^R(u) \cup \eta^R(v))$. Since space-filling SLE$_\kappa$ for $\kappa \in (6,8)$ does not have quadruple points~\cite[Corollary 7.5]{GHM20-SLE-KPZ}, we know that $\widetilde \eta$ proceeds to fill $D_1$ and then $D_2$ after $\sigma_v$. Let $\sigma_v'$ be the time that $\widetilde \eta$ finishes filling $D_1$, then $R_t \ge R_{\sigma_v'}$ for $t \in [\sigma_v,\sigma_v']$ and $L_t \ge L_{\sigma_v'}$ for $t \in [\sigma_v',\sigma_u]$. Hence $\sigma_v'$ is a two-sided $\pi$-cone time of $(L,R)$. It follows from~\cite{Evans-bm-cone} (applied to a linear transformation of $(L,R)$ chosen so that the coordinates are independent) that the Hausdorff dimension of such times is zero. By local absolute continuity and the KPZ relation~\cite{GHM20-SLE-KPZ}, the set of these exceptional points $v$ has Hausdorff dimension zero and hence its complement is dense on $\eta^L(u)$. The case for $\eta^R(\cdot)$ follows from a similar argument.
\end{proof}

\begin{figure}[htbp]
    \centering
    \begin{minipage}{0.52\textwidth}
    \includegraphics[width=\textwidth]{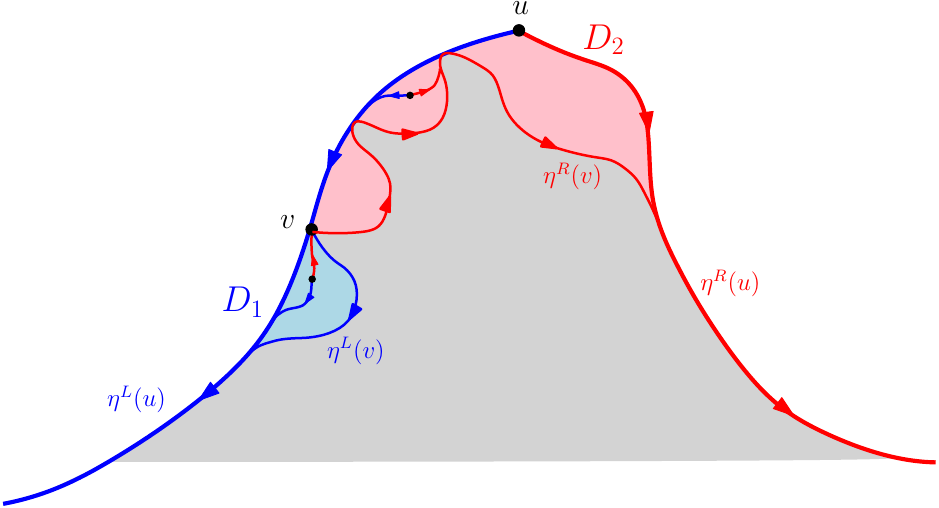}
    \end{minipage}
    \hfill
    \begin{minipage}{0.44\textwidth}
    \includegraphics[width=\textwidth]{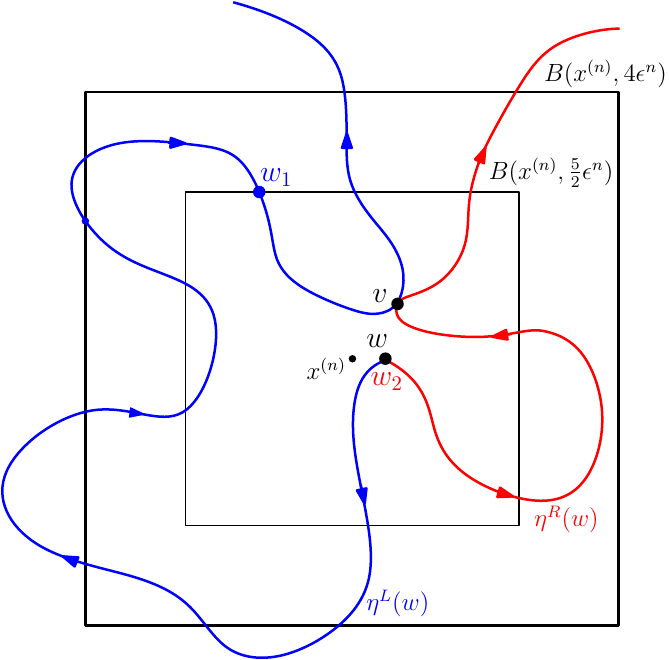}
    \end{minipage}
    \caption{\textbf{Left:} The exceptional point $v \in \eta^L(u)$ in Property~\eqref{property:ig-dense}. After hitting $v$ (and having filled the gray region), the curve $\widetilde \eta$ proceeds to fill $D_1$ (in light-blue) and then $D_2$ (in pink). \textbf{Right:} An illustration for the proof of Lemma~\ref{lem:flow-line-not-touch}. If $\eta^L(w)$ and $\eta^R(w)$ intersect at $v \neq w$, then we can construct $w_1,w_2 \in B(x^{(n)},3\epsilon^n)$ that violates the $n$-good property of $x^{(n)}$.}
    \label{fig:flow-line}
\end{figure}

In the following remark, we give a heuristic argument to derive the Hausdorff dimension of the set $\widehat{\mathcal{K}}$ defined below, again using Property~\eqref{property:disk-mrt} and the KPZ relation~\cite{GHM20-SLE-KPZ}.

\begin{remark}[Hausdorff dimension of $\mathcal{K}$]\label{rem:dimension-k}
    Since $\mathcal{K}$ has a positive probability of being empty, we consider the set $\widehat{\mathcal{K}}$ consisting of points $v \in \overline{\mathbb B}$ such that $\eta^L(v) \cap \eta^R(v)=\{v\}$. The Hausdorff dimension of $\widehat{\mathcal{K}}$ is then expected to be $4-16/\kappa$, which is close to 2 when $\kappa$ is close to 8. We present a heuristic argument based on Property~\eqref{property:disk-mrt}. Recall the setting in the proof of Property~\eqref{property:ig-dense}. Let $X$ be the pre-image under $\widetilde \eta$ of the point $v$ such that $\eta^L(v)$ and $\eta^R(v)$ do not intersect except at $v$. Then $X$ is the set of times that are not contained in any $\pi/2$-cone intervals for both $(L,R)$ and its time reversal.
    Let $X_1$ be the set of times that are not contained in any $\pi/2$-cone intervals for $(L,R)$, which is the pre-image under $\widetilde \eta$ of the point $v$ such that $\eta^L(v)$ does not hit $\eta^R(v)$ on its left side, then $\dim_{\mathcal H}(X_1)=\kappa/8$~\cite[Example 2.3]{GHM20-SLE-KPZ}. If $x$ is chosen proportionally to the Lebesgue measure in $[0,A]$, the probability that $X_1$ intersects $(x-\epsilon/2,x+\epsilon/2)$ decays as $\epsilon^{1-\kappa/8}$ as $\epsilon \to 0$.
    This suggests that the probability that $X$ intersects $(x-\epsilon/2,x+\epsilon/2)$ decays as $\epsilon^{2(1-\kappa/8)}$ as $\epsilon \to 0$, and hence $\dim_{\mathcal H}(X)=1-2(1-\kappa/8)=\kappa/4-1$. By local absolute continuity and the KPZ relation~\cite{GHM20-SLE-KPZ}, we have
    \[ \dim_{\mathcal H}(\widehat{\mathcal{K}})=(2+\tfrac{\gamma^2}{2})\dim_{\mathcal H}(X)-\tfrac{\gamma^2}{2}\dim_{\mathcal H}(X)^2=4-\tfrac{16}{\kappa}. \]
    Note that when the dimension of $\widehat{\mathcal{K}}$ is smaller than 1, i.e., $\kappa < 16/3$, $\mathcal{K}$ cannot contain a topological Sierpi\'nski carpet by~\cite[Proposition 3.5]{falconer-fractal-geometry}.
\end{remark}

Finally, we give the proof of Lemma~\ref{lem:flow-line-not-touch}.

\begin{proof}[Proof of Lemma~\ref{lem:flow-line-not-touch}]
    Fix $w \in [-1/2,1/2]^2$ so that the conditions hold. Assume for contradiction that $\eta^L(w)$ and $\eta^R(w)$ intersect at a point $v$ other than $w$; see Figure~\ref{fig:flow-line} (right). Since the origin is $0$-good, $\eta^L(w)$ and $\eta^R(w)$ do not intersect after exiting $B(w,\epsilon)$, hence $|w-v|_\infty \leq \epsilon$. Let $n$ be a positive integer such that $\epsilon^{n+1} < |w-v|_\infty \le \epsilon^n$, and let $x^{(n)}$ be the closest (under $|\cdot|_\infty$-distance) vertex to $w$ on $\epsilon^n \mathbb{Z}^2$. Then $|w-x^{(n)}|_\infty<\epsilon^n$, so $x^{(n)}$ is $n$-good by our assumption.

    We now construct a pair $(w_1,w_2)$ that violates the $n$-good property of $x^{(n)}$; see Figure~\ref{fig:flow-line} (right). Consider the flow line $\eta^L(w)$ stopped upon hitting $v$. If the flow line $\eta^L(w)$ hits $v$ before exiting $B(x^{(n)},4\epsilon^n)$, we simply let $w_1=w$. Otherwise, let $\sigma_1$ be the last time that $\eta^L(w)$ exits $B(x^{(n)},4\epsilon^n)$ before hitting $v$, and let $w_1$ be the point where $\eta^L(w)|_{[\sigma_1,\infty)}$ first hits $\overline{B(x^{(n)},\frac{5}{2}\epsilon^n)}$. Then $w_1 \in B(x^{(n)},3\epsilon^n)$, and
    \[ |w_1-v|_\infty \ge |w_1-x^{(n)}|_\infty-|x^{(n)}-w|_\infty-|w-v|_\infty \ge \tfrac{5}{2}\epsilon^n-\epsilon^n-\epsilon^n>\epsilon^{n+1}. \]
    By possibly switching to another point on $\eta^L(w)$ close to $w_1$ (see Property~\eqref{property:ig-dense}), we may assume that the flow line $\eta^L(w_1)$ coincides with the part of $\eta^L(w)$ after hitting $w_1$.
    Hence, in either case, the flow line $\eta^L(w_1)$ hits $v$ after exiting $B(w_1,\epsilon^{n+1})$ but before exiting $B(x^{(n)},4\epsilon^n)$. Similarly, we may define $w_2 \in B(x^{(n)},3\epsilon^n)$ such that $\eta^R(w_2)$ hits $v$ after exiting $B(w_2,\epsilon^{n+1})$ but before exiting $B(x^{(n)},4\epsilon^n)$. This pair $(w_1,w_2)$ then contradicts the $n$-good property of $x^{(n)}$, thereby concluding the proof.
\end{proof}

\subsection{Bounding $n$-good probability: Proof of Lemma~\ref{lem:kappa-to-8}}\label{subsec:n-good-probability}

In this subsection, we will prove Lemma~\ref{lem:kappa-to-8}. We start by recalling some terminology.

Let $\widetilde \eta$ be the chordal space-filling SLE$_\kappa$ in $\mathbb B$ from $-i$ to $i$, and when parameterized by capacity viewed from $i$, it is a chordal SLE$_\kappa$ process $\eta$ from $-i$ to $i$. By~\cite[Theorem 1.19]{IG4}, the time reversal of $\widetilde \eta$ is a chordal space-filling SLE$_\kappa(\frac{\kappa}{2}-4;\frac{\kappa}{2}-4)$ from $i$ to $-i$, and when parameterized by capacity viewed from $-i$, it is a chordal SLE$_\kappa(\frac{\kappa}{2}-4;\frac{\kappa}{2}-4)$ from $i$ to $-i$ with force points immediately to the left and right of $i$. Denote this curve by $\overline \eta$. From now on, we always assume that $\eta$, $\overline \eta$, and $\widetilde \eta$ are coupled in this way.

Recall that a bubble of $\eta$ is a connected component of $\mathbb B \setminus \eta([0,\infty])$, and the diameter of a set $A \subset \mathbb R^2$ is defined as $\mathrm{diam}(A)=\sup\{|x-y|_\infty:x,y\in A\}$. We first prove an intermediate result, stating that with high probability, chordal SLE$_\kappa$ does not have large bubbles in the bulk when $\kappa$ is close to 8.

\begin{lemma}\label{lem:bubble-diameter}
    For any $q>0$ and $\rho>0$, there exists $\delta_1=\delta_1(q,\rho)>0$ such that for all $\kappa \in (8-\delta_1,8)$, with probability at least $1-q$, all the bubbles of $\eta$ that intersect $(-3/4,3/4)^2$ have diameter less than $\rho$. The same holds for $\overline\eta$.
\end{lemma}

\begin{proof}
    For $\kappa \in (6,8]$, let $\widetilde \eta_k$ be the chordal space-filling SLE$_\kappa$ from $-i$ to $i$ in $\mathbb B$, where the subscript emphasizes the dependence on $\kappa$. The proof is divided into two steps. We first show that as $\kappa$ tends to 8, the curves $\widetilde \eta_\kappa$ converge in law to $\widetilde \eta_8$ under the curve topology when away from $\partial \mathbb B$; see Claim~\eqref{eq:kappa-to-8-convergence} below. Formally, we identify $\partial \mathbb B$ with a point so that $\mathbb B \cup \{\partial \mathbb B\}$ is homeomorphic to a sphere. Let $\widetilde d$ be the metric on $\overline{\mathbb{B}}$ induced by the round metric. Then, $\widetilde d$ is up-to-constant equivalent to $|\cdot|_\infty$-distance when restricted to any open subset $U \subset \mathbb B$ that has positive distance from $\partial \mathbb B$. For two curves $\gamma_1,\gamma_2$ parameterized by $[0,1]$, the distance between them is given by
    \[ \mathsf d(\gamma_1,\gamma_2)=\inf_\psi \sup_{t \in [0,1]} \widetilde d(\gamma_1(t),\gamma_2 \circ \psi(t)), \]
    where the infimum is taken over all increasing homeomorphisms $\psi$ from $[0,1]$ to itself. We claim that
    \begin{equation}\label{eq:kappa-to-8-convergence}
        \widetilde \eta_\kappa \to \widetilde \eta_8 \mbox{ in law with respect to } \mathsf d, \mbox{ as } \kappa \uparrow 8,
    \end{equation}
    which essentially follows from~\cite{AM-SLE8}. In the second step, combined with the fact that SLE$_8$ is a continuous space-filling curve, this will imply that neither $\eta$ nor $\overline \eta$ can form large bubbles in the bulk.
    
    \textbf{Step 1: Proof of Claim~\eqref{eq:kappa-to-8-convergence}}. We recommend that readers refer to Figure~\ref{fig:convergence} while reading the proof. For $\kappa \in (6,8)$, let $\eta'_\kappa$ be the whole-plane space-filling SLE$_\kappa$ normalized so that $\eta'_\kappa(0)=0$ and parameterized by Lebesgue measure, and let $\mathfrak h$ be a whole-plane GFF modulo $2\pi\chi \mathbb Z$. In what follows, a \emph{pocket} formed by curves (typically flow lines) refers to the bounded complementary connected component of their union. Fix $w \in [100,\infty)$ on the real line, let $\gamma^L(0)$ and $\gamma^R(0)$ (resp.\ $\gamma^L(w)$ and $\gamma^R(w)$) be deterministic simple curves starting from 0 (resp.\ $w$), such that $\gamma^L(0)$ and $\gamma^R(0)$ intersect and merge with $\gamma^L(w)$ and $\gamma^R(w)$ at $w+i$ and $w-i$, respectively, and the pocket formed by these four curves is contained in $B(0,w+10)$ and contains $B(w/2,w/4)$. Consider the $\pm \pi/2$ angle flow lines $\eta_\kappa^L(\cdot)$ and $\eta_\kappa^R(\cdot)$ of $\mathfrak h$, let $E_\kappa(w)$ be the event that
   \begin{align*}
       &\mbox{the flow line } \eta_\kappa^L(0) \mbox{ merges with } \eta_\kappa^L(w) \mbox{ inside } B(w+i,1/100), \mbox{the flow line } \eta_\kappa^R(0) \mbox{ merges with } \eta_\kappa^R(w) \\
       &\mbox{inside } B(w-i,1/100), \mbox{and each flow line stays within a distance } 1/100 \mbox{ from the deterministic curve.}
   \end{align*}
   By Lemmas 3.2 and 3.4 of~\cite{AM-SLE8}, there exists $p \in (0,1)$ depending only on the deterministic curves (and not on $\kappa$) such that $\mathbb P[E_\kappa(w)] \ge p$ for all $\kappa \in (6,8)$.
    
    Conditioned on the event $E_\kappa(w)$, let $U_\kappa$ be the connected component of $\mathbb C \setminus (\eta_\kappa^L(0) \cup \eta_\kappa^R(0) \cup \eta_\kappa^L(w) \cup \eta_\kappa^R(w))$ that contains $w/2$, and let $\tau_\kappa^1=\inf\{t \in \mathbb R: \eta'_\kappa(t) \in U_\kappa\}$ and $\tau_\kappa^2=\inf\{t>\tau_\kappa^1: \eta'_\kappa(t) \notin \overline U_\kappa\}$ be the first entering and exiting time of $\eta'_\kappa$ from $U_\kappa$. Let $x_\kappa^L \in B(w+i,1/100) $ and $x_\kappa^R \in B(w-i,1/100)$ be the point on $\partial U_\kappa$ where the flow lines merge. Let $f_\kappa:U_\kappa \to \mathbb B$ be the unique conformal map such that $f_\kappa(\eta'_\kappa(\tau_\kappa^1))=-i$, $f_\kappa(\eta'_\kappa(\tau_\kappa^2))=i$, and $f_\kappa(w/2) \in \mathbb R$. Then $f_\kappa(\eta'_\kappa|_{[\tau_\kappa^1,\tau_\kappa^2]})$ is a space-filling SLE$_\kappa(\frac{\kappa}{2}-4;\frac{\kappa}{2}-4)$ from $-i$ to $i$ in $\mathbb B$, with force points located at $f_\kappa(x_\kappa^L)$ and $f_\kappa(x_\kappa^R)$~\cite{IG4}. Moreover, by construction, there exists a constant $c(w)>0$ with $\lim_{w \to \infty} c(w)=0$, such that $f_\kappa(x_\kappa^L),f_\kappa(x_\kappa^R) \in \partial \mathbb B \cap B(i,c(w))$.

    \begin{figure}[htbp]
        \centering
        \includegraphics[width=0.92\textwidth]{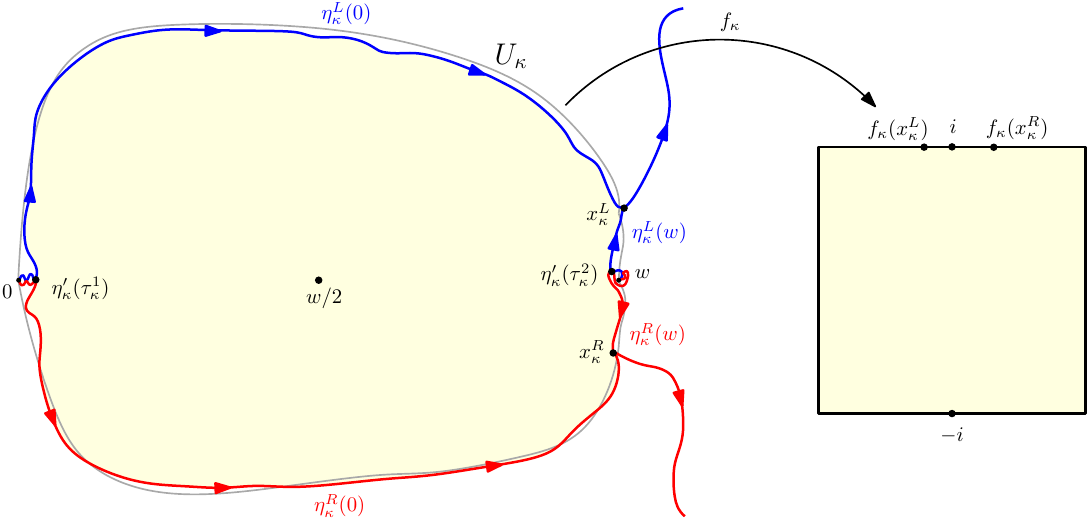}
        \caption{\textbf{Left.} On the event $E_\kappa(w)$, the $\pm \pi/2$ angle flow lines emanating from 0 and $w$ stay within small neighborhoods of the corresponding deterministic curves (in gray) and merge at $x_\kappa^L$ and $x_\kappa^R$ near the point $w+i$ and $w-i$, respectively. The whole-plane space-filling SLE$_\kappa$ restricted to $[\tau_\kappa^1,\tau_\kappa^2]$ is a space-filling SLE$_\kappa(\frac{\kappa}{2}-4;\frac{\kappa}{2}-4)$ process in the light-yellow domain $U_\kappa$, with force points at $x_\kappa^L$ and $x_\kappa^R$. \textbf{Right.} Pushing forward via the conformal map $f_\kappa$ yields a space-filling SLE$_\kappa(\frac{\kappa}{2}-4;\frac{\kappa}{2}-4)$ process in $\mathbb B$.}
        \label{fig:convergence}
    \end{figure}

    When parameterized by capacity viewed from $i$, the curve $f_\kappa(\eta'_\kappa|_{[\tau_\kappa^1,\tau_\kappa^2]})$ is a chordal SLE$_\kappa(\frac{\kappa}{2}-4;\frac{\kappa}{2}-4)$ from $-i$ to $i$ in $\mathbb B$. Let $h_\kappa^1$ (resp.\ $h_\kappa^2$) be the imaginary geometry GFF on $\mathbb B$ such that the counterflow line from $-i$ to $i$ corresponds to $f_\kappa(\eta'_\kappa|_{[\tau_\kappa^1,\tau_\kappa^2]})$ (resp.\ chordal SLE$_\kappa$). Note that the boundary conditions of $h_\kappa^1$ and $h_\kappa^2$ agree on $\partial \mathbb B \setminus B(i,c(w))$ and their difference on $\partial \mathbb B \cap B(i,c(w))$ tends to 0 as $\kappa$ tends to 8. Hence, for fixed $w$ there exists a coupling of $h_\kappa^1$ and $h_\kappa^2$ such that the Radon-Nikodym derivative of $h_\kappa^1|_{\mathbb B \setminus B(i,2c(w))}$ with respect to $h_\kappa^2|_{\mathbb B \setminus B(i,2c(w))}$ a.s. tends to 1 as $\kappa$ tends to 8.

    We now consider the family of curves $(\eta'_\kappa)$ conditioned on the event $E_\kappa(w)$. For any sequence $(\kappa_n)_{n \ge 1}$ in $(6,8)$ with $\kappa_n \uparrow 8$, by~\cite[Proposition 5.1]{AM-SLE8} and Arzel\`a-Ascoli theorem, there exists a weak subsequential limit of the laws of $(\eta'_{\kappa_n})$ on the space of curves on $\mathbb{C}$ endowed with local uniform distance. Passing to a further subsequence, we may assume that $(\eta'_{\kappa_n},f_{\kappa_n}^{-1},\tau_{\kappa_n}^1,\tau_{\kappa_n}^2)$ jointly converges in law to $(\eta',f^{-1},\tau^1,\tau^2)$, where $f_\kappa^{-1}$ is viewed as random variable taking values in the space of conformal maps defined on $\mathbb B$ equipped with the Carath\'eodory topology (this is the reason we use the distance $\mathsf d$ in~\eqref{eq:kappa-to-8-convergence}). By Skorokhod's representation theorem, we may assume that there exists a coupling under which the convergence holds almost surely along this subsequence. Since $f_\kappa$ is conformal and hence continuous in $\mathbb B$, we see that $f_{\kappa_n}(\eta'_{\kappa_n}|_{[\tau_{\kappa_n}^1,\tau_{\kappa_n}^2]})$ converges to $f(\eta')$ with respect to the distance $\mathsf d$.
    
    Recall that the law of chordal SLE$_\kappa$ weighted by the Radon-Nikodym derivative of $h_\kappa^1$ with respect to $h_\kappa^2$ yields the law of $f_\kappa(\eta'_\kappa|_{[\tau_\kappa^1,\tau_\kappa^2]})$ reparameterized by capacity. By~\cite[Proposition 4.43]{Lawler-conformal}, the law of chordal SLE$_\kappa$ before exiting $\mathbb B \setminus B(i,2c(w))$ converges in the Carath\'eodory sense to that of chordal SLE$_8$ as $\kappa$ tends to 8. This implies that the limiting curve $f(\eta')$ must be chordal SLE$_8$ before exiting $\mathbb B \setminus B(i,2c(w))$. Therefore, we conclude that the law of $f_\kappa(\eta'_\kappa|_{[\tau_\kappa^1,\tau_\kappa^2]})$ before exiting $\mathbb B \setminus B(i,2c(w))$ converges to that of SLE$_8$ with respect to $\mathsf d$ as $\kappa \uparrow 8$. By reweighting by the Radon-Nikodym derivative of $h_\kappa^2$ with respect to $h_\kappa^1$, we see that the chordal space-filling SLE$_\kappa$ also converges in law to SLE$_8$ as $\kappa$ tends to 8 with respect to $\mathsf d$ before exiting $\mathbb B \setminus B(i,2c(w))$. The conclusion follows since $w$ is arbitrary.
    
    \textbf{Step 2: Proof of the lemma using Claim~\eqref{eq:kappa-to-8-convergence}}. Let $\gamma$ be a chordal SLE$_8$ from $-i$ to $i$ in $\mathbb B$, parameterized by $[0,1]$ so that $\gamma(0)=-i$ and $\gamma(1)=i$.
    We denote by $\widetilde{\mathrm{diam}}$ the diameter associated with $\widetilde d$, i.e., $\widetilde{\mathrm{diam}}(A)=\sup\{ \widetilde d(x,y): x,y \in A\}$ for $A \subseteq \overline{\mathbb B}$. Since $\widetilde d$, when restricted to any open set $U \subset \mathbb B$ that has a positive distance from $\partial \mathbb B$, is up-to-constant equivalent to $d$, there exists a constant $c>0$ such that $\widetilde{\mathrm{diam}}(A) \ge c \cdot \mathrm{diam}(A) $ for any $A \subseteq \overline{\mathbb B}$ satisfying $A \cap (-3/4,3/4)^2 \neq \emptyset$.

    Fix $\rho \in (0,1/4)$. For $\lambda>0$, let $F(\lambda;\rho)$ be the event that there exist $t_1,t_2 \in [0,1]$ with $t_1<t_2$ such that
    \begin{equation}\label{eq:bottle-neck-1}
        \widetilde d(\gamma(t_1),\gamma(t_2)) \le 2\lambda \quad \mbox{and} \quad \inf_{\mathcal P:\gamma(t_1) \to \gamma(t_2)} \widetilde{\mathrm{diam}}(\mathcal P) \ge \rho/2,
    \end{equation}
    where the infimum is taken over all continuous (under the Euclidean metric) paths $\mathcal P$ from $\gamma(t_1)$ to $\gamma(t_2)$ that lie entirely in the interior of $\gamma([0,t_2])$ except possibly at the endpoints.
    Then we have
    \begin{equation}\label{eq:bottle-neck-2}
        \lim_{\lambda \to 0} \mathbb P[F(\lambda;\rho)]=0.
    \end{equation}
    In fact, suppose that~\eqref{eq:bottle-neck-2} does not hold. By monotonicity of $F(\lambda;\rho)$ in $\lambda$, the event $\cap_{n \ge 1} F(1/n;\rho)$ occurs with positive probability. On this event, for any $n \ge 1$ there exist $t_1^{(n)},t_2^{(n)} \in [0,1]$ such that~\eqref{eq:bottle-neck-1} holds for $(t_1^{(n)},t_2^{(n)})$ and $\lambda=1/n$. Let $(t_1,t_2)$ be a subsequential limit of $(t_1^{(n)},t_2^{(n)})_{n \ge 1}$, then since $\gamma$ is a continuous curve, we see that either $\gamma(t_1)=\gamma(t_2)$ or $\gamma(t_1),\gamma(t_2) \in \partial \mathbb B$, but $\inf_{\mathcal P:\gamma(t_1) \to \gamma(t_2)} \widetilde{\mathrm{diam}}(\mathcal P) \ge \rho/2$. The former case contradicts the fact that almost surely for all $t$, the set $\gamma([0,t])$ is homeomorphic to a closed disk, and the latter case contradicts the fact that almost surely for all $t$, the set $\mathbb B \setminus \gamma([0,t])$ is connected. This concludes~\eqref{eq:bottle-neck-2}.

    \begin{figure}[htbp]
        \centering
        \includegraphics[width=0.96\textwidth]{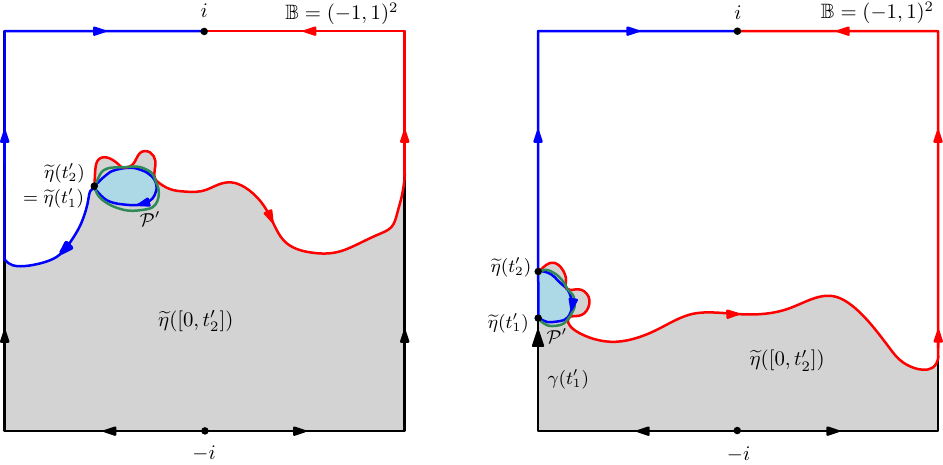}
        \caption{Two possible cases for how a bubble $b$ (in light-blue) of $\eta$ is formed by $\eta([t_1,t_2])$. The blue (resp.\ red) curve is the $\pi/2$ (resp.\ $-\pi/2$) angle flow line from $\widetilde \eta((t_2')^{-})$, while the light-gray regions are $\widetilde \eta([0,t_2'])$. The continuous path $\mathcal P'$ (in green) from $\widetilde \eta(t_1')$ to $\widetilde \eta(t_2')$ within $\widetilde \eta([0,t_2'])$ must enclose the bubble $b$, and thus necessarily has large diameter.
        \textbf{Left.} $b$ is in the interior so that $\eta(t_1)=\eta(t_2)$. \textbf{Right.} $b$ lies along the boundary so that $\eta(t_1),\eta(t_2) \in \partial \mathbb B$.}
        \label{fig:bubble}
    \end{figure}
    
    We choose $\lambda_0 \in (0,c\rho/4)$ so that $\mathbb P[F(\lambda_0;c\rho)] \leq q$. By Claim~\eqref{eq:kappa-to-8-convergence} and Skorokhod's representation theorem, we may assume that there is a coupling and a parameterization of curves such that $\sup_{t \in [0,1]} \widetilde d(\widetilde \eta_\kappa(t),\gamma(t)) \to 0$ almost surely as $\kappa$ tends to 8. Since both $\widetilde \eta_\kappa$ and $\gamma$ are continuous curves, we can choose $\delta_1$ so that for $\kappa \in (8-\delta_1,8)$ and we have $\widetilde d(\widetilde \eta_\kappa(t),\gamma(t)) \leq \lambda_0$ almost surely for all $t \in [0,1]$.

    Fix $\kappa \in (8-\delta_1,8)$ and omit the subscript $\kappa$ hereafter. Suppose that there exists a bubble $b$ of $\eta$ intersecting $(-3/4,3/4)^2$ and with $d$-diameter $\geq \rho$, let $t_1$ (resp.\ $t_2$) be the time at which $\eta$ starts (resp.\ finishes) tracing $\partial b$, then $b$ is disconnected from $i$ by $\eta([t_1,t_2])$ and we have $\widetilde d(\eta(t_1),\eta(t_2))=0$ (i.e., either $\eta(t_1)=\eta(t_2)$ or $\eta(t_1),\eta(t_2) \in \partial \mathbb B$, depending on whether $\partial b$ shares a non-trivial segment with $\partial \mathbb B$); see Figure~\ref{fig:bubble}. Let $t_1',t_2'$ be the corresponding times for $\widetilde \eta$ such that reparameterizing $\widetilde \eta([t_1',t_2'])$ yields $\eta([t_1,t_2])$. Specifically, $t_1'$ is the first time $\widetilde \eta$ hits $\eta(t_1)$ and $t_2'$ is the first time after $t_1'$ that $\widetilde \eta$ hits $\eta(t_2)$. Then, we have $\widetilde d(\widetilde \eta (t_1'),\widetilde \eta(t_2'))=0$. Moreover, any continuous path $\mathcal{P}'$ from $\widetilde \eta (t_1')$ to $\widetilde \eta (t_2')$ within $\widetilde \eta ([0, t_2'])$, together with $\partial \mathbb B$, must enclose the bubble $b$, which implies that $\mathcal{P}'$ intersects $(-3/4,3/4)^2$ and ${\rm diam}(\mathcal{P}') \geq \rho$ (and hence $\widetilde{{\rm diam}}(\mathcal{P}') \geq c \rho$). Using the fact that $\widetilde d(\widetilde \eta_\kappa(t),\gamma(t)) \leq \lambda_0$ for all $t \in [0,1]$, we see that $\widetilde d(\gamma(t_1'),\gamma(t_2')) \le 2\lambda_0$, and $\inf_{\mathcal P:\gamma(t_1') \to \gamma(t_2')} \widetilde{\mathrm{diam}}(\mathcal P) \ge c\rho - 2 \lambda_0 \geq c \rho/2$, where $\mathcal P$ is defined as in~\eqref{eq:bottle-neck-1}. By the choice of $\lambda_0$, the event $F(\lambda_0;c\rho)$ (and therefore the existence of such a bubble) occurs with probability at most $q$. This concludes the case for $\eta$.

    By Claim~\eqref{eq:kappa-to-8-convergence} and~\cite[Theorem 1.19]{IG4}, the time reversal of $\widetilde \eta$, which is a chordal space-filling SLE$_\kappa(\frac{\kappa}{2}-4;\frac{\kappa}{2}-4)$ from $i$ to $-i$ on $\mathbb B$, also converges in law to chordal SLE$_8$ under curve topology with respect to $\mathsf d$. The case for $\overline\eta$ then follows from the same argument as in Step 2.
\end{proof}

\begin{figure}[htbp]
    \centering
    \includegraphics[width=0.45\textwidth]{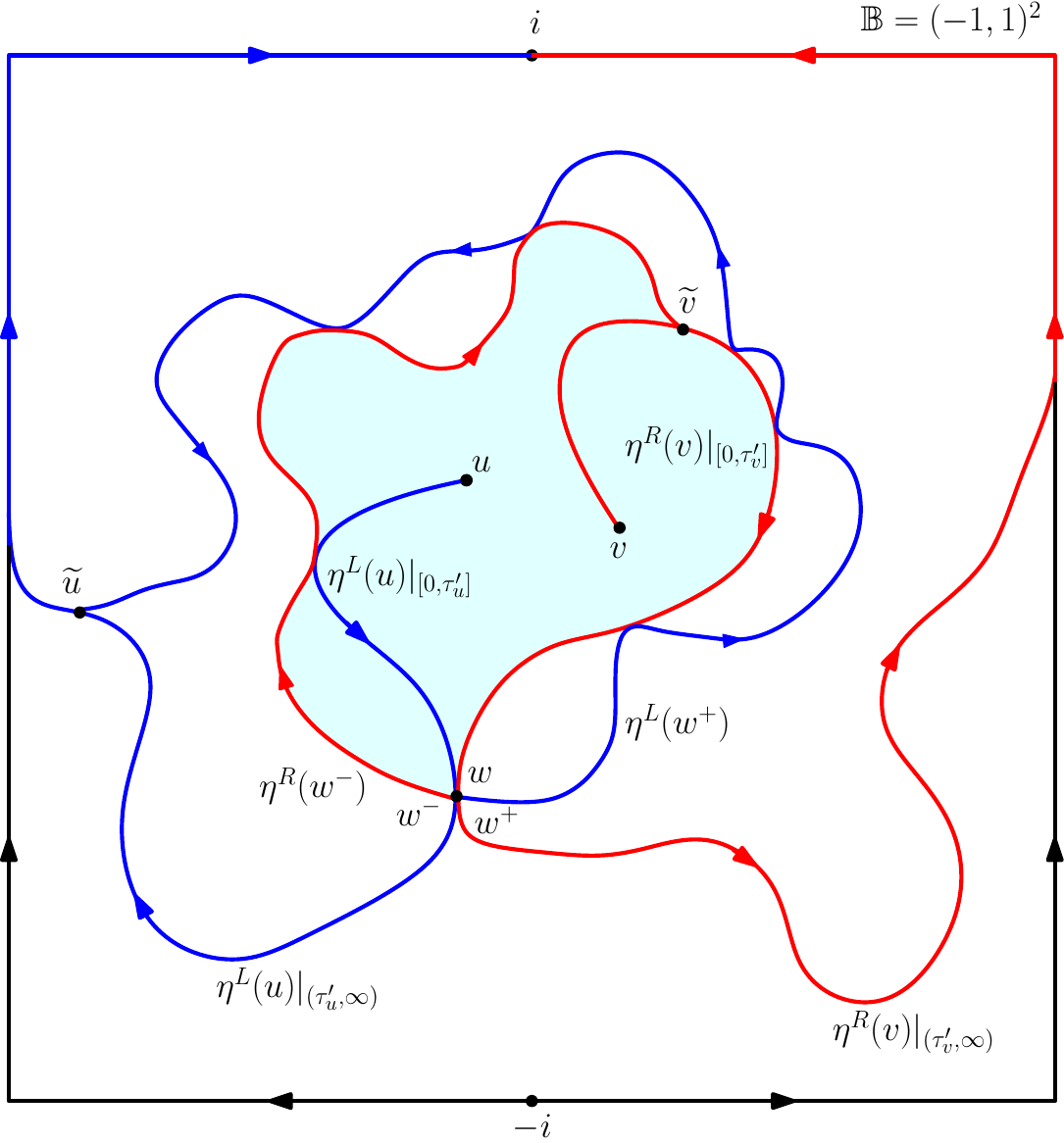}
    \caption{An illustration for the proof of Lemma~\ref{lem:kappa-to-8} for the case $n=0$. If $\eta^L(u)$ and $\eta^R(v)$, after leaving $B(u,\epsilon)$ and $B(v,\epsilon)$ respectively, intersect at some point $w$, then either $\eta^L(w^+)$ merges with $\eta^L(u)$ before $w$, or $\eta^R(w^-)$ merges with $\eta^R(v)$ before $w$. In either case, the pocket (in light-cyan) formed by two flow lines is contained in a bubble of $\eta$ or $\overline\eta$ with diameter at least $\epsilon$.}
    \label{fig:n-good}
\end{figure}

We now provide the proof of Lemma~\ref{lem:kappa-to-8} based on Lemma~\ref{lem:bubble-diameter}. Recall from Definition~\ref{def:n-good} the notion of $n$-good for $x \in V_n=\epsilon^n \mathbb Z^2 \cap (-3/4,3/4)^2$.

\begin{proof}[Proof of Lemma~\ref{lem:kappa-to-8}]
    Fix $q>0$ and $\epsilon<100^{-1}$. Let $\delta_1=\delta_1(q/2,\epsilon)$ be the constant in Lemma~\ref{lem:bubble-diameter}. We first show that for $\kappa \in (8-\delta_1,8)$, with probability at least $1-q$, the origin is $0$-good. By Lemma~\ref{lem:bubble-diameter}, it suffices to show that on the event that the origin is $0$-bad, then
    \begin{equation}\label{eq:find-bubble}
        \mbox{there exists a bubble of } \eta \mbox{ or } \overline \eta \mbox{ that intersects } (-3/4,3/4)^2 \mbox{ with diameter at least } \epsilon.
    \end{equation}
    Pick $u,v \in (-3/4,3/4)^2$ such that $\eta^L(u)|_{[\tau_u,\infty)}$ and $\eta^R(v)|_{[\tau_v,\infty)}$ intersect at some point $w$. We first assume that $\eta^L(u)|_{[\tau_u,\infty)}$ hits $\eta^R(v)|_{[\tau_v,\infty)}$ on its right side at $w$ at the stopping time $\tau_u'$; see Figure~\ref{fig:n-good}. In this case, we will show that there exists a bubble of $\overline\eta$ with diameter at least $\epsilon$. Let $w^+$ be the point immediately to the left of $\eta^R(v)$ at $w$; more precisely, $w^+$ is a prime end of $\mathbb B \setminus \eta^R(v)$.
    Similarly, let $\tau_v'$ be the stopping time that $\eta^R(v)|_{[\tau_v,\infty)}$ hits $\eta^L(u)|_{[\tau_u,\infty)}$ on its left side at $w$, and let $w^-$ be the point immediately to the right of $\eta^L(u)$ at $w$.
    We now consider two flow lines $\eta^L(w^+)$ and $\eta^R(w^-)$; see Figure~\ref{fig:n-good}. Let $\widetilde u$ be the point that $\eta^L(w^+)$ merges with $\eta^L(u)$, and let $\widetilde v$ be the point that $\eta^R(w^-)$ merges with $\eta^R(v)$. We claim that either $\widetilde u \in \eta^L(u)|_{[0,\tau_u']}$ or $\widetilde v \in \eta^R(v)|_{[0,\tau_v']}$. In fact, if $\widetilde u \in \eta^L(u)|_{(\tau_u',\infty)}$, then $w^-$ is a prime end of the pocket formed by $\eta^L(u)$ and $\eta^L(w^+)$. By flow line interaction rule, $\eta^R(w^-)$ cannot cross either $\eta^L(u)$ or $\eta^L(w^+)$, so it can only exit this pocket through $w$, i.e., the point $\widetilde v$ that it hits $\eta^R(v)$ lies on $\eta^R(v)|_{[0,\tau_v']}$. This proves the previous claim.

    Now, if $\widetilde v \in \eta^R(v)|_{[0,\tau_v']}$ (see Figure~\ref{fig:n-good}), then the pocket formed by $\eta^R(v)$ and $\eta^R(w^-)$ is contained in a bubble of $\overline \eta$, whose diameter is at least $\mathrm{diam}(\eta^L(u)|_{[0,\tau_u']}) \ge \mathrm{diam}(\eta^L(u)|_{[0,\tau_u]}) \ge \epsilon$. This is because when the time reversal of $\widetilde \eta$ hits $w^-$ for the first time, this pocket is enclosed and disconnected from $-i$ by $\eta^R(w^-)$. Similarly, if $\widetilde u \in \eta^L(u)|_{[0,\tau_u']}$, then the pocket formed by $\eta^L(u)$ and $\eta^L(w^+)$ is contained in a bubble of $\overline\eta$ with diameter at least $\epsilon$. In either case, there exists a bubble of $\overline\eta$ that intersects $(-3/4,3/4)^2$ with diameter at least $\epsilon$. Analogously, for the case when $\eta^L(u)|_{[\tau_u,\infty)}$ hits $\eta^R(v)|_{[\tau_v,\infty)}$ on its left side, there exists a bubble of $\eta$ that intersects $(-3/4,3/4)^2$ with diameter at least $\epsilon$. This proves~\eqref{eq:find-bubble} and the conclusion follows.

    We now use local absolute continuity (Lemma~\ref{lem:transfer}) to extend this result to all $n \ge 1$. Let $\mathbb{P}$ (resp.\ $\widetilde{\mathbb{P}}$) denote the law of the imaginary geometry GFF on $\mathbb B$ (resp.\ whole-plane GFF modulo $2\pi\chi \mathbb Z$). Let $E_0$ be the event that the following is true: For any $u,v \in (-3/8,3/8)^2$, the flow lines $\eta^L(u)$ and $\eta^R(v)$ do not intersect after first exiting $B(u,\epsilon/8)$ and $B(v,\epsilon/8)$, respectively, and before exiting $(-1/2,1/2)^2$. The same argument as for case $n=0$ shows that given any $p>0$, there exists $\delta'$ depending only on $p$ and $\epsilon$ such that $\mathbb P[E_0^c] \le p$ for all $\kappa \in (8-\delta',8)$.
    Applying Lemma~\ref{lem:transfer} with $B(z,r) = (-1/2,1/2)^2$, we find that given any $p>0$, there exists $\delta''$ depending only on $p$ and $\epsilon$ such that $\widetilde{\mathbb P}[E_0^c] \le p$ for all $\kappa \in (8-\delta'',8)$. Recall from Observation~\eqref{claim:remark-mn-3} in Remark~\ref{rmk:(m,n)-good} that the event that $x \in \epsilon^n \mathbb Z^2 \cap (-3/4,3/4)^2$ is $n$-good is measurable with respect to $h|_{B(x,4\epsilon^n)}$ modulo $2\pi\chi \mathbb Z$. Note that the event $E_0$ is defined in the same way as $\{ x \mbox{ is } n \mbox{-good} \}$ under appropriate scaling. Using the scaling invariance of the whole-plane GFF, we have $\widetilde{\mathbb P}[E_0^c] = \widetilde{\mathbb P}[x \mbox{ is } n \mbox{-bad}]$. Applying Lemma~\ref{lem:transfer} again with $B(z,r) = B(x, 4\epsilon^n)$, we find that there exists $\delta'''$ depending only on $q$ and $\epsilon$ such that $\mathbb P[x \mbox{ is } n \mbox{-bad}] \le q$ for all $\kappa \in (8-\delta,8)$, $n \geq 1$, and $x \in V_n$. Taking $\delta = \min \{ \delta_1, \delta''' \}$ yields the lemma.
\end{proof}

\section{Percolation on the Gaussian free field: Proof of Lemma~\ref{lem:induction-bound}}\label{sec:GFF}

In this section, we first derive Lemma~\ref{lem:induction-bound} from local absolute continuity of the Gaussian free field (GFF) and the following Proposition~\ref{prop:whole-GFF} about the whole-plane GFF. In Sections~\ref{subsec:outline-prop5.1}--\ref{subsec:count}, we prove Proposition~\ref{prop:whole-GFF}. For an integer $R \geq 1$, let $\mathscr{L}_R$ denote the lattice $[-R, R]^2 \cap \mathbb{Z}^2$.

\begin{proposition}\label{prop:whole-GFF}
    For any $T,\varrho\geq 1$, there exists a constant $c = c(T,\varrho)>0$ such that for all $s \in (0,T)$ and $\iota>0$, the following holds for all sufficiently large $R$. Let $\mathfrak h$ be a whole-plane Gaussian free field normalized so that $\mathfrak h_{2R}(0) = 0$, where $\mathfrak h_{2R}$ is the average of $\mathfrak h$ over $\partial B(0, 2R)$. For each $y \in \mathscr L_R$, let $\mathcal{E}_y$ be an event that is measurable with respect to $\mathfrak h|_{B(y, \varrho)}$ modulo $s \mathbb{Z}$ and satisfies
    \begin{equation}\label{eq:prop5.1-1}
    \mathbb{P}[\mathcal{E}_y] \geq 1 - c.
    \end{equation}
    Then, with probability at least $1-\iota$, 
    \begin{equation}\label{eq:prop5.1-2}
    \mbox{all connected components of } \cup_w \overline{B(w, 1)} \mbox{ have } |\cdot|_\infty \mbox{-diameter at most } \iota R,
    \end{equation}
    where the union is over those $w \in \mathscr{L}_R$ for which $\mathcal{E}_w$ does not occur, and two boxes that share exactly one point are also viewed as connected.
\end{proposition}

Proposition~\ref{prop:whole-GFF} can be viewed as percolation on the Gaussian free field: If we interpret the occurrence of the event $\mathcal E_y$ as the vertex $y$ being open, then Proposition~\ref{prop:whole-GFF} essentially states that if the probability of each vertex being open is overwhelmingly high, the overall configuration resembles a supercritical percolation model. We state this proposition in a general form as it may be of independent interest. For the proof of Lemma~\ref{lem:induction-bound}, we only require the special case where $s=2\pi\chi \in (\frac{2\pi}{\sqrt{6}},\frac{2\pi}{\sqrt{2}})$ and $\varrho=4$, together with Lemma~\ref{lem:whole-plane-absolute}.

\begin{proof}[Proof of Lemma~\ref{lem:induction-bound} assuming Proposition~\ref{prop:whole-GFF}]
    Recall the setup of Lemma~\ref{lem:induction-bound} and Definition~\ref{def:(m,n)-good}. Let $h$ be the imaginary geometry GFF on $\mathbb{B}$. Fix $p>0$ and $\epsilon>0$, and set $R = \lfloor \epsilon^{-1} \rfloor$. For a given $n \geq 0$ and $x \in V_{n} = \epsilon^{n}\mathbb{Z}^2 \cap (-3/4,3/4)^2$, define $\mathcal L_x := B(x, \epsilon^{n}) \cap V_{n+1}$. By assumption, for any $y \in \mathcal L_x$, we have
    $$
    \mathbb{P}[y \mbox{ is } (m-1,n+1) \mbox{-good}] \geq 1-p \quad \mbox{and} \quad \mathbb{P}[y \mbox{ is } (0,n+1) \mbox{-good}] \geq 1-p.
    $$
    This implies that
    \begin{equation}\label{eq:lem2.10-sec5-bound}
        \mathbb{P}[y \mbox{ is } (m-1,n+1) \mbox{-good and } (0,n+1) \mbox{-good}] \geq 1-2p.
    \end{equation}
    From now on, we will not use the specific definitions of $(m-1,n+1)$-good or $(0,n+1)$-good, but only the fact that they are measurable with respect to $h|_{B(y, 4 \epsilon^{n+1})}$ modulo $2 \pi \chi \mathbb{Z}$ (Observation~\eqref{claim:remark-mn-3} from Remark~\ref{rmk:(m,n)-good}). We now compare the law of
    \begin{equation}\label{eq:lem2.10-sec5-field}
        \widetilde h(z):=h(x + \epsilon^{n+1} z) \quad \mbox{and} \quad \mathfrak h (z) \quad \mbox{for } z \in B(0,3R/2).
    \end{equation}
    We perform the domain Markov decomposition in $B(0,15R/8)$. Let $\widetilde h'$ (resp.\ $\mathfrak h'$) be the harmonic extension of $\widetilde h|_{B(0,2R) \setminus B(0,15R/8)}$ (resp.\ $\mathfrak h|_{\mathbb{C} \setminus B(0,15R/8)}$) in $B(0,15R/8)$ and write $\mathring{\widetilde h} = \widetilde h - \widetilde h'$ (resp.\ $\mathring{\mathfrak h} = \mathfrak h - \mathfrak h'$). By Lemma~\ref{lem:gff-estimate-5.11}, as $M \to \infty$, with probability at least $1- o_M(1)$ uniformly in the other variables,
    $$
    \sup_{u,v \in B(0,7R/4)}| \widetilde h'(u) - \widetilde h'(v)| \leq M \quad \mbox{and} \quad \sup_{u,v \in B(0,7R/4)}| \mathfrak h'(u) - \mathfrak h'(v)| \leq M.
    $$
    On this event, Lemma~\ref{lem:whole-plane-absolute} allows us to compare the Radon-Nikodym derivatives of the two fields in~\eqref{eq:lem2.10-sec5-field}. Applying Proposition~\ref{prop:whole-GFF} with $s = 2 \pi \chi \in (\frac{2\pi}{\sqrt{6}},\frac{2\pi}{\sqrt{2}})$ and $\varrho = 4$ then yields Lemma~\ref{lem:induction-bound}.
    
    To be precise, we first take the constant $c_1$ in Lemma~\ref{lem:induction-bound} sufficiently small so that for all $p \in (0,c_1)$, the bound~\eqref{eq:lem2.10-sec5-bound}---combined with absolute continuity---ensures that~\eqref{eq:prop5.1-1} holds. Applying Proposition~\ref{prop:whole-GFF} with $\iota< \frac{1}{4}$ sufficiently small (depending only on $p$) and using absolute continuity again, Claim~\eqref{eq:prop5.1-2} implies that the event of $x$ being $(m,n)$-good (i.e., the $|\cdot|_\infty$-diameter of connected components is bounded above by $\iota R \epsilon^{n + 1} \leq \frac{1}{4} \epsilon^{n}$) holds with probability at least $1-p$, provided $\epsilon$ is sufficiently small (depending only on $p$), which completes the proof.
\end{proof}

\subsection{Proof outline of Proposition~\ref{prop:whole-GFF}}\label{subsec:outline-prop5.1}

In this section, we prove Proposition~\ref{prop:whole-GFF} modulo Lemmas~\ref{lem:gff-nice} and~\ref{lem:path}. These two lemmas are then proved in Sections~\ref{subsec:coarse} and~\ref{subsec:count}, respectively. Without loss of generality, we assume that $\varrho$ is a positive integer.

\subsubsection{Coarse graining of the GFF}

Recall that $\mathfrak h$ is a whole-plane GFF normalized so that $\mathfrak h_{2R} (0) = 0$. For $y \in \mathscr{L}_R$, let $\mathfrak h^{y, 2\varrho}$ be the harmonic extension of $\mathfrak h|_{\mathbb{C} \setminus B(y, 2\varrho)}$ on $B(y, 2\varrho)$. By Lemma~\ref{lem:whole-plane-markov}, $\mathfrak h - \mathfrak h^{y, 2\varrho}$ has the law of a zero-boundary GFF on $B(y,2\varrho)$, independent of $\mathfrak h|_{\mathbb{C} \setminus B(y,2\varrho)}$. Fix a constant $M>0$ to be chosen later. For $y \in \mathscr{L}_R$, define the event
\begin{equation}\label{eq:def-nice}
\mathcal{H}_y(M) := \Big\{ \sup_{u,v \in B(y, 3\varrho/2)} | \mathfrak h^{y, 2\varrho} (u) - \mathfrak h^{y, 2\varrho} (v)| \leq M \Big\}.
\end{equation}
Using local absolute continuity (Lemma~\ref{lem:whole-plane-absolute}), we obtain the following lemma.
\begin{lemma}\label{lem:sec5-pre}
    For any $\lambda>0$, $T,\varrho \ge 1$ and $M >0$, there exists a constant $c=c(\lambda,T,\varrho,M)>0$ such that the following holds for any $s \in (0,T)$, $y \in \mathscr{L}_R$ and any event $\mathcal{E}_y$ that is measurable with respect to $\mathfrak h|_{B(y, \varrho)}$ modulo $s \mathbb{Z}$ and satisfies $\mathbb{P}[\mathcal{E}_y] \geq 1 - c$. On the event $\mathcal{H}_y(M)$, conditioned on $\mathfrak h|_{\mathbb{C} \setminus B(y,2\varrho)}$, the event $\mathcal{E}_y$ occurs with probability at least $1 - \lambda$.
\end{lemma}
Define the set
\begin{equation}
    \mathscr{U} = \mathscr{U}(M) := \{y \in \mathscr{L}_R: \mathcal{H}_y(M) \mbox{ does not occur} \}.
\end{equation}
The following lemma controls the geometry of $\mathscr{U}$, whose proof will be postponed to Section~\ref{subsec:coarse}.

\begin{lemma}\label{lem:gff-nice}
    Fix $N \geq 100$. For all sufficiently large $M$ (which may depend on $N$ and $\varrho$), with probability $1 - o_R(1)$ as $R \to \infty$, the set $\mathscr{U}$ is $N$-separated in the following sense. Let $\mathsf{m} = \mathsf{m}(N, R, \varrho)$ be the largest integer $m$ such that $N^m \leq \varrho^{-1} 10^{-4} R$. There exists a sequence of vertex sets $\{\mathcal{A}_m\}_{0 \leq m \leq \mathsf m - 1}$ satisfying:
    \begin{enumerate}[i.]
        \item $\mathcal{A}_m \subset \varrho N^m \mathbb{Z}^2 \cap \mathscr{L}_R$ for $0 \leq m \leq \mathsf m - 1$; \label{lem5.4-condition1}
        \item For $0 \leq m \leq \mathsf m - 1$, each vertex in $\mathcal{A}_m$ has $|\cdot|_\infty$-distance at least $\frac{1}{10} \varrho N^{m+1}$ away from each other; \label{lem5.4-condition2}
        \item We have $\mathscr{U} \subset \cup_{m=0}^{\mathsf{m} - 1} \cup_{y \in \mathcal{A}_m} B(y, 10 \varrho N^m)$. \label{lem5.4-condition3}
    \end{enumerate}
\end{lemma}

Lemma~\ref{lem:sec5-pre} essentially states that on $\mathbb{Z}^2 \setminus \mathscr{U}$, the set of vertices where $\mathcal{E}_y$ does not occur can be stochastically dominated by a Bernoulli site percolation with arbitrary low intensity. In order to prove Proposition~\ref{prop:whole-GFF}, it suffices to enumerate the connected paths on $\mathbb{Z}^2 \setminus \mathscr{U}$, which will be carried out in Lemma~\ref{lem:path} below.

\subsubsection{Counting paths on a fractal graph}

Fix $N \geq 100$ and an arbitrary sequence of vertex sets $\{\mathcal{A}_m\}_{m \geq 0}$ on $\mathbb{Z}^2$ such that

\begin{enumerate}
    \item \label{condition:fractal-1} $\mathcal{A}_m \subset \varrho N^m \mathbb{Z}^2$ for $m \geq 0$;
    
    \item \label{condition:fractal-2} For $m \geq 0$, each vertex in $\mathcal{A}_m$ has $|\cdot|_\infty$-distance at least $\frac{1}{10} \varrho N^{m+1}$ away from each other.
\end{enumerate}
For two subsets $A,B \subset \mathbb{R}^2$, recall that $d(A,B)$ denotes the $|\cdot|_\infty$-distance between them.

\begin{definition}\label{def:*-connect}
    Given any vertex set $\{\mathcal{A}_m\}_{m \geq 0}$ satisfying Conditions~\ref{condition:fractal-1} and~\ref{condition:fractal-2}, let 
    \begin{equation}\label{eq:step-fractal}
        \mathcal{S}_0:= \{ B(x, 1) : x \in \mathbb{Z}^2\} \quad \mbox{and} \quad \mathcal{S}_m := \{ B(x , 10 \varrho N^{m-1}) : x \in \mathcal{A}_{m-1} \} \quad \mbox{for } m \geq 1.
    \end{equation}
    We consider \textbf{self-avoiding} paths $P=(x_1,x_2,\ldots,x_L)$ where each $x_i$ takes values in $\cup_{m=0}^\infty \mathcal{S}_m$. By self-avoiding we mean that $x_i \neq x_j$ for any $i \neq j$. The length of $P$ is denoted by ${\rm len}(P)=L$. The diameter of $P$ is referred to as the $|\cdot|_\infty$-diameter of $\cup_{i=1}^L x_i$ viewed as a subset of $\mathbb C$. We say that $P$ is \textbf{connected} if $d(x_i, x_{i+1}) \leq 40 \varrho$ for $1 \leq i \leq L-1$.
    
    For $k \geq 0$, we say that a self-avoiding connected path is at level $k$ if it consists only of boxes of the form $\cup_{m=0}^k \mathcal{S}_m$. In particular, a self-avoiding connected path at level $0$ consists only of boxes of the form $B(x,1)$.
\end{definition}

The following lemma, which will be proved in Section~\ref{subsec:count}, provides an upper-bound on the number of self-avoiding connected paths, and also states that each such path contains a uniformly positive fraction of $1$-boxes.

\begin{lemma}\label{lem:path}
    There exists a constant $C_1 = C_1(\varrho) \geq 100$ such that the following holds for all $N \geq C_1$ and any vertex sets $\{\mathcal{A}_m\}$ satisfying Conditions~\ref{condition:fractal-1} and~\ref{condition:fractal-2}, and $k \geq 0$.
    
    \begin{enumerate}
        \item Any self-avoiding connected path at level $k$ whose $|\cdot|_\infty$-diameter is at least $N^k$ has length at least $N^{k/2}$. \label{claim:lem-path-1}

        \item For all $L \geq N^{k/2}$, the number of self-avoiding connected paths at level $k$ with length $L$ and starting from a fixed box is at most $(C_1)^L$. Moreover, any such path consists of at least $1/2$ fraction of boxes from $\mathcal{S}_0$. \label{claim:lem-path-2}
    \end{enumerate}
    
\end{lemma}

\subsubsection{Concluding: Proof of Proposition~\ref{prop:whole-GFF}}

Fix $N = C_1$ as in Lemma~\ref{lem:path} and any $M = M(N,\varrho)$ sufficiently large so that the condition in Lemma~\ref{lem:gff-nice} holds. We assume that there exists a sequence of subsets $\{\mathcal{A}_m\}_{0 \leq m \leq \mathsf m - 1}$ satisfying Conditions~\ref{lem5.4-condition1}, \ref{lem5.4-condition2}, \ref{lem5.4-condition3} in Lemma~\ref{lem:gff-nice}, which occurs with probability $1 - o_R(1)$.

Suppose that $\cup_w \overline{B(w, 1)}$ in~\eqref{eq:prop5.1-2} contains a connected component with $|\cdot|_\infty$-diameter at least $\frac{1}{4} R \geq 2500 \varrho N^{\mathsf m}$. Then, there exists a sequence of vertices $(w_1, w_2, \ldots, w_L)$ in $\mathscr{L}_R$ such that the event $\mathcal{E}_{w_i}$ does not occur, and

\begin{equation}\label{eq:sec5.1-path-1}
    |w_i - w_{i+1}|_\infty \leq 2 \quad \mbox{for } 1 \leq i \leq L -1 \quad \mbox{and} \quad |w_1 - w_L|_\infty \geq \tfrac{1}{4} R - 2.
\end{equation}

\begin{figure}[htbp]
    \centering
    \includegraphics[width=0.9\textwidth]{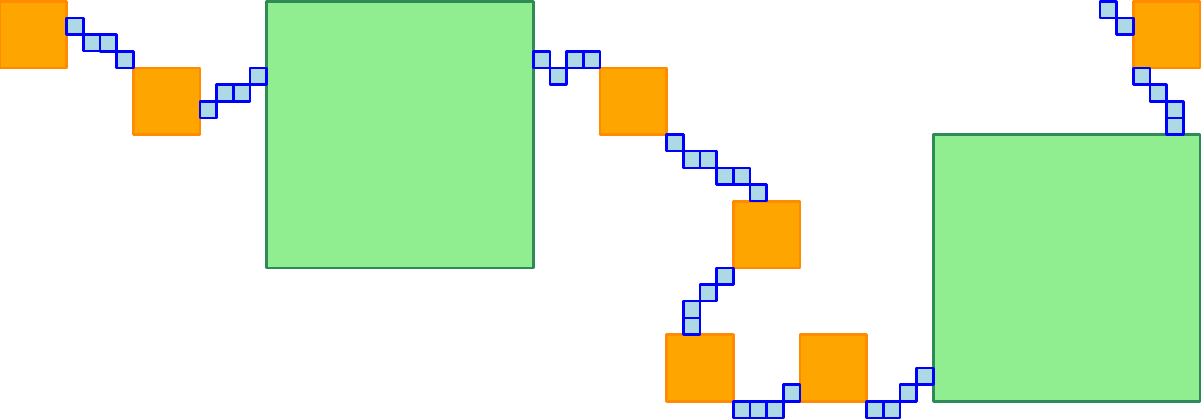}
    \caption{A schematic diagram of $(B_1,B_2,\ldots,B_L)$, which, after a loop erasure procedure, yields a self-avoiding connected path $(\widetilde B_1,\ldots,\widetilde B_\tau)$. Note that different boxes may have overlap; however, for illustrative purposes, all boxes in the figure are depicted as mutually disjoint.}
    \label{fig:connected-path}
\end{figure}

We now replace $(w_1, w_2, \ldots, w_L)$ with a self-avoiding connected path at level $\mathsf m$, as in Definition~\ref{def:*-connect}; see Figure~\ref{fig:connected-path}. For $1 \leq i \leq L$, if $w_i$ is contained in a box of the form $B(z, 10 \varrho N^m)$ for some $0 \leq m \leq \mathsf m-1$ and $z \in \mathcal{A}_m$, we replace $w_i$ with any such box, denoted by $B_i$. Otherwise, we set $B_i = B(w_i, 1)$. By Condition~\ref{lem5.4-condition3} in Lemma~\ref{lem:gff-nice}, the event $\mathcal{H}_{w_i}(M)$ occurs whenever $B_i = B(w_i, 1)$. Moreover, $d(B_i, B_{i+1})\leq 2$ for $1 \leq i \leq L-1$, and $d(B_1, B_L) \geq |w_1 - w_L|_\infty - 40 \varrho N^{\mathsf m - 1}$. We then perform a \emph{loop erasure} procedure on $(B_1, B_2,\ldots, B_L)$. Specifically, let $\widetilde B_1=B_1$ and, if we have already chosen $(\widetilde B_1,\ldots, \widetilde B_j)$, we define $\widetilde B_{j+1}$ to be the last element after $\widetilde B_j$ that is different from $\widetilde B_j$ and lies within $|\cdot|_\infty$-distance at most $40 \varrho$ from $\widetilde B_j$. Finally, we remove the last step in this loop-erased sequence. This defines a self-avoiding connected path\footnote{By construction, we always have $\widetilde B_i \neq \widetilde B_{i+1}$. Moreover, $d(\widetilde B_i, \widetilde B_j) > 40 \varrho$ whenever $|i-j| \geq 2$. Since $d(B_i, B_{i+1}) \leq 2$, the loop erasure procedure necessarily produces a connected path.} $(\widetilde B_1, \ldots, \widetilde B_\tau)$ in the sense of Definition~\ref{def:*-connect} where each $\widetilde B_i$ takes values in $\cup_{m=0}^{\mathsf m} \mathcal{S}_m$. It is easy to see that $d(\widetilde B_1, \widetilde B_\tau) \geq d(B_1, B_L) - 60 \varrho N^{\mathsf m - 1}$.\footnote{Note that $\widetilde B_1 = B_1$ and $\widetilde B_{\tau+1} = B_L$. Since $d(\widetilde B_\tau, \widetilde B_{\tau+1}) \leq 40 \varrho$, we deduce $d(\widetilde B_1, \widetilde B_\tau) \geq d(B_1, B_L) - 40 \varrho - 20 \varrho N^{\mathsf m - 1} \geq d(B_1, B_L) - 60 \varrho N^{\mathsf m - 1}$.}

We claim that all 1-boxes in $(\widetilde B_1, \ldots, \widetilde B_\tau)$ have $|\cdot|_\infty$-distance at least $20 \varrho$ from each other. By definition, $d(\widetilde B_i, \widetilde B_j) > 40 \varrho$ if $|i-j| \geq 2$. Thus, it suffices to consider the case where $\widetilde B_i, \widetilde B_{i+1}$ are two neighboring 1-boxes. Let $B$ be the step in $(B_1, \ldots, B_L)$ next to $\widetilde B_{i+1}$, which exists since $\widetilde B_{i+1}$ is not the last step. Then, we have $d(\widetilde B_i, B) > 40 \varrho$. Since $d(\widetilde B_{i+1}, B) \leq 2$, we conclude that $d(\widetilde B_i, \widetilde B_{i+1}) \geq 40 \varrho - 2 - 2 \geq 20 \varrho$.

Finally, we upper-bound the probability that there exists such a self-avoiding connected path at level $\mathsf m$ such that for any $1$-box $B(w_i,1)$ on this path, the event $\mathcal{H}_{w_i}(M)$ occurs while $\mathcal{E}_{w_i}$ does not occur. Let $\ell$ be the length of such path. Note that the diameter of the path is at least
$$
d(\widetilde B_1, \widetilde B_\tau) \geq d(B_1, B_L) - 60 \varrho N^{\mathsf m - 1} \geq |w_1 - w_L|_\infty - 100 \varrho N^{\mathsf m - 1} \geq N^{\mathsf m}.
$$
By Lemma~\ref{lem:path}, the length $\ell \geq N^{\mathsf m / 2}$ and the path contains at least $1/2$ fraction of $1$-boxes; moreover, for each fixed $\ell$, there are at most $(\mathsf m +1) \times |\mathscr{L}_R| \times (C_1)^{\ell}$ choices of path, where $(\mathsf m +1) \times |\mathscr{L}_R|$ counts the choice of the first box. In addition, for each fixed path, since all $1$-boxes $B(w_i,1)$ in it have mutual $|\cdot|_\infty$-distance at least $20 \varrho$ and the event $\mathcal{H}_{w_i}(M)$ occurs, applying Lemma~\ref{lem:sec5-pre} with $\lambda = (C_1)^{-4}$, there exists $c= c(T, \varrho, M)>0$ such that the probability that all $1$-boxes satisfy the event $\mathcal{E}_{w_i}$ is at most $\lambda^{\ell/2}$. Since $M$ is a fixed constant determined by $\varrho$, the constant $c$ depends only on $T$ and $\varrho$. Therefore, the probability of the existence of such path is at most
$$
\sum_{\ell \geq N^{\mathsf m/2}} \mathsf (\mathsf m +1) \times |\mathscr{L}_R| \times (C_1)^\ell \times \lambda^{\ell/2} \quad \mbox{with } \lambda = (C_1)^{-4}.
$$
It is easy to see that this summation is $o_{\mathsf m}(1)$, i.e., $o_R(1)$. This concludes the proposition. \qed

\subsection{Coarse-graining of the GFF: Proof of Lemma~\ref{lem:gff-nice}}\label{subsec:coarse}

We first define the harmonic extension decomposition of $\mathfrak h$. Recall that $\mathfrak h$ is a whole-plane GFF normalized so that $\mathfrak h_{2R}(0) = 0$ and it satisfies the domain Markov property stated in Lemma~\ref{lem:whole-plane-markov}.

\begin{definition}[Harmonic extension decomposition]\label{def:harmonic-extension}

For $x,y \in \mathbb{C}$ and $r,s>0$ such that $B(x,r) \subset B(y,s) \subset B(0, 2R)$, let $\mathfrak h^{y, s}$ be the harmonic extension of $\mathfrak h|_{\mathbb{C} \setminus B(y, s)}$ on $B(y, s)$, then $\mathring{\mathfrak h}^{y, s} = \mathfrak h - \mathfrak h^{y, s}$ is a zero-boundary GFF on $B(y, s)$ independent of $\mathfrak h|_{\mathbb{C} \setminus B(y, s)}$. We further define $\mathfrak h^{x, r}_{y,s}$ to be the harmonic extension of $\mathring{\mathfrak h}^{y, s}|_{B(y,s) \setminus B(x,r)}$ on $B(x, r)$.

The key property of this decomposition is that $\mathfrak h^{x, r}_{y, s}$ and $\mathfrak h^{x', r'}_{y', s'}$ are independent whenever $\overline{B(y,s)} \cap \overline{B(y',s')} = \emptyset$. This follows from the independence of $\mathring{\mathfrak h}^{y, s}$ and $\mathring{\mathfrak h}^{y', s'}$.

\end{definition}

Fix $N \geq 100$ and let the coarse-graining scale be $a_m = \varrho N^m$ for $m \geq 0$. For $R \geq 1$, recall that $\mathsf{m} = \mathsf{m}(N, R, \varrho)$ is the largest integer $m$ such that $N^m \leq \varrho^{-1} 10^{-4} R$. Then, we have 
\begin{equation}\label{eq:bound-N}
    \varrho^{-1} 10^{-4} N^{-1} R \leq N^{\mathsf m} \leq \varrho^{-1} 10^{-4} R.
\end{equation}
For each $z \in \mathscr{L}_R$ and $0 \leq m \leq \mathsf m-1$, let $z_m$ be the closest vertex (under $|\cdot|_\infty$-distance) to $z$ in the set $a_m \mathbb{Z}^2 \cap \mathscr{L}_R$. In particular, we have
\begin{equation}\label{eq:dist-zm}
    z_0 = z \quad \mbox{and} \quad |z - z_m|_\infty < a_m.
\end{equation}
This implies that 
\begin{equation}\label{eq:sec5.2-ball}
B(z,2\varrho) \subset B(z_1, 2a_1) \quad \mbox{and} \quad B(z_m, 4a_m) \subset B(z_{m+1}, 2a_{m+1}) \quad \mbox{for } 1 \leq m \leq \mathsf m-1.
\end{equation}
Using the notation in Definition~\ref{def:harmonic-extension}, we can decompose $\mathfrak h^{z,2\varrho}$ as follows:
\begin{equation}\label{eq:gff-decompose}
\mathfrak h^{z,2\varrho}(x) = \mathfrak h^{z,2\varrho}_{z_1, 4a_1}(x) + \sum_{m=1}^{\mathsf m-1} \mathfrak h^{z_m, 4a_m}_{z_{m+1}, 4a_{m+1}}(x) + \mathfrak h^{z_{\mathsf m}, 4 a_{\mathsf m}}(x) \quad \mbox{for } x \in B(z,2\varrho).
\end{equation}

To bound the fluctuation of $\mathfrak h^{z,2\varrho}$ on $B(z,3\varrho/2)$, it suffices to bound the fluctuation of each term on the right-hand side of~\eqref{eq:gff-decompose}. In particular, for $z \in \mathscr{L}_R$, the following conditions are sufficient:
\begin{enumerate}[(a)]
    \item $\sup_{u,v \in B(z,3\varrho/2)} |\mathfrak h^{z,2\varrho}_{z_1, 4a_1}(u) - \mathfrak h^{z,2\varrho}_{z_1, 4a_1}(v)| \leq M$. \label{sec5.2:condition-a}

    \item For $1 \leq m \leq \mathsf m - 1$, we have $$\sup_{u,v \in B(z_m, 3a_m)} |\mathfrak h^{z_m, 4a_m}_{z_{m+1}, 4a_{m+1}}(u) - \mathfrak h^{z_m, 4a_m}_{z_{m+1}, 4a_{m+1}}(v)| \leq N^{m/2} M.$$ \label{sec5.2:condition-b}

    \item $\sup_{u,v \in B(z_{\mathsf m}, 3 a_{\mathsf m})} |\mathfrak h^{z_{\mathsf m}, 4 a_{\mathsf m}}(u) - \mathfrak h^{z_{\mathsf m}, 4 a_{\mathsf m}}(v)| \leq N^{\mathsf m/2} M$. \label{sec5.2:condition-c}
\end{enumerate}

If Conditions~\eqref{sec5.2:condition-a},~\eqref{sec5.2:condition-b} and~\eqref{sec5.2:condition-c} are satisfied, then the event $\mathcal{H}_z(CM)$ occurs for some universal constant $C>0$.

\begin{lemma}\label{lem:nice-fluct}
    For $z \in \mathscr{L}_R$ that satisfies Conditions~\eqref{sec5.2:condition-a},~\eqref{sec5.2:condition-b} and~\eqref{sec5.2:condition-c}, we have 
    $$
    \sup_{u,v \in B(z, 3\varrho/2)} | \mathfrak h^{z, 2\varrho} (u) - \mathfrak h^{z, 2\varrho} (v)| \leq CM
    $$
    where $C>0$ is a universal constant.
\end{lemma}

\begin{proof}
Fix $z \in \mathscr{L}_R$ that satisfies the assumption. By Condition~\eqref{sec5.2:condition-b}, for all $1 \leq m \leq \mathsf{m} - 1$, we have
$$
\sup_{u,v \in B(z_m, 3a_m)} |\mathfrak h^{z_m, 4a_m}_{z_{m+1}, 4a_{m+1}}(u) - \mathfrak h^{z_m, 4a_m}_{z_{m+1}, 4a_{m+1}}(v)| \leq N^{m/2} M.
$$
There exists a universal constant $\widetilde C>0$ such that for any harmonic function $f$ on a simply connected domain $D$, we have
$$
|\nabla f(w)| \leq \widetilde C \cdot d(w, \partial D)^{-1} \sup_{u,v \in D} |f(u) - f(v)| \quad \mbox{for all } w \in D.
$$
Recall from~\eqref{eq:sec5.2-ball} that $B(z,3\varrho/2) \subset B(z_m, 2 a_m)$, and hence $d(B(z,3\varrho/2), \partial B(z_m, 3a_m)) \geq a_m$. Applying the above inequality with $f=\mathfrak h^{z_m, 4a_m}_{ z_{m+1}, 4a_{m+1}}$ and $D = B(z_m, 3a_m)$, we obtain that for all $1 \leq m \leq \mathsf{m} - 1$,
$$
\sup_{u \in B(z, 3\varrho/2)} |\nabla \mathfrak h^{z_m, 4a_m}_{z_{m+1}, 4a_{m+1}}(u)| \leq \widetilde C \cdot (a_m)^{-1} N^{m/2} M = \widetilde C \varrho^{-1} N^{-m/2}M .
$$
Similarly, we can deduce from Condition~\eqref{sec5.2:condition-c} that
$$
\sup_{u \in B(z, 3\varrho/2)} |\nabla \mathfrak h^{z_{\mathsf m}, 4a_{\mathsf m}}(u)| \leq \widetilde C \varrho^{-1} N^{-\mathsf m/2}M .
$$
Combining the above two inequalities with Condition~\eqref{sec5.2:condition-a} and~\eqref{eq:gff-decompose} yields
\begin{align*}
    \sup_{u,v \in B(z, 3\varrho/2)} | \mathfrak h^{z, 2\varrho} (u) - \mathfrak h^{z, 2\varrho} (v)| &\leq \sup_{u,v \in B(z, 3\varrho/2)} | \mathfrak h^{z, 2\varrho}_{z_1, 4a_1} (u) - \mathfrak h^{z,2\varrho}_{z_1, 4a_1} (v)| \\
    &+6\varrho \sum_{m=1}^{\mathsf m-1} \sup_{u \in B(z, 3\varrho/2)} |\nabla \mathfrak h^{z_m, 4a_m}_{z_{m+1}, 4a_{m+1}}(u)| + 6\varrho \sup_{u \in B(z, 3\varrho/2)} |\nabla \mathfrak h^{z_m, 4a_m}_{z_{m+1}, 4a_{m+1}}(u)|  \\
    &\leq M + 6 \varrho \sum_{m=1}^{\mathsf m-1} \widetilde C \varrho^{-1} N^{-m/2}M + 6 \varrho \cdot \widetilde C \varrho^{-1} N^{-\mathsf m/2} M \leq C M. \qedhere
\end{align*}
\end{proof}

Next, we use a coarse-graining argument to prove Lemma~\ref{lem:gff-nice}. For $i \geq 1$, a vertex $x \in a_i \mathbb{Z}^2 \cap \mathscr{L}_R$ is called an $a_i$-vertex. We first inductively define the notions of \textbf{nice $a_i$-vertices} and \textbf{bad $a_i$-vertices} for $1 \leq i \leq \mathsf{m}$. An important property, which will be clear from the following definition, is that the nice/bad statuses for two $a_i$-vertices are independent if their $|\cdot|_\infty$-distance is at least $9 a_i$ from each other (recall Definition~\ref{def:harmonic-extension}).

\begin{definition}\label{def:nice-boxes}
    \begin{enumerate}[(1)]
        \item \label{condition-1} A vertex $x \in a_1 \mathbb{Z}^2 \cap \mathscr{L}_R$ is called a \textbf{nice $a_1$-vertex} if for all $y \in \mathbb{Z}^2$ satisfying $B(y, 2\varrho) \subset B(x, 4a_1)$, we have
        $$
        \sup_{u,v \in B(y,3\varrho/2)} |\mathfrak h^{y,2\varrho}_{x,4a_1}(u) - \mathfrak h^{y,2\varrho}_{x,4a_1}(v)| \leq M.
        $$
        Otherwise, it is called a \textbf{bad $a_1$-vertex}.
        \item For $2 \leq j \leq \mathsf{m}$, given the definition of nice and bad $a_{j-1}$-vertices, we say $x \in a_j \mathbb{Z}^2 \cap \mathscr{L}_R$ is a \textbf{nice $a_j$-vertex} if the following condition~\eqref{condition-nice-a} holds and one of the conditions~\eqref{condition-nice-b} or~\eqref{condition-nice-c} holds:
        \begin{enumerate}[(i)]
            \item \label{condition-nice-a} For all $y \in a_{j-1} \mathbb{Z}^2$ satisfying $B(y, 4a_{j-1}) \subset B(x, 4a_j)$, we have
            $$
            \sup_{u,v \in B(y,3a_{j-1})} |\mathfrak h^{y,4a_{j-1}}_{x, 4a_j}(u) - \mathfrak h^{y,4a_{j-1}}_{x, 4a_j}(v)| \leq N^{(j-1)/2} M.
            $$
            \item \label{condition-nice-b} All $a_{j-1}$-vertices in $a_{j-1} \mathbb{Z}^2 \cap B(x, a_j)$ are nice $a_{j-1}$-vertices.
        
            \item \label{condition-nice-c} There are bad $a_{j-1}$-vertices contained in $a_{j-1} \mathbb{Z}^2 \cap B(x, a_j)$, but there exists $z \in a_{j-1} \mathbb{Z}^2 \cap \mathscr{L}_R$ such that for all bad $a_{j-1}$-vertices $y \in B(x, a_j)$, the boxes $B(y, 4a_{j-1})$ are contained in $B(z, 10a_{j-1})$.
        \end{enumerate}
        Otherwise, it is called a \textbf{bad $a_j$-vertex}.
    \end{enumerate}
\end{definition}

The following lemma follows from the estimates in Lemma~\ref{lem:gff-estimate-5.11}.

\begin{lemma}\label{lem:good}
    Fix $N \geq 100$. For all sufficiently large $M$ (which may depend on $N$ and $\varrho$), with probability $1 - o_R(1)$ as $R \to \infty$, all vertices in $a_{\mathsf m} \mathbb{Z}^2 \cap \mathscr{L}_R$ are nice $a_{\mathsf m}$-vertices.
\end{lemma}

\begin{proof}
    The constants $C$ in this proof may change from line to line but depend only on $\varrho$. For $1 \leq j \leq \mathsf m$, let $p_j = p_j(M, R, N, \varrho)$ be the maximum probability that a fixed $a_j$-vertex in $\mathscr{L}_R$ is $a_j$-bad. By Claim~\ref{lem5.9-claim-2} in Lemma~\ref{lem:gff-estimate-5.11}, we have
    $$
    p_1 \leq C a_1^2 e^{-M^2/C} \leq C N^2 e^{-M^2/C}.
    $$
    Next, we estimate $p_j$ in terms of $p_{j-1}$. By Claim~\ref{lem5.9-claim-2} in Lemma~\ref{lem:gff-estimate-5.11}, the probability that Condition~\eqref{condition-nice-a} fails is at most $C (a_j/ a_{j-1})^2 \exp(-N^{j-1} M^2/C)$. Moreover, if neither the Condition~\eqref{condition-nice-b} nor~\eqref{condition-nice-c} holds, then there exist $y,y' \in a_{j-1} \mathbb{Z}^2 \cap B(x,a_j)$ with $|y - y'|_\infty \geq 9 a_{j-1}$ such that $y$ and $y'$ are both bad $a_{j-1}$-vertices. By independence, the probability of this event is at most $C (a_j/ a_{j-1})^4 p_{j-1}^2$. Therefore, 
    $$
    p_j \leq C N^2 e^{-N^{j-1}M^2/C} + C N^4 p_{j-1}^2 \quad \mbox{for } 2 \leq j \leq \mathsf m.
    $$
    It follows that for all sufficiently large $M$ (which may depend on $N$ and $\varrho$), we have $p_j \leq N^{-100j}$ for all $1 \le j \le \mathsf m$. Recall from~\eqref{eq:bound-N} that there are at most $C(R/a_{\mathsf m})^2 \leq CN^2$ many $a_{\mathsf m}$-vertices in $\mathscr{L}_R$. The desired lemma follows from a union bound.
\end{proof}

\begin{lemma}\label{lem:cover}
    Suppose that all vertices $z \in a_{\mathsf m} \mathbb{Z}^2 \cap \mathscr{L}_R$ are nice $a_{\mathsf m}$-vertices. Then, we can find a sequence of vertex subsets $\{\mathcal{A}_m\}_{0 \leq m \leq \mathsf m - 1}$ satisfying:
    \begin{enumerate}[(I)]
        \item $\mathcal{A}_m \subset a_m \mathbb{Z}^2 \cap \mathscr{L}_R$ for $0 \leq m \leq \mathsf m - 1$;\label{lem5.12-condition-1}
        \item For $0 \leq m \leq \mathsf m - 1$, each vertex in $\mathcal{A}_m$ has $|\cdot|_\infty$-distance at least $\frac{1}{10} \varrho N^{m+1}$ away from each other;\label{lem5.12-condition-2}
        \item For any $z \in \mathscr{L}_R$ satisfying that $z \not\in \cup_{m=0}^{\mathsf{m} - 1} \cup_{y \in \mathcal{A}_m} B(y, 10 \varrho N^m)$, the vertex $z$ satisfies Conditions~\eqref{sec5.2:condition-a} and~\eqref{sec5.2:condition-b} before Lemma~\ref{lem:nice-fluct}.\label{lem5.12-condition-3}
    \end{enumerate}
\end{lemma}

\begin{proof}
    For $0 \leq m \leq \mathsf m - 1$, let $T_{m+1}$ be the set of all bad $a_{m+1}$-vertices in $a_{m+1} \mathbb{Z}^2 \cap \mathscr{L}_R$. Note that $T_{\mathsf m} = \emptyset$. For $0 \leq m \leq \mathsf m - 1$, we define $S_m$ as the set of bad $a_m$-vertices contained in $\mathscr{L}_R \setminus (\cup_{x \in T_{m+1}} B(x,a_{m+1}))$.
    We claim that there exists $\mathcal{A}_m$ satisfying Conditions~\eqref{lem5.12-condition-1} and~\eqref{lem5.12-condition-2} for this $m$, such that
    \begin{equation}\label{eq:lem5.10-cover}
        \cup_{y \in S_m} B(y,4a_m) \mbox{ is covered by } \cup_{z \in \mathcal A_m} B(z,10a_m).
    \end{equation} 
    This claim follows essentially from Conditions~\eqref{condition-nice-b} and~\eqref{condition-nice-c} in Definition~\ref{def:nice-boxes}. To see this, let $\mathcal A_m \subset a_m \mathbb Z^2 \cap \mathscr L_R$ be a subset of minimal cardinality satisfying~\eqref{eq:lem5.10-cover} (which clearly exists). Suppose for contradiction that Condition~\eqref{lem5.12-condition-2} fails, i.e., there exist $z_1,z_2 \in \mathcal A_m$ such that $|z_1-z_2|_\infty < \frac{1}{10} a_{m+1}$. For $i=1,2$, define $Y_i$ as the set of vertices $y \in S_m$ for which $B(y,4a_m) \cap B(z_i,10a_m) \neq \emptyset$. Then, for any $y_1 \in Y_1$ and $y_2 \in Y_2$, we have
    \[ |y_1-y_2|_\infty \le |y_1-z_1|_\infty+|z_1-z_2|_\infty+|z_2-y_2|_\infty \le 14a_m+\tfrac{1}{10} a_{m+1}+14a_m < \tfrac{1}{2} a_{m+1}. \]
    Hence, there exists $w \in a_{m+1} \mathbb Z^2 \cap \mathscr L_R$ such that $B(w,a_{m+1})$ covers $Y_1 \cup Y_2$. By definition of $S_m$, we have $w \not \in T_{m+1}$, and thus $w$ must be a nice $a_{m+1}$-vertex. Conditions~\eqref{condition-nice-b} and~\eqref{condition-nice-c} then imply that all boxes $B(y, 4 a_m)$ with $y \in Y_1 \cup Y_2$ are covered by $B(z,10 a_m)$ for some $z \in a_m \mathbb{Z}^2 \cap \mathscr{L}_R$. Replacing $\{z_1,z_2\}$ with this $z$ yields another $\mathcal A_m'$ satisfying~\eqref{eq:lem5.10-cover} with smaller cardinality, contradicting the minimality of $|\mathcal A_m|$.

    Finally, we verify that $\{\mathcal A_m\}_{0 \le m \le \mathsf m-1}$ satisfies Condition~\eqref{lem5.12-condition-3}. Let $z \notin \cup_{m=0}^{\mathsf{m} - 1} \cup_{y \in \mathcal{A}_m} B(y, 10 \varrho N^m)$. Recalling the definition of $z_m$, we see that $z_m$ is a nice $a_m$-vertex for all $1 \leq m \leq \mathsf m$. Condition~\eqref{sec5.2:condition-a} before Lemma~\ref{lem:nice-fluct} then follows immediately from Condition~\eqref{condition-1} in Definition~\ref{def:nice-boxes}, combined with~\eqref{eq:sec5.2-ball}. Similarly, Condition~\eqref{sec5.2:condition-b} before Lemma~\ref{lem:nice-fluct} follows from Condition~\eqref{condition-nice-a} in Definition~\ref{def:nice-boxes}.
\end{proof}

We are now ready to prove Lemma~\ref{lem:gff-nice}.

\begin{proof}[Proof of Lemma~\ref{lem:gff-nice}]
    First, by Claim~\ref{lem5.9-claim-1} in Lemma~\ref{lem:gff-estimate-5.11} and a union bound, Condition~\eqref{sec5.2:condition-c} (stated before Lemma~\ref{lem:nice-fluct}) holds simultaneously for all possible $z_{\mathsf m} \in a_{\mathsf m} \mathbb{Z}^2 \cap \mathscr{L}_R$ with probability $1 - o_R(1)$ as $R \to \infty$. Next, Lemma~\ref{lem:good} implies for all sufficiently large $M$ (which may depend on $N$ and $\varrho$), with probability $1-o_R(1)$, we can construct a sequence of vertex sets $\{\mathcal{A}_m\}_{0 \leq m \leq \mathsf m-1}$ as specified in Lemma~\ref{lem:cover}. Finally, Lemma~\ref{lem:nice-fluct} guarantees that this collection satisfies Conditions~\ref{lem5.4-condition1},~\ref{lem5.4-condition2},~\ref{lem5.4-condition3} as required in Lemma~\ref{lem:gff-nice}, after replacing $M$ with $CM$.
\end{proof}

\subsection{Counting paths in a fractal graph: Proof of Lemma~\ref{lem:path}}\label{subsec:count}

All constants $C$ in this subsection may vary from line to line, but depend only on $\varrho$. We assume that $N \geq 10^{10} \varrho $ and fix an arbitrary sequence of vertex sets $\{\mathcal{A}_m\}_{m \geq 0}$ on $\mathbb{Z}^2$ that satisfies Conditions~\ref{condition:fractal-1} and~\ref{condition:fractal-2}.

\begin{definition}[weight]
    Let $\beta>0$ be a constant to be chosen later. Recalling the definition of self-avoiding connected paths and the collection $\{\mathcal{S}_m\}_{m \geq 0}$ from Definition~\ref{def:*-connect}, we assign a weight $w(a) = \beta N^{-8m}$ to each box $a \in \mathcal{S}_m$ for $m \ge 0$. For any path $P$, its weight $w(P)$ is defined as the product of the weights of all boxes along the path.
\end{definition}

\begin{lemma}\label{lem:weight}
    There exists a constant $C = C(\varrho)>0$ such that for any $\beta \leq C^{-1}$, $N \geq C$, and any box $a \in \cup_{m=0}^\infty \mathcal{S}_m$, we have
    $$
    \sum_{P : a \to \bullet} w(P) \leq 1,
    $$
    where the sum is taken over all self-avoiding connected paths starting from $a$.
\end{lemma}

\begin{proof}
    For any $a \in \cup_{m=0}^\infty \mathcal{S}_m$, we first show that
    \begin{equation}\label{eq:weight<1}
        \sum_{ \substack{b \in \cup_{m=0}^\infty \mathcal S_m,\, d(a,b) \leq 40 \varrho}} \sqrt{w(a) w(b)} \leq 1/2.
    \end{equation}
    Assume that $a \in \mathcal{S}_l$. For each $0 \leq j \leq l$, there are at most $(200 \varrho N^l)^2$ possible choices of $b \in \mathcal{S}_j$ satisfying $d(a,b) \leq 40 \varrho$. On the other hand, for each $j > l$, Condition~\ref{condition:fractal-2} implies that there is at most one $b \in \mathcal{S}_j$ with $d(a,b) \leq 40 \varrho$. The left-hand side of~\eqref{eq:weight<1} is therefore bounded above by
    $$
    \beta \Big(\sum_{j=0}^l N^{-4l} \times N^{-4j} \times (200 \varrho N^l)^2 +  \sum_{j=l+1}^\infty N^{-4l} \times N^{-4j} \Big).
    $$
    For sufficiently small $\beta$ and sufficiently large $N$, inequality~\eqref{eq:weight<1} holds. The lemma follows from~\eqref{eq:weight<1} combined with the fact that $\sqrt{w(a)} \leq 1/2$.
\end{proof}

\begin{lemma}\label{lem:length-bound}
    There exists a constant $C = C(\varrho) \geq 100$ such that for all $N \geq 100 C$ and $k \geq 0$, all self-avoiding connected paths at level $k$ with $|\cdot|_\infty$-diameter $D \geq N^k$ have length at least $\frac{1}{C}(1 - \frac{C}{N})^kD$.
\end{lemma}

\begin{proof}
    We set $C = 10000 \varrho$ and prove the lemma by induction on $k$. The base case $k = 0$ holds trivially. For the inductive step, assume the statement holds for some $k \ge 0$. To prove the case $k+1$, consider any self-avoiding connected path $(x_1, x_2, \ldots, x_L)$ at level $(k+1)$ with diameter $D \geq N^{k+1}$. We divide it into sub-paths using the boxes from $\mathcal{S}_{k+1}$:
    $$ P_0, y_1, P_1, y_2, \ldots, y_M, P_M, $$
    where $y_1,\ldots,y_M \in \mathcal S_{k+1}$ and $P_0,\ldots,P_M$ are self-avoiding connected paths at level $k$ (with $P_0$ and $P_M$ possibly empty). Then,
    \begin{equation}\label{eq:lem5.13-1}
        {\rm diam}(P_0) + (2 \times 40\varrho + 20\varrho N^k) + {\rm diam}(P_1) + (2 \times 40\varrho + 20\varrho N^k) + \cdots +{\rm diam}(P_M) \geq D.
    \end{equation}
    We claim that for $N \geq 100 C$, 
    \begin{equation}\label{eq:lem5.13-2}
        \tfrac{C}{2N} ({\rm diam}(P_0) + {\rm diam}(P_1) + \cdots +  {\rm diam}(P_M)) \geq M (2 \times 40\varrho + 20\varrho N^k).
    \end{equation}
    We separate into two cases $M = 1$ and $M \geq 2$ (the case $M = 0$ is trivial). If $M \geq 2$, for each $1 \leq i \leq M-1$, since the $|\cdot|_\infty$-distance between the centers of $y_i$ and $y_{i+1}$ is at least $\frac{1}{10} \varrho N^{k+1}$ (Condition~\ref{condition:fractal-2}), we have ${\rm diam}(P_i) \geq \frac{1}{10} \varrho N^{k+1} - 20\varrho N^k - 2 \times 40 \varrho\geq \frac{1}{20} \varrho N^{k+1}$. This implies that
    $$
    \tfrac{C}{2N} ({\rm diam}(P_0) + {\rm diam}(P_1) + \cdots +  {\rm diam}(P_M)) \geq \tfrac{C}{2N} (M-1) \tfrac{1}{20} \varrho N^{k+1} \geq 100 M\varrho N^k \geq M(2 \times 40\varrho + 20\varrho N^k).
    $$
    If $M = 1$, we have $\min \{ {\rm diam}(P_0), {\rm diam}(P_M) \} \geq \frac{1}{3} N^{k+1}$, which leads to
    $$
    \tfrac{C}{2N} ({\rm diam}(P_0) + {\rm diam}(P_1) + \cdots +  {\rm diam}(P_M)) \geq \tfrac{C}{2N} \times \tfrac{1}{3} N^{k+1} \geq 2 \times 40\varrho + 20\varrho N^k.
    $$
    Combining both cases establishes Claim~\eqref{eq:lem5.13-2}.
    
    Let $A = \frac{1}{C} (1 - \frac{C}{N})^k$. For $1 \leq i \leq M-1$, recall that ${\rm diam}(P_i) \geq \frac{1}{20} \varrho N^{k+1} \geq N^k$. By the induction hypothesis, ${\rm len}(P_i) \geq A \cdot {\rm diam}(P_i)$ holds for all $1 \leq i \leq M - 1$. Similarly, we can also show that ${\rm len}(P_0) + A N^k \geq A \cdot {\rm diam}(P_0)$ and ${\rm len}(P_M) + A N^k \geq A \cdot {\rm diam}(P_M)$. Combining these estimates with~\eqref{eq:lem5.13-1} and~\eqref{eq:lem5.13-2} yields
    \begin{align*}
        AD &\leq A (1 + \tfrac{C}{2N}) ({\rm diam}(P_0) + {\rm diam}(P_1) + \cdots +  {\rm diam}(P_M)) \\
        &\leq (1 + \tfrac{C}{2N}) (2A N^k + {\rm len}(P_0) + {\rm len}(P_1) + \cdots +  {\rm len}(P_M)).
    \end{align*}
    Therefore, the total length $L$ satisfies
    \[ L \geq \sum_{i=0}^M {\rm len}(P_i) \geq (1 + \tfrac{C}{2N})^{-1} AD - 2A N^k. \]
    Since $C \geq 100$ and $N \geq 100C$, this yields $L \ge (1 - \frac{C}{N}) A D$ whenever $D \geq N^{k+1}$, thereby completing the induction.
\end{proof}

Next, we upper-bound the fraction of boxes from $\mathcal{S}_m$ in a self-avoiding connected path.

\begin{lemma}\label{lem:box-bound}
     There exists a constant $C = C(\varrho) \geq 100$ such that for all $N \geq 100 C$ and $k \geq m \geq 1$, the following holds. Every self-avoiding connected path at level $k$ with length $L \geq 1$ contains at most $2^{k-m} + \frac{C}{(N - C)^m} L$ boxes from $\mathcal{S}_m$.
\end{lemma}

\begin{proof}
    Let $\widetilde C$ be the constant from Lemma~\ref{lem:length-bound}. Fix $k \geq 1$ and consider a self-avoiding connected path $P$ at level $k$ with length $L \geq 1$. For $1 \leq i \leq k$, let $a_i$ be the number of boxes from $\mathcal{S}_i$ in $P$.
    
    For fixed $1 \leq m \leq k$, we divide $P$ into sub-paths using the boxes from $\mathcal{S}_m$:
    $$ P_0, y_1, P_1, y_2, \ldots, y_M, P_M, $$
    where $y_1,\ldots,y_M \in \mathcal{S}_m$. For each $1 \leq j \leq M-1$, if $P_j$ contains no box from $\cup_{i=m+1}^k \mathcal{S}_i$, then $P_j$ is a path at level $(m-1)$. For such path $P_j$, Condition~\ref{condition:fractal-2} implies that ${\rm diam}(P_j) \geq \frac{1}{10} \varrho N^m - 20  \varrho  N^{m-1} - 2 \times 40  \varrho  \geq \frac{1}{20}  \varrho  N^m$. Using Lemma~\ref{lem:length-bound}, we deduce ${\rm len}(P_j) \geq \frac{1}{\widetilde C} ( 1 - \frac{\widetilde C}{N})^{m-1} \frac{1}{20} \varrho N^m \geq \frac{1}{20 \widetilde C} ( N - \widetilde C)^m$. In addition, the number of paths $P_j$ that may contain boxes from $\cup_{i=m+1}^k \mathcal{S}_i$ is at most $\sum_{i = m+1}^k a_i$. Combining these, $L \ge \sum_{j=1}^{M-1} {\rm len}(P_j) \ge (M-1-\sum_{i=1}^{m+1}a_i) \cdot \tfrac{1}{20 \widetilde C} ( N - \widetilde C)^m$, which implies the recursive inequality
    \[ a_m = M \leq 1 + 20 \widetilde C ( N - \widetilde C)^{-m} L + \sum_{i=m+1}^k a_i. \]
    By induction, we derive
    \[ a_m \leq 2^{k-m} + \sum_{i=m}^k 20 \widetilde C ( N - \widetilde C)^{-i} 2^{i-m} L \leq 2^{k-m} + \frac{C}{(N-C)^m} L. \qedhere \]
\end{proof}

Finally, we present the proof of Lemma~\ref{lem:path}.

\begin{proof}[Proof of Lemma~\ref{lem:path}]
    Claim~\ref{claim:lem-path-1} follows directly from Lemma~\ref{lem:length-bound}, and the second assertion of Claim~\ref{claim:lem-path-2} is an immediate consequence of Lemma~\ref{lem:box-bound}. It suffices to prove the first assertion of Claim~\ref{claim:lem-path-2}. Suppose that $\beta^{-1}$ and $N$ are sufficiently large (depending on $\varrho$) so that Lemma~\ref{lem:weight} holds. We claim that there exists a constant $c = c(\beta)>0$ such that every self-avoiding connected path $P$ at level $k$ with length $L \geq N^{k/2}$ satisfies
    \begin{equation}\label{eq:weight-exponential}
        w(P) \geq c^L.
    \end{equation}
    Let $c = \beta u$ for some $u \in (0,1)$. For $0 \leq j \leq k$, let $a_j$ be the number of boxes from $\mathcal{S}_j$ in $P$. In fact, by Lemma~\ref{lem:box-bound},
    $$
    (\beta u)^{-L} w(P) = (\beta u)^{-L} \prod_{j=0}^k (\beta N^{-8j})^{a_j}  \geq u^{-L} \prod_{j=1}^k (N^{-8j})^{2^{k-j} + \frac{C}{(N - C)^j} L}.
    $$
    Note that $\prod_{j=1}^\infty (N^{-8j})^{\frac{C}{(N - C)^j}}$ is lower-bounded by a constant that may depend on $\varrho$, and 
    $$
    \prod_{j=1}^k (N^{-8j})^{2^{k-j}} \geq \exp (-o_N(1) N^{k/2})),
    $$
    so~\eqref{eq:weight-exponential} holds if $u$ is sufficiently small. The conclusion then follows from Lemma~\ref{lem:weight} and~\eqref{eq:weight-exponential}.
\end{proof}

\section{Finding the Sierpi\'nski carpet: Proof of Proposition~\ref{prop:find-carpet}}\label{sec:carpet}

Suppose that the origin is $(m,0)$-good for all $m \geq 0$. We will find a topological Sierpi\'nski carpet in $\mathcal{K}$ similar to~\eqref{eq:fractal-carpet-1} in Proposition~\ref{prop:carpet-fractal}. Recall the setting in Section~\ref{subsec:disconnectivity}.

For $w \in V_n$, we say that $w$ is \textbf{$(\infty, n)$-bad} if there exists an integer $m\geq 0$ (which may depend on $w$) such that $w$ is $(m,n)$-bad. Otherwise, we call it \textbf{$(\infty, n)$-good}. This is similar to the condition $\infty$-good in Proposition~\ref{prop:carpet-fractal}. Also observe that every $(0,n)$-bad vertex is automatically $(\infty, n)$-bad, and the origin is $(\infty, 0)$-good. 

Let $J_n$ denote the set of $(\infty, n)$-bad vertices in $V_n$. We claim that for $n \geq 0$, if $x \in V_n$ is $(\infty, n)$-good, then the $(\infty, n+1)$-bad vertices in $B(x,\epsilon^n)$ do not percolate in the sense in Definition~\ref{def:(m,n)-good}. Specifically, 
\begin{equation}\label{eq:sec6-1}
    \mbox{all connected components in } \cup_{w \in B(x,\epsilon^n) \cap J_{n+1}} \overline{B(w, \epsilon^{n+1})} \mbox{ have } |\cdot|_\infty \mbox{-diameter at most } \tfrac{1}{4} \epsilon^n.
\end{equation}
In fact, for each $m \geq 0$, since $x$ is $(m,n)$-good, by Definition~\ref{def:(m,n)-good}, the boxes in $\{ \overline{B(w, \epsilon^{n+1})} \}_{w \in B(x,\epsilon^n) \cap V_{n+1}}$ for which $w$ is either $(m-1, n+1)$-bad or $(0, n+1)$-bad do not form connected components with $|\cdot|_\infty$-diameter greater than $\frac{1}{4} \epsilon^n$. Recall from Observation~\eqref{claim:remark-mn-1} of Remark~\ref{rmk:(m,n)-good} that $(m,n)$-good condition for $m \ge 1$ becomes sharper as $m$ increases. Taking $m$ to infinity yields~\eqref{eq:sec6-1}.

Next, we follow the strategy of Proposition~\ref{prop:carpet-fractal} to find a topological Sierpi\'nski carpet in $\mathcal{K}$; see Figure~\ref{fig:fractal}. Let $\mathcal{K}_0 = [-1/2,1/2]^2$. We then define $\mathcal{K}_1$ by removing from $\mathcal{K}_0$ a collection of boxes as follows. We first remove from $\mathcal{K}_0$ all boxes of the form $\overline{B(x, \epsilon)}$ for which $x \in \epsilon \mathbb{Z}^2 \cap (-3/4,3/4)^2$ and $x$ is $(\infty, 1)$-bad. When two $\epsilon$-boxes intersect exactly at one point, we further remove from $\mathcal{K}_0$ the other two $\epsilon^2$-boxes incident to that point. Finally, we remove all remaining connected components that are surrounded by removed boxes and let $\mathcal{K}_1$ be the closure of the resulting set. The sets $\{\mathcal{K}_n\}_{n \geq 1}$ are defined analogously. Observe that this is a decreasing sequence of closed sets.

Let $\mathcal{K}_\infty = \cap_{n \geq 1} \mathcal{K}_n$. We will show that 
\begin{equation}\label{eq:sec6-carpet}
    \mathcal{K}_\infty \mbox{ is a topological Sierpi\'nski carpet.}
\end{equation}
Let $\{Q_i\}_{i \geq 1}$ be the complementary connected components of $\mathcal{K}_\infty$. We first observe that for any $k \ge 1$ and a complementary component $\mathcal Q^k$ of $\mathcal{K}_k$, if we denote by $\mathcal Q^{k+m}$ the complementary component of $\mathcal{K}_{k+m}$ that contains $\mathcal Q^k$, then $\cup_{m \ge 0} \mathcal Q^{k+m}$ is a complementary component of $\mathcal{K}_\infty$. Conversely, for each complementary component $Q_i$ of $\mathcal{K}_\infty$, there exists a positive integer $k$ such that $Q_i$ contains a complementary component of $\mathcal{K}_k$. The crucial observation for proving~\eqref{eq:sec6-carpet} is that the evolutions of two complementary components of $\mathcal{K}_k$ can never touch each other. In fact, if $\mathcal Q_1^k$ and $\mathcal Q_2^k$ are two complementary components of $\mathcal{K}_k$, then $d(\mathcal Q_1^k,\mathcal Q_2^k) \ge \epsilon^k(1-2 \epsilon )$ (recall that $\epsilon$ is a dyadic constant so $1/2$ is an integer multiple of $\epsilon^k$). This is because before possibly removing some $\epsilon^{k+1}$-boxes at the corner, the $|\cdot|_\infty$-distance between two different components is at least $\epsilon^k$, and such removal operation reduces the distance by no more than $2 \epsilon^{k+1}$.
    For $m \ge 0$, let $\mathcal Q_1^{k+m}$ (resp.\ $\mathcal Q_2^{k+m}$) be the complementary connected components of $\mathcal{K}_{k+m}$ that contains $\mathcal Q_1^k$ (resp.\ $\mathcal Q_2^k$). By~\eqref{eq:sec6-1}, for each $\epsilon^k$-box that intersects $\mathcal{K}_k$, the connected components of boxes centered at $(\infty, k+1)$-bad vertices in it have $|\cdot|_\infty$-diameter at most $\frac{1}{4}\epsilon^k$. This in particular implies that all the connected components of boxes centered at $(\infty, k+1)$-bad vertices contained in $\mathcal{K}_k$ have $|\cdot|_\infty$-diameter at most $\frac{1}{4}\epsilon^k$.
    Therefore, we have
    \[ d(\mathcal Q_1^{k+1},\mathcal Q_2^{k+1}) \ge d(\mathcal Q_1^k,\mathcal Q_2^k)-2 \times \tfrac{1}{4} \epsilon^k - 2 \times \epsilon^{k+2} \ge \epsilon^k (1- 2 \epsilon - \tfrac{1}{2} - 2\epsilon^2), \]
    where the additional $2 \times \epsilon^{k+2}$ comes from the potential removal of some $\epsilon^{k+2}$-boxes at the corner in the construction of $\mathcal{K}_{k+1}$. In general, we can prove inductively that
    \[ d(\mathcal Q_1^{k+m},\mathcal Q_2^{k+m}) \ge \epsilon^k(1-2 \epsilon - (\tfrac{1}{2} + 2\epsilon^2) \sum_{i=0}^{m-1} \epsilon^i), \quad \mbox{for all } m \ge 1. \]
    This implies that $d(\mathcal Q_1^{k+m},\mathcal Q_2^{k+m}) \ge \epsilon^k/3$ for all $m \ge 0$, so $\cup_{m \ge 0} \mathcal Q_1^{k+m}$ and $\cup_{m \ge 0} \mathcal Q_2^{k+m}$ have disjoint closures.
    Therefore, for each $i\neq j$, we can find a positive integer $k$ such that $Q_i$ (resp.\ $Q_j$) contains a \emph{unique} complementary connected component $\mathcal Q_i^k$ (resp.\ $\mathcal Q_j^k$) of $\mathcal{K}_k$, and hence $Q_i=\cup_{m \ge 0} \mathcal Q_i^{k+m}$ and $Q_j=\cup_{m \ge 0} \mathcal Q_j^{k+m}$ have disjoint closures. A similar argument shows that each $Q_i$ is a simply connected domain whose boundary is a simple continuous curve; hence, each $Q_i$ is a Jordan domain.
    Moreover, since $\eta$ has Hausdorff dimension $1+\kappa/8$~\cite{Beffara-dimension}, the set $\mathcal{K}_\infty$ has no interior points. By Whyburn's criterion~\cite{Whyburn-carpet}, the set $\mathcal{K}_\infty$ is a topological Sierpi\'nski carpet. Since $\mathcal{K} \supset \mathcal{K}_\infty$ (Lemma~\ref{lem:flow-line-not-touch}), we conclude the proof. \qed
    
\section{Conformal welding problem for SLE}\label{sec:welding}

In this section, we explain how Theorem~\ref{thm:carpet} implies that the conformal welding problem for SLE almost surely does not have a unique solution (Theorem~\ref{thm:welding}).

We first give a precise definition of conformal welding. Let $\mathbb D_1$ be the unit disk and $\mathbb D_2 = \widehat{\mathbb C} \setminus \overline{\mathbb D}_1$. Given an oriented Jordan loop $\eta$ on $\widehat{\mathbb C}$, let $D_1$ and $D_2$ denote the left and right connected components of $\widehat{\mathbb C} \setminus \eta$, respectively. For $i=1,2$, let $\psi_i$ be a conformal map from $\mathbb D_i$ to $D_i$. By Carath\'eodory's theorem, $\psi_i$ extends continuously to the boundary of $\mathbb D_i$, hence $\phi:=\psi_2^{-1} \circ \psi_1:\partial \mathbb D_1 \to \partial \mathbb D_2$ defines an orientation-preserving homeomorphism from $\mathbb S^1$ to itself. We refer to $\phi$ as a \emph{conformal welding homeomorphism} for $D_1$ and $D_2$ with interface $\eta$. Note that $\phi$ is only identified up to pre- and post-composition with automorphism of $\mathbb D$. Moreover, the image of $\eta$ under any M\"obius transformation would yield the same welding homeomorphism as $\eta$ itself. Therefore, the solution to the conformal welding problem is characterized by the map
\[ \mathcal W: \mathcal L/{\sim_{\text{M\"ob}}} \to \mathrm{Homeo}^+(\mathbb S^1)/{\sim_{\mathrm{Aut}}}, \quad [\eta] \mapsto [\phi], \]
where $\mathcal L/{\sim_{\text{M\"ob}}}$ is the collection of oriented Jordan loops on $\widehat{\mathbb C}$ under the equivalence relation that $\eta_1 \sim_{\text{M\"ob}} \eta_2$ if $\eta_2=\omega(\eta_1)$ for some M\"obius transformation $\omega$, and $\mathrm{Homeo}^+(\mathbb S^1)/{\sim_{\mathrm{Aut}}}$ is the set of orientation-preserving homeomorphism from $\mathbb S^1$ to itself under the equivalence relation that $\phi_1 \sim_{\mathrm{Aut}} \phi_2$ if $\phi_1=\rho_1 \circ \phi_2 \circ \rho_2$ for $\rho_i \in \mathrm{Aut}(\mathbb D_i)$. Note that $\mathcal{W}$ is neither surjective nor injective; see, e.g.,~\cite{Bishop-welding}.

Conformal removability is related to the uniqueness of conformal welding, i.e., injectivity of $\mathcal W$. It is known that conformal removability of $\eta$ implies the uniqueness of conformal welding (see, e.g.,~\cite[Corollary 5.24]{Younsi-survey}). However, it remains an open question whether the converse is true, i.e., whether conformal non-removability guarantees the non-uniqueness of conformal welding. Indeed, if $\eta$ is conformally non-removable, then there exists a homeomorphism $f:\widehat{\mathbb C} \to \widehat{\mathbb C}$ which is conformal on $\widehat{\mathbb C} \setminus \eta$ but not on all of $\widehat{\mathbb C}$. Then, pushing forward $\psi_i$ by $f$ would yield a new solution to conformal welding with interface $f(\eta)$. However, it is still possible that $f(\eta)=\omega(\eta)$ for some M\"obius transformation $\omega$; see~\cite{younsi-counterexample}.

Pioneered by Sheffield's work~\cite{Sheffield-zipper}, it was shown in~\cite{AHS-loop, Baverez-welding} that SLE$_\kappa$ loop with $\kappa \in (0,4]$ arises from conformal welding. Specifically, if one endows $\mathbb{D}_1$ and $\mathbb{D}_2$ with $\gamma$-LQG ($\gamma=\sqrt{\kappa}$) quantum disks and defines the homeomorphism $\phi$ by identifying the LQG quantum boundary length on $\mathbb S^1$, then a conformal welding exists and the welding interface has the same law as SLE$_\kappa$ loop measure. On the other hand, using the conformal removability of SLE$_\kappa$~\cite{JS-removable, Rohde-Schramm-basic, KMS-removability-SLE4}, we have that the conformal welding is almost surely unique for these welding homeomorphisms, i.e., the pre-image of $\phi$ under $\mathcal W$ is almost surely unique. Later, it was shown in~\cite{DMS21-LQG-MRT} that SLE$_\kappa$ for $\kappa \in (4,8)$ also arises from conformal welding. Since SLE$_\kappa$ for $\kappa \in (4,8)$ is a random non-simple curve that touches itself, we need to first extend the definition of conformal welding to beaded domains as defined below.

We call a connected closed set $A \subset \widehat{\mathbb C}$ a \emph{beaded domain} if $\widehat{\mathbb C} \setminus A$ is simply connected and there exists a continuous non-self-crossing loop $\mathcal{L}$ such that $A$ is the union of $\mathcal{L}$ and the set of points enclosed by $\mathcal{L}$. Moreover, each connected component of the interior of $A$, together with its prime-end boundary, is homeomorphic to a closed disk. For two beaded domains $B_1$ and $B_2$, we parameterize their corresponding loops by $\mathbb S^1$ and define a homeomorphism $\phi: \mathbb{S}^1 \to \mathbb{S}^1$. A (non-simple) conformal welding with welding homeomorphism $\phi$ consists of a loop $\eta$ on $\widehat{\mathbb{C}}$ together with two homeomorphisms $\psi_i$ from the left and right connected components of $\widehat{\mathbb C} \setminus \eta$, together with the prime-end boundary, to $B_i$ such that $\psi_i$ is conformal on each connected component of its interior and, moreover, $\phi = \psi_2 \circ \psi_1^{-1}$.

It was shown in~\cite{ACSW-loop} based on~\cite{DMS21-LQG-MRT} that SLE$_\kappa$ loop for $\kappa \in (4,8)$ arises from conformal welding for beaded domains. Specifically, define two random beaded domains using the stable looptree~\cite{Nicolas-looptree} obtained by $\alpha$-stable L\'evy process with $\alpha = \kappa/4$ and replace the interior of each loop with conditionally independent $\gamma$-LQG ($\gamma = \sqrt{16/\kappa}$) quantum disks. If one defines a homeomorphism by identifying the generalized quantum boundary length on these two beaded domains, then a conformal welding exists and the welding interface has the same law as SLE$_\kappa$ loop measure. 

When $\eta$ is an SLE$_\kappa$-type curve for $\kappa \in (8 - \delta_0,8)$, it is enough to deduce from Theorem~\ref{thm:carpet} the non-uniqueness of conformal welding. This is based on an explicit construction of $f$ from~\cite{Ntalampekos-carpet}.

\begin{proof}[Proof of Theorem~\ref{thm:welding}]
    It suffices to show that for $\kappa \in (8 - \delta_0, 8)$ and an SLE$_\kappa$ loop $\eta$, there almost surely exists a homeomorphism $f:\widehat{\mathbb C} \to \widehat{\mathbb C}$ which is conformal on $\widehat{\mathbb C} \setminus \eta$ but not on all of $\widehat{\mathbb C}$, and $f(\eta) \neq \omega(\eta)$ for all M\"obius transformation $\omega$. By Theorem~\ref{thm:carpet}, $\eta$ almost surely contains a topological Sierpi\'nski carpet. Let $S$ be such a carpet contained in $\eta$, and let $S'$ be any topological Sierpi\'nski carpet in $[0,1]^2$ with positive Lebesgue measure\footnote{For instance, in the first step, we divide $[0,1]^2$ into 9 congruent squares and remove the central one. In the $k$-th step ($k \ge 1$), we divide each remaining square into $(2k+1)^2$ congruent squares and remove the central one. The final set is a topological Sierpi\'nski carpet with area $\prod_{k=1}^\infty (1-1/(2k+1)^2)=\pi/4$.}. By~\cite[Theorem 1.3]{Ntalampekos-carpet}, there exists a homeomorphism $f:\widehat{\mathbb{C}} \to \widehat{\mathbb{C}}$ such that $f(S) = S'$ and $f$ is conformal on $\widehat{\mathbb{C}} \setminus S$. Since $\eta \supset S$, we have $f(\eta) \supset f(S) = S'$ and $f$ is conformal on $\widehat{\mathbb C} \setminus \eta$. Since $\eta$ almost surely has Lebesgue measure zero and the image of a set of Lebesgue measure zero under M\"obius transformations still has Lebesgue measure zero, it follows that $f(\eta) \neq \omega(\eta)$ for all M\"obius transformation $\omega$. This also implies that $f$ is not a M\"obius transformation, and hence not conformal on all of $\widehat{\mathbb{C}}$.
\end{proof}

\section{Open problems}\label{sec:open}

In this section, we discuss some open problems.

\subsubsection*{Fractal percolation}

We begin with some problems in fractal percolation that we find fundamental but to the best of our knowledge are not well understood. Recall from Proposition~\ref{prop:super-sierpinski} that we showed that when $p > p_c$, the final retained set $A_\infty$ in fractal percolation contains a topological Sierpi\'nski carpet. We conjecture that this also holds at the critical value $p = p_c$.

\begin{prob}\label{prob:critical-TSC}
    Prove that when $p = p_c$, the final retained set in fractal percolation contains a topological Sierpi\'nski carpet with positive probability.
\end{prob}

A topological Sierpi\'nski carpet can be found if condition~\eqref{property:strong-disconnect} holds for $p \ge p_c$. In that case, removing the fillings of outermost removed box clusters yields the largest possible topological Sierpi\'nski carpet contained in $A_\infty$. In Proposition~\ref{prop:super-sierpinski}, we were only able to show that \eqref{property:strong-disconnect} holds for all $p \in (0,1)$ except for a countable set. We conjecture that it actually holds for all $p \in (0,1)$.

\begin{prob}\label{prob:fractal-strong-disconnect}
    Prove that for any $p \in (0,1)$, a.s., the outer boundaries of any two different removed box clusters do not intersect.
\end{prob}

If Problem~\ref{prob:fractal-strong-disconnect} can be solved at $p = p_c$, the argument in Proposition~\ref{prop:super-sierpinski} would yield a positive answer to Problem~\ref{prob:critical-TSC}.

The Brownian loop soup is an analogous model to fractal percolation and similar problems have been answered affirmatively~\cite{SW-CLE}, as we now recall. Sample a Poissonian random collection of Brownian loops in the unit disk $\mathbb D$ with intensity $c>0$. One can view Brownian loops as analogous to removed boxes in fractal percolation. Two Brownian loops are considered connected if they can be linked by a finite chain of intersecting loops, thereby forming Brownian loop soup clusters. It was shown in~\cite{SW-CLE} that when $c > 1$, there is a single Brownian loop soup cluster whose closure is $\overline{\mathbb D}$, while for $c \in (0,1]$, there are infinitely many Brownian loop coup clusters whose boundaries are disjoint. Moreover, if one considers the set of points not surrounded by any Brownian loop soup cluster---equivalently, the region obtained by removing from $\overline{\mathbb{D}}$ the fillings of all outermost Brownian loop soup clusters---then the resulting set has the law of a CLE$_\kappa$ carpet for $\kappa \in (8/3,4]$, where $\kappa$ is determined by the relation $c = (3 \kappa - 8)(6 - \kappa)/ 2 \kappa$. By~\cite{SW-CLE} and Whyburn's criterion~\cite{Whyburn-carpet}, CLE$_\kappa$ carpet is a topological Sierpi\'nski carpet.

The key properties of Brownian loop soup used in~\cite{SW-CLE} are the domain Markov property and conformal invariance. While the latter is absent in fractal percolation---so the arguments of~\cite{SW-CLE} do not directly apply---their results give support to Problems~\ref{prob:critical-TSC} and~\ref{prob:fractal-strong-disconnect}.

\subsubsection*{Conformal removability and bubble connectivity of SLE$_\kappa$}

As mentioned in Section~\ref{sec:intro}, conformal removability (or non-removability) is essential in the study of SLE$_\kappa$ via conformal welding.
    
\begin{prob}
    Prove Conjecture~\ref{conj:removability}.
\end{prob}

\begin{prob}
    Determine the value of $\kappa_c$ in Conjecture~\ref{conj:removability}.
\end{prob}

One might hope to adapt the argument in Proposition~\ref{prop:super-sierpinski} to prove Conjecture~\ref{conj:removability}. The main missing ingredient is the monotonicity of SLE$_\kappa$ with respect to $\kappa$.
We hope that the mating-of-trees theory~\cite{DMS21-LQG-MRT} may help establish such properties; see~\cite{GP20-connectivity} for its application in proving the connectivity of the adjacency graph of SLE$_\kappa$ for $\kappa \in (4,\kappa_0)$, where $\kappa_0 \approx 5.6158$. In addition, it would be interesting to determine whether SLE$_6$ is conformally removable, as it describes the scaling limit of Bernoulli percolation~\cite{Smirnov-per} and enjoys additional nice properties, for instance, locality. Let us also mention that analogous questions for Brownian motion are also open: Consider a planar Brownian motion with unit time duration. Werner~\cite[Open problem 4]{MP-brownian-motion-book} asks whether the adjacency graph of the complementary connected components is connected. It is also unknown whether the trace of Brownian motion is conformally removable.

Even if Conjecture~\ref{conj:removability} were proved, additional arguments would be needed to understand the critical case $\kappa = \kappa_c$. In light of the results for Brownian loop soup~\cite{SW-CLE}, we conjecture that the scenario~\eqref{conj:case2} in Conjecture~\ref{conj:removability} occurs for $\kappa = \kappa_c$; see also Footnote~\ref{footnote:gk-open}.

\begin{prob}
    Let $\kappa_c$ be as defined in Conjecture~\ref{conj:removability}. Prove that the adjacency graph of the complementary connected components of SLE$_{\kappa_c}$ is disconnected, and that the trace of SLE$_{\kappa_c}$ contains a topological Sierpi\'nski carpet.
\end{prob}

In the regime where the adjacency graph of the complementary connected components of SLE$_\kappa$ is disconnected, we may introduce a stronger notion analogous to Problem~\ref{prob:fractal-strong-disconnect}. We say the adjacency graph is \emph{strongly disconnected} if the closures of different connected components of bubbles are pairwise disjoint.

\begin{prob}\label{prob:sle-strong-disconnect}
    Let $\kappa_c$ be as defined in Conjecture~\ref{conj:removability}. Prove that the adjacency graph of the complementary connected components of SLE$_\kappa$ is strongly disconnected for $\kappa \in (\kappa_c, 8)$.
\end{prob}

If Problem~\ref{prob:sle-strong-disconnect} can be solved, then one can obtain the largest possible topological Sierpi\'nski carpet in the trace of SLE$_\kappa$ by removing from the plane the interiors of fillings of outermost connected components of bubbles. It would be interesting to study its fractal properties, for instance, its Hausdorff dimension.

\subsubsection*{Conformal removability and bubble connectivity of exotic SLE$_{\underline{\kappa}}(\rho)$}

For $\underline{\kappa} \in (0,4)$ and $\rho > -2 - \underline{\kappa}/2$, as with ordinary SLE$_{\underline{\kappa}}$, one can define the SLE$_{\underline{\kappa}}(\rho)$ process via the Loewner evolution, except that we keep track of an additional force point. By convention, the force point is always positioned to the right of the tip. As shown in~\cite{IG1, MS-light-cone, CLEpercolation}, SLE$_{\underline{\kappa}}(\rho)$ is almost surely a continuous curve connecting the two boundary marked points. However, it exhibits several different phases depending on the value of $\rho$, as illustrated in Figure~\ref{fig:rhokappachart} (left). Conformal welding theory for SLE$_{\underline{\kappa}}(\rho)$ was first established in the ordinary regime $\rho>-2$~\cite{Sheffield-zipper, DMS21-LQG-MRT} and has been further extended to the exotic regime $\rho<-2$~\cite{MSW22-simple-CLE-LQG,KM24-SLE-light-cone-LQG}. Let $\eta$ be an SLE$_{\underline{\kappa}}(\rho)$ from 0 to $\infty$ in $\mathbb{H}$ with force point at $0^+$. We will consider whether $\eta$ is conformally removable and whether the adjacency graph of the connected components of $\mathbb{H} \setminus \eta$ is connected.

\begin{figure}
    \centering
    \includegraphics[width=0.96\linewidth]{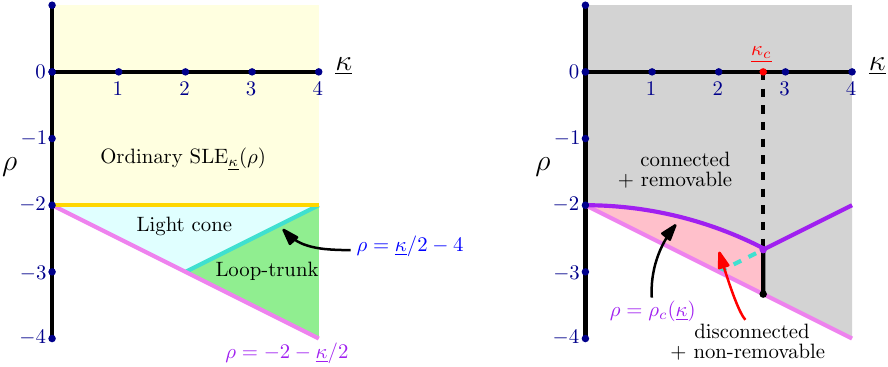}
    \caption{\textbf{Left:} Phase diagram of SLE$_{\underline{\kappa}}(\rho)$ for $\underline{\kappa} \in (0,4)$. The $\rho>-2$ regime corresponds to ordinary SLE$_{\underline{\kappa}}(\rho)$. The light cone and loop-trunk regimes are studied in~\cite{MS-light-cone} and~\cite{CLEpercolation}, respectively. \textbf{Right:} Conjectured phase diagram of SLE$_{\underline{\kappa}}(\rho)$ in terms of bubble connectivity and conformal removability. In the light cone regime, there exists a curve $\underline{\kappa} \mapsto \rho_c(\underline{\kappa})$, above which the adjacency graph of the complementary components of SLE$_{\underline{\kappa}}(\rho)$ is connected, and below which it is disconnected. In the loop-trunk regime, the connectivity property of SLE$_{\underline{\kappa}}(\rho)$ does not depend on $\rho$ and is equivalent to that of SLE$_{16/\underline{\kappa}}$.}
    \label{fig:rhokappachart}
\end{figure}

When $\rho>-2$, $\eta$ is a simple curve and is absolutely continuous with respect to ordinary SLE$_{\underline{\kappa}}$ away from the boundary, which implies that $\eta$ is conformally removable~\cite{DMS21-LQG-MRT} and that the adjacency graph is connected. When $\rho \in (-2 - \underline{\kappa}/2,-2) \cap [\underline{\kappa}/2-4,-2)$, SLE$_{\underline{\kappa}}(\rho)$ is a non-simple curve and admits a construction via GFF light cones~\cite{MS-light-cone}, with opening angle that increases with $\rho$. As a consequence, for fixed $\underline{\kappa}$, the connectivity of the adjacency graph of $\mathbb{H} \setminus \eta$ is decreasing in $\rho$, and there exists a (possibly trivial) phase transition. Moreover, using the argument in Proposition~\ref{prop:super-sierpinski}, we expect that in the disconnected phase, the trace contains a topological Sierpi\'nski carpet. When $\rho$ tends to $-2$ from below, the opening angle tends to 0, and it is likely that the adjacency graph of $\mathbb{H} \setminus \eta$ is connected. When $\rho$ tends to $-2-\underline{\kappa}/2$ from above for $\underline{\kappa} \in (0,2]$, the Hausdorff dimension of $\eta$ tends to 2~\cite{light-cone-dimension}, and we expect that the method in this paper can show that the adjacency graph is disconnected. In light of~\cite{KMS-removability-SLE4, KMS-removability-nonsimple, Ntalampekos-carpet}, we have the following conjecture. 

\begin{conjecture}\label{conj:exotic-1}
    For each $\underline{\kappa} \in (0,4)$, there exists $\rho_c = \rho_c(\kappa) \in (-2-\underline{\kappa}/2,-2) \cap [\underline{\kappa}/2-4,-2) $ such that the following holds. Let $\eta$ be an SLE$_{\underline{\kappa}}(\rho)$ curve from $0$ to $\infty$ on $\mathbb{H}$ with force point always located to the right of the tip.
    \begin{enumerate}
        \item For $\rho \in (\rho_c, -2)$, almost surely, the range of $\eta$ is conformally removable and the adjacency graph of $\mathbb{H} \setminus \eta$ is connected.

        \item For $\rho \in (-2-\underline{\kappa}/2,\rho_c) \cap [\underline{\kappa}/2-4,\rho_c)$, the range of $\eta$ is conformally non-removable, the adjacency graph of $\mathbb{H} \setminus \eta$ is disconnected, and the range of $\eta$ contains a topological Sierpi\'nski carpet.
    \end{enumerate}
\end{conjecture}

For $\underline{\kappa} \in (2,4)$ and $\rho = \underline{\kappa}/2 - 4$, the range of SLE$_{\underline{\kappa}}(\rho)$ has the same law as the range of SLE$_{\kappa}(\kappa/2-4)$ with force point at $0^-$, where $\kappa=16/\underline{\kappa}$~\cite{MS-light-cone}. Note that SLE$_{\kappa}(\kappa/2-4)$ is absolutely continuous with respect to SLE$_{\kappa}$ away from the boundary. By~\cite{GP20-connectivity}, for $\kappa \in (4, \kappa_0)$, where $\kappa_0 \approx 5.6158$, the adjacency graph of SLE$_\kappa$ is connected, which implies that $\rho_c(\underline{\kappa}) = \underline{\kappa}/2-4$ for $\underline{\kappa} \in (16/\kappa_0, 4)$. In light of this, we have the following conjecture for the behavior of $\rho_c$ and its relation to conformal removability and bubble connectivity of SLE$_\kappa$. In particular, taken together with Conjecture~\ref{conj:exotic-1}, it would imply Conjecture~\ref{conj:removability} with $\kappa_c = 16/\underline{\kappa_c}$.

\begin{conjecture}\label{conj:exotic-2}
    The function $\underline{\kappa} \mapsto \rho_c(\underline{\kappa})$ is continuous for $\underline{\kappa} \in (0,4)$. There exists $\underline{\kappa_c} \in (2,4)$ such that $\rho_c(\underline{\kappa}) > \underline{\kappa}/2 - 4$ for $\underline{\kappa} \in (2, \underline{\kappa_c})$ and $\rho_c(\underline{\kappa}) = \underline{\kappa}/2 - 4$ for $\underline{\kappa} \in (\underline{\kappa_c}, 4)$. Moreover, for $\kappa \in (4, 16/ \underline{\kappa_c})$, the SLE$_{\kappa}$ curve is almost surely conformally removable and the adjacency graph of its complementary is connected.
\end{conjecture}

Finally, we investigate the loop-trunk regime studied in~\cite{CLEpercolation} where $\underline{\kappa} \in (2,4)$ and $\rho \in (-\underline{\kappa}/2-2, \underline{\kappa}/2-4]$. As shown in Section 10.1.2 of~\cite{CLEpercolation}, the trace of SLE$_{\underline{\kappa}}(\rho)$ can be decomposed into a trunk and loops: the trunk has the same law as a SLE$_{\kappa}(-\kappa(\rho+2)/4-2)$ with force point at $0^-$, and loops that form on the left side of the trunk have the law of BCLE$_{\underline{\kappa}}(\rho + \underline{\kappa}/2)$ in each left complementary connected components. By construction, the adjacency graph of the complementary connected components of a BCLE$_{\underline{\kappa}}(\rho + \underline{\kappa}/2)$ is almost surely connected. Moreover, it follows from the argument in Section 3.5 of~\cite{DMS21-LQG-MRT} that BCLE$_{\underline{\kappa}}(\rho + \underline{\kappa}/2)$ is conformally removable. Therefore, we deduce the following proposition.

\begin{proposition}
    For $\underline{\kappa} \in (2,4)$ and $\rho \in (-2-\underline{\kappa}/2, \underline{\kappa}/2-4]$, SLE$_{\underline{\kappa}}(\rho)$ is conformally removable (resp.\ the adjacency graph is connected) if and only if SLE$_{16/\underline{\kappa}}$ is conformally removable (resp.\ the adjacency graph is connected).
\end{proposition}

In Figure~\ref{fig:rhokappachart} (right), we summarize the conjectured phase of SLE$_{\underline{\kappa}}(\rho)$. This conjectured phase diagram suggests that conformal removability is equivalent to the bubble connectivity property in most of the parameter space, except on critical lines.

\subsubsection*{Connectivity of other random fractals}

In a seminal work~\cite{Peres-intersection}, Peres showed that planar Brownian motion is intersection-equivalent to the final retained set of a certain fractal percolation process in which a $2^{-k}$-box is retained with probability $1/k$. This property allows one to show that the intersection of two independent planar Brownian motions is intersection-equivalent to the final retained set of a similar fractal percolation process, which in particular contains non-singleton connected components. Based on this, Peres~\cite{peres-brownian-intersection} posed the following question: For two independent planar Brownian motions run for unit time duration, does their intersection almost surely contain non-singleton connected components?

Motivated by this, we ask an analogous question for SLE$_\kappa$. 

\begin{prob}\label{prob:intersection-sle}
    For $\kappa \in (4,8)$ and an integer $n \geq 2$, let $\eta_1,\eta_2,\ldots, \eta_n$ be $n$ independent SLE$_\kappa$ curves from 0 to $\infty$ in $\mathbb{H}$. Does the intersection $\cap_{i=1}^n \eta_i$ contain non-singleton connected components? Here we view $\eta_i$ as the trace of SLE$_\kappa$, which is a closed subset of $\overline{\mathbb{H}}$.
\end{prob}

We note that when $\kappa \leq 4$, each curve is simple so their intersection is totally disconnected. When $\kappa \geq 8$, the SLE$_\kappa$ curves are space-filling, so their intersection has interior points. For the regime $\kappa \in (4,8)$, we first compute the Hausdorff dimension of $\cap_{i=1}^n \eta_i$. Since the Hausdorff dimension of a single SLE$_\kappa$ curve is $1 + \kappa/8$~\cite{Rohde-Schramm-basic, Beffara-dimension}, the intersection $\cap_{i=1}^n \eta_i$ is expected to have dimension $\max \{0, 2 - n(1 - \kappa/8) \}$. If the value is smaller than 1, the intersection is totally disconnected by~\cite[Proposition 3.5]{falconer-fractal-geometry}. This suggests that for $n \geq 3$ and $\kappa$ close to 4, the intersection of $n$ independent SLE$_\kappa$ curves should be totally disconnected. When $n = 2$ and $\kappa$ is close is 4, even though the Hausdorff dimension calculation is not sufficient, we still expect the intersection to be totally disconnected, since SLE$_\kappa$ converges to the simple curve SLE$_4$ as $\kappa \to 4$. Finally, let us comment that the fractional percolation strategy in Section~\ref{sec:main-proof} can be adapted to show that for any fixed $n \geq 2$ and $\kappa$ sufficiently close to 8, the intersection $\cap_{i=1}^n \eta_i$ contains non-singleton connected components (possibly even a topological Sierpi\'nski carpet). This reveals another phase transition in the behavior of SLE$_\kappa$ within the regime $\kappa \in (4,8)$.

The case $\kappa = 6$ may be of particular interest, as it describes the scaling limit of Bernoulli percolation~\cite{Smirnov-per}. The Hausdorff dimension calculation above suggests that the intersection of five or more independent SLE$_6$ curves does not contain non-singleton connected components.

\begin{prob}
    Does the intersection of two, three, or four independent SLE$_6$ curves contain non-singleton connected components?
\end{prob}

\subsubsection*{Intersection-equivalence for SLE}

Motivated by~\cite{Peres-intersection}, it is natural to ask whether SLE$_\kappa$ is intersection-equivalent to certain fractal percolation processes. This question is related to Problem 7.2 in~\cite{Schramm-ICM} which seeks a criterion for a closed set $A$ such that $\mathbb{P}[\eta \cap A \neq \emptyset] >0$, where $\eta$ is the trace of SLE.

\begin{prob}
    Does there exist a fractal percolation process intersection-equivalent to SLE$_\kappa$?
\end{prob}

\bibliographystyle{alpha}
\bibliography{ref}

\end{document}